\documentclass[reqno,11pt]{amsart}
\usepackage{amsmath,enumerate,amsfonts,amssymb,amsthm,mathrsfs}
\usepackage{graphicx}
\usepackage{bm}
\allowdisplaybreaks[4]
\usepackage{times}
\usepackage{comment}
\usepackage{float}
\usepackage[numbers,sort&compress]{natbib}
\usepackage{hyperref}
\usepackage{framed}
\usepackage[normalem]{ulem}
\usepackage[margin=1.25in]{geometry}
\numberwithin{equation}{section}
\newtheorem{thm}{Theorem}[section]
\newtheorem{lem}[thm]{Lemma}
\newtheorem{prop}[thm]{Proposition}
\newtheorem{defn}[thm]{Definition}

\newtheorem{cor}[thm]{Corollary}
\newtheorem{rem}[thm]{Remark}
\theoremstyle{plain}

\renewcommand{\d}{\mathrm{d}}
\renewcommand{\H}{\mathbb{H}}
\newcommand{\B}{\mathcal{B}}
\newcommand{\R}{\mathbb{R}}

\newcommand{\Cn}{\C^n}
\newcommand{\Rn}{\mathbb{R}^n}

\renewcommand{\S}{\mathbb{S}}
\newcommand{\C}{\mathbb{C}}
\newcommand{\lam}{\lambda}
\newcommand{\pa}{\partial}

\newcommand{\var}{\varepsilon}
\newcommand{\be}{\begin{equation}}
\newcommand{\ee}{\end{equation}}
\newcommand{\Lam}{\Lambda}
\newcommand{\si}{\sigma}
\newcommand{\al}{\alpha}
\newcommand{\om}{\omega}
\newcommand{\Om}{\Omega}
\newcommand{\Hn}{\mathbb{H}^{n}}

\def\XXint#1#2#3{{\setbox0=\hbox{$#1{#2#3}{\int}$}\vcenter{\hbox{$#2#3$}}\kern-.5\wd0}}

\author{Jiechen Qiang}
\address{School of Mathematical Sciences, Laboratory of Mathematics and Complex Systems, MOE,
 Beijing Normal University, Beijing 100875, China}
\email{jcqiang@mail.bnu.edu.cn}

\author{Zhongwei Tang}
\address{School of Mathematical Sciences, Laboratory of Mathematics and Complex Systems, MOE,
 Beijing Normal University, Beijing 100875, China}
\email{tangzw@bnu.edu.cn}
\thanks{ Z. Tang is supported by NSFC grants 12471099 and 12271223.}

\author{Heming Wang}
\address{State Key Laboratory of Mathematical Sciences, Academy of Mathematics and Systems Science,
 Chinese Academy of Sciences, Beijing 100190, China}
\email{hmwang@amss.ac.cn}
\thanks{H. Wang is supported by China Postdoctoral Science Foundation 2026M793418.}

\date{}
\title[CR Nirenberg problem, Part I]{The CR Nirenberg Problem, Part I: \\  Blow-up Analysis and Compactness}
\hypersetup{
 colorlinks,citecolor=blue,filecolor=blue,linkcolor=blue,urlcolor=black}

\begin{document}
\begin{abstract}
We establish a priori estimates and compactness results for positive solutions of the CR Nirenberg problem  under natural weighted flatness conditions. The proofs combine refined blow-up analysis, Pohozaev identities, and the complete classification of positive entire solutions to the critical CR Yamabe equation recently obtained  by Li--Liu--Ma \cite{LLM2026} and Li \cite{L2026}. Existence and degree results are established in our subsequent paper \cite{QTWPartII}.
\end{abstract}
\maketitle
{\small\tableofcontents}

\section{Introduction}
Curvature is a fundamental object in geometric analysis, and its transformation under conformal changes gives rise to many central problems in conformal geometry. In particular, a celebrated problem posed by Nirenberg in the 1960s asks whether a given function $K$ on the standard sphere $(\S^N,g_0)$, $N\geq2$, can be realized as the scalar curvature of a metric conformal to $g_0$ (or as the Gauss curvature when $N=2$). This is widely known as the Nirenberg problem and is also called the prescribed scalar curvature problem on $\S^N$. If we denote $g=e^{2v}g_0$ in the  case $N=2$
and $g=v^{4/(N-2)}g_0$ in the $N\geq 3$ dimensional case, this problem amounts to finding a solution $v$ of the equation
\[
-\Delta_{g_{0}}v+1=Ke^{2v} \quad\text{ on }\, \S ^2,
\]
and a positive solution $v$ of the equation
\[
-\Delta_{g_{0}}v+c(N)R_{g_0}v=c(N)Kv^{\frac{N+2}{N-2}} \quad\text{ on }\, \S^N \quad  \text{ for } \, N\geq 3,
\]
where  $\Delta_{g_{0}}$ is the Laplace--Beltrami operator on $(\S^N, g_{0})$, $c(N)=(N-2)/(4(N-1))$  and  $R_{g_0}=N(N-1)$ is the scalar curvature associated to $g_0$. The Nirenberg  problem has been studied extensively and it would be impossible to mention here all works in this area. The developments most closely related to this paper are the fine blow-up analysis of solutions and approximate solutions, and gluing constructions; see, e.g., \cite{Li1993,Li1995,Li1996,CGY1993,CY1991,WY2010,LWX2018,BC1988,BC1991,Li93d,SZ1996,MM2020} and the references therein.

Analogous prescribed curvature problems arise on CR manifolds, with the conformal sub-Laplacian in place of the conformal Laplacian.  The geometry of CR manifolds, which generalizes the geometry of real hypersurfaces in complex manifolds, was developed in fundamental works of Chern-Moser \cite{CM1974}, Fefferman \cite{F1976}, Jacobowitz \cite{J1990}, Jerison-Lee \cite{JL1984,JL1987,JL1988,JL1989}, Tanaka \cite{T1975}, Webster \cite{W1978}, among many others. On strictly pseudoconvex CR manifolds, the resulting geometry has close parallels with Riemannian conformal geometry. In particular, the CR Yamabe problem, which concerns constant Webster scalar curvature, was studied by Jerison--Lee \cite{JL1987,JL1989,JL1988} and subsequently by Gamara and Gamara--Yacoub \cite{G2001,GY2001}; see these works and the references therein.  For more recent and further studies, see \cite{DRY2022,CAY2010,CAY2013,R2013,R2012,GAG2015,SG2011} and related references.

As a natural analogue of the Nirenberg problem for the CR geometry, one can consider the prescribed Webster (pseudo-Hermitian) scalar curvature problem on the standard   CR sphere  which can be formulated as follows. 

Let $n\geq 1$ and let  $(\S^{2n+1},\theta_{0})$ be the unit   CR sphere in $\C^{n+1}$, where $\theta_{0}$ is the standard contact form. Given a  function $\bar{R}$ on $\S ^{2n+1}$, it is natural to ask: Does there exist a contact form $\theta$ conformally related to $\theta_0$ in the sense that $\theta=v^{2 / n} \theta_0$ for some  function $v>0$ such that $\bar{R}$ is the Webster scalar curvature associated with the contact form $\theta$? As in the Riemannian case, for the Webster scalar curvature associated with $\theta$, we have the transformation law
\[
\operatorname{Scal}_{\theta}=v^{-\frac{n+2}{n}}\Big(-\frac{2(n+1)}{n} \Delta_{\theta_0} v+\operatorname{Scal}_{\theta_0} v\Big).
\]
Here  $\Delta_{\theta_0}$ is the sub-Laplacian with respect to the contact form $\theta_0$ and $\operatorname{Scal}_{\theta_0}=n(n+1)/2$ is the Webster scalar curvature  associated with the contact form $\theta_0$. Clearly, the problem of solving $\operatorname{Scal}_{\theta}=\bar{R}$ is equivalent to finding positive solutions $v$ to the following PDE:
\be\label{maineq0}
L_{\theta_{0}}v:=-\Delta_{\theta_0}v +\frac{n^2}{4} v=\bar{c}(n) \bar{R} v^{1+\frac{2}{n}},\quad v>0 \quad \text { on }\, \S^{2n+1}.
\ee
Here $\bar c(n)=\frac{n}{2(n+1)}$ and $L_{\theta_0}$ is the \emph{conformal sub-Laplacian}. More generally, set $L_\theta=-\Delta_\theta+\bar c(n)\operatorname{Scal}_\theta$. For $\theta=v^{2/n}\theta_0$, its conformal covariance is given by $L_\theta(\phi)=v^{-\frac{n+2}{n}}L_{\theta_0}(v\phi)$ for every $\phi\in C^\infty(\S^{2n+1})$.
Let $\phi'$ be an extension of $\phi$ to $\C^{n+1}$. The sub-Laplacian $\Delta_{\theta_0} \phi$, defined on $\S^{2n+1}$ in the coordinate system $\zeta = (\zeta_1, \ldots, \zeta_{n+1}) \in \S^{2n+1}$, is given by the following expression:
\be\label{Delta}
\Delta_{\theta_0}\phi=\sum_{j=1}^{n+1} \frac{\pa^2 \phi'}{\pa \zeta_j \pa \bar{\zeta}_j}-\sum_{j, k=1}^{n+1} \zeta_j \bar{\zeta}_k \frac{\pa^2 \phi'}{\pa \zeta_j \pa \bar{\zeta}_k}- \frac{n}{2}\sum_{k=1}^{n+1}\Big(\zeta_k \frac{\pa \phi'}{\pa \zeta_k}+\bar{\zeta}_k \frac{\pa \phi'}{\pa \bar{\zeta}_k}\Big),
\ee
where $\bar\zeta_j$ denotes the complex conjugate of $\zeta_j$; see, e.g., \cite{BFM2013,JL1987}.

The solvability of \eqref{maineq0} is subject to restrictions on the prescribed curvature. Multiplying \eqref{maineq0} by $v$ and applying integration by parts, we see that a simple necessary condition is  $\max_{\S^{2n+1}} \bar{R} > 0$. More subtle obstructions arise from the conformal symmetry of the CR sphere. In particular, the Kazdan--Warner identity \cite[Theorem~B]{Cheng} states that, for every infinitesimal CR automorphism $X$ of $\S^{2n+1}$,
\be\label{KW1}
\int_{\S^{2 n+1}}X(\bar{R})  v^{2+\frac{2}{n}} \, \d V_{\theta_0}=0
\ee
for any positive solution $v$ of  \eqref{maineq0},  where  $\d V_{\theta_0}$  denotes the volume form associated with $\theta_0$. Identity \eqref{KW1} already rules out certain positive prescribed curvatures. For example, \eqref{maineq0} admits no positive solution when $\bar R(\zeta)=c+\operatorname{Re}\zeta_1$, where $c>1$ so that $\bar R$ is everywhere positive. Indeed, the infinitesimal CR automorphism $X$ with complex coordinate components $X_j(\zeta)=\delta_{1j}-\zeta_1\zeta_j$
satisfies
\[
X(\bar R)=1-\operatorname{Re}(\zeta_1^2)\geq 0,
\]
with strict inequality on a nonempty open subset of $\S^{2n+1}$. Since $v>0$, the integral in \eqref{KW1} would then be strictly positive, contradicting the Kazdan--Warner identity \eqref{KW1}.

 Another  difficulty in studying \eqref{maineq0} is the lack of compactness due to the presence of the Sobolev critical exponent $2+\frac{2}{n}$. A typical phenomenon encountered here is \emph{bubbling blow up}. Bubbles are solutions of \eqref{maineq0} with $\bar{R}=1$ and describe the local profiles of blowing-up sequences. Under the Cayley transform, the finite-volume condition becomes the critical integrability condition on the Heisenberg group; Jerison--Lee \cite{JL1988} classified the corresponding finite-energy entire solutions. From the variational point of view, such bubbles generate noncompact Palais--Smale sequences for the Euler--Lagrange functional associated with \eqref{maineq0}. Two main approaches have been used to understand the blow up phenomenon:  subcritical approximations (see, e.g., \cite{PR2003,TWZ}), or the construction of pseudo-gradient flows (see, e.g., \cite{GH2017,GHM2020,SG2011}).

The prescribed Webster scalar curvature problem on the standard CR sphere is referred to as the \emph{CR Nirenberg problem}. Further geometric and analytic background will be recalled in Section~\ref{sec:2.1}.  Throughout the paper, we assume $n\geq1$ and use the normalized prescribed Webster scalar curvature $R=2 \bar{R} /(n(n+1))$, so that $R\equiv1$ corresponds to the standard contact form. With this normalization, equation~\eqref{maineq0} becomes
 \be\label{maineq}
 L_{\theta_0} v=\frac{n^2}{4} R v^{1+2 / n}, \quad v>0 \quad \text { on }\, \S^{2n+1}.
\ee
Problem \eqref{maineq0} has been studied extensively over the past decades. Existing approaches include the Webster scalar curvature flow \cite{Ho20151,Ho20152}, finite-dimensional reduction and variational methods \cite{CPY2013,MU2002,FU2004}, and the theory of critical points at infinity \cite{GH2017,GHM2020,SG2011,RG2012,RG2012-2,R2012,R2013}. For further existence results, see, e.g., \cite{Ho2016,Ho2020} and the references therein.

In this paper, we establish refined blow-up estimates and
compactness results for positive solutions of the CR Nirenberg
problem on $\S^{2n+1}$ under weighted flatness conditions.
The resulting a priori estimates provide the analytic foundation
for the existence and topological degree results in our
subsequent paper \cite{QTWPartII}.

To state these results, we first introduce the Heisenberg coordinates and the assumptions on the prescribed curvature $R$ used
below. The conformal sub-Laplacian has a concrete realization on the Heisenberg group $\Hn$, a flat CR manifold.  The Heisenberg group $\Hn$ is the Lie group $\Cn\times\R$ with group law
\[
(z,t)\circ(\hat z,\hat t)
=(z+\hat z,t+\hat t+2\operatorname{Im}(z\cdot\overline{\hat z})),\quad \forall\, (z,t), (\hat z,\hat t)\in \Cn\times\R.
\]
Here $z\cdot\overline{\hat z}=\sum_{j=1}^n z_j\overline{\hat z_j}$. We write $\xi=(z,t)=(x,y,t)$, where $z=x+\mathrm i y$ and $x,y\in\Rn$. The identity element is $0=(0,0)$ and $\xi^{-1}=(-z,-t)$. For a positive integer $k$, $\xi^k$ denotes the $k$-fold group product. We set
\[
Q=2n+2,\quad Q^*=\frac{2Q}{Q-2},\quad \rho(\xi)=|\xi|=(|z|^4+t^2)^{1/4},\quad
 d(\xi,\xi_0)=|\xi_0^{-1}\circ\xi|.
\]
The left translations and Heisenberg dilations are given by
\[
\tau_{\xi_0}(\xi)=\xi_0\circ\xi,
\quad \delta_\lam (z,t)=(\lam  z,\lam ^2t),\quad \lam>0.
\]
For $r>0$, we write $B_r(\xi_0)=\{\xi\in\Hn:d(\xi,\xi_0)<r\}$ and $B_r=B_r(0)$.
The standard left-invariant vector fields are
\begin{equation}\label{vectorfields}
X_j=\frac\pa {\pa  x_j}+2y_j\frac\pa {\pa  t},
\quad
Y_j=\frac\pa {\pa  y_j}-2x_j\frac\pa {\pa  t},
\quad T=\frac\pa {\pa  t},
\end{equation}
and we use the normalization
\[
\nabla_{\H}u=\frac12(X_1u,\ldots,X_nu,Y_1u,\ldots,Y_nu),
\quad
\Delta_{\H}=\frac14\sum_{j=1}^n(X_j^2+Y_j^2).
\]
We  also use the right-invariant horizontal vector fields
\begin{equation}\label{right-vectorfields}
 \overline X_j=\frac\pa {\pa  x_j}-2y_j\frac\pa {\pa  t},\quad
 \overline Y_j=\frac\pa {\pa  y_j}+2x_j\frac\pa {\pa  t},
 \quad 1\le j\le n.
\end{equation}
Writing $Z=(X_1,\ldots,X_n,Y_1,\ldots,Y_n)$ and $\overline Z=(\overline X_1,\ldots,\overline X_n,\overline Y_1,\ldots,\overline Y_n)$, we set
\[
 \widetilde\nabla f=\overline Zf,\quad
 \overline\nabla f=(\overline Zf,Tf).
\]
Thus $\widetilde\nabla f$ has $2n$ components, whereas
$\overline\nabla f$ has $2n+1$ components; in contrast with
$\nabla_{\H}$, neither notation includes the factor $1/2$.
Unless otherwise specified, the norms of these derivative arrays are
the corresponding Euclidean norms. In computations on $\Hn$, the symbol $\nabla$ denotes the Euclidean gradient in coordinates. For a vector field $D$, $(Df)(\eta\circ\xi)$ denotes the derivative evaluated at $\eta\circ\xi$, whereas $D_\xi[f(\eta\circ\xi)]$ differentiates the composition. We also write $Z^2u=(Z_aZ_bu)_{a,b=1}^{2n}$ and $ZTu=(Z_aTu)_{a=1}^{2n}$; thus $ZR=2\nabla_\H R$.

On a pseudo-Hermitian manifold, $\nabla_\theta$ denotes the horizontal
gradient and $T_\theta$ the Reeb field. For $R\in C^1(\S^{2n+1})$, we write
\[
\operatorname{Crit}(R)
=\{P\in \S^{2n+1}:\d R(P)=0\}
=\{P\in \S^{2n+1}:\nabla_{\theta_0}R(P)=0,\ T_{\theta_0}R(P)=0\}.
\]
The coordinate field $T=\pa_t$ is distinct from $T_\theta$.
For $\zeta,\eta\in\C^{n+1}$, let
$\langle\zeta,\eta\rangle=\sum_{j=1}^{n+1}\zeta_j\overline{\eta_j}$.
On the sphere we use the contact distance
\begin{equation}\label{eq:contact-distance}
d_{\theta_0}(P,P')=\sqrt{2|1-\langle P,P'\rangle|}
\end{equation}
and the balls $\B(P,r)=\{P'\in \S^{2n+1}:d_{\theta_0}(P,P')<r\}$. 

On $\Hn$, we use Haar measure $\d\xi=\d x\,\d y\,\d t$, write $\d z=\d x\,\d y$, and use Euclidean surface measure $\d S$ on boundaries.
On $\S^{2n+1}$, integrals and $L^q$ norms use the Euclidean surface measure $\d\si $ unless otherwise stated. The contact volume form is
\[
\d V_{\theta_0}=\theta_0\wedge(\d\theta_0)^n
=2^{2n+1}n!\,\d\si .
\]

The CR sphere minus one point and the Heisenberg group are CR equivalent. The Cayley transform from $\S^{2n+1}\backslash\{\mathcal{S}\}$ to $\Hn$ is the inverse of
\be\label{Cayley}
	F: \Hn \to \S^{2n+1}\backslash\{\mathcal{S}\}, \quad (z,t) \mapsto\Big(\frac{2 z}{1+|z|^2-\mathrm{i} t}, \frac{1-|z|^2+\mathrm{i} t}{1+|z|^2-\mathrm{i} t}\Big),
\ee
where $\mathcal{S}=(0,\ldots,0,-1)$ is the south pole of $\S^{2n+1}$. 

For $P\in\S^{2n+1}$,  choose a unitary matrix $U_P$ with $U_P(0,\ldots,0,1)=P$ and set
\begin{equation}\label{eq:normalized-cayley}
 F_P=U_P\circ F\circ\delta_{1/2},\quad F_P(0)=P.
\end{equation}
For $\lam \ge1$, use \eqref{eq:normalized-cayley} to define the CR conformal automorphism
\begin{equation}\label{def:conformal-map}
\varphi_{P,\lam }:=F_P\circ\delta_{\lam ^{-1}}\circ F_P^{-1}.
\end{equation}

The choices of $U_P$ range over the compact unitary group; a global smooth choice is not required. On the sphere, the hypotheses on $R$ are understood in these centered Cayley charts on fixed small coordinate balls. Their contact distances are uniformly comparable with $d$, and their conformal factors equal $1$ at the origin. The usual pseudo-Hermitian normal coordinates on a general CR manifold and these exact sphere charts are not identified without a coordinate change.

Under this correspondence, any positive solution of \eqref{maineq} gives rise to a positive entire solution of
\be\label{maineq1}
-\Delta_{\H}u
=\frac{n^2}{4}R(\xi)\,u^{\frac{Q+2}{Q-2}},
\quad
u>0
\quad\text{ in }\,\Hn,
\ee
where $Q=2n+2$, $\Delta_{\H}$ is the sub-Laplacian on $\Hn$, and, by a slight abuse of notation, $R$ denotes the pullback of the prescribed curvature under the inverse Cayley transform \eqref{Cayley}. Conversely, a positive entire solution of \eqref{maineq1} corresponds to a positive solution on the sphere provided that its inverse Cayley transform extends as a positive $C^2$ function across the omitted point. We refer to Section~\ref{sec:heisenberg-prelim} for the precise transformation formulas.

For later use, set
\begin{equation}\label{standard-bubble}
\Lam (z,t):=((1+|z|^2)^2+t^2)^{(2-Q)/4},
\quad U_{0,1}:=2^n\Lam .
\end{equation}
The conformal factors of $F$ and $F_P$ satisfy
\[
H_F(\xi)=2^n\Lam (\xi)=U_{0,1}(\xi),\quad
H_P(\xi)=\Lam (\delta_{1/2}\xi),
\]
so that $H_F(0)=2^n$ and $H_P(0)=1$. The function $U_{0,1}$ is the standard bubble for \eqref{maineq1} with $R\equiv1$; its translates and dilates are recalled in Section~\ref{sec:heisenberg-prelim}.

Before stating the main results, we specify the regularity and weighted flatness conditions on $R$.
For $R$, we distinguish ordinary regularity from regularity adapted to the contact structure, which is essential in CR dimension one.
For an open set $U\subset\Hn$ and $\al \in (0,1)$, $\Gamma_\al (U)$ consists of bounded continuous functions with finite norm
\[
\|f\|_{\Gamma_\al (U)}
=\|f\|_{L^\infty(U)}+
\sup_{\substack{\xi,\eta\in U\\\xi\ne\eta}}
\frac{|f(\xi)-f(\eta)|}{d(\xi,\eta)^\al }.
\]
The higher-order spaces are recalled in Section~\ref{function space}.

\begin{defn}\label{def:intrinsic-coefficients}
Let $\al \in (0,1)$ and let $U\subset\Hn$ be relatively compact. We write $R\in\mathcal A_\al (U)$ if $R\in C^1(U)$, every ordered derivative $Z_aZ_bR$ is represented by a bounded measurable function, and $TR\in\Gamma_\al (U)$. We also require finiteness of the norm
\begin{equation}\label{eq:intrinsic-coefficient-norm}
 \|R\|_{\mathcal A_\al (U)}
 =\|R\|_{L^\infty(U)}+\|ZR\|_{L^\infty(U)}
  +\sum_{a,b=1}^{2n}\|Z_aZ_bR\|_{L^\infty(U)}
  +\|TR\|_{\Gamma_\al (U)}.
\end{equation}
On the sphere, this condition and norm are understood uniformly on fixed smaller balls in the centered charts \eqref{eq:normalized-cayley}. The local version is denoted by $\mathcal A_{\al ,loc}$. All distributional derivatives are taken along the indicated fields.
\end{defn}

Motivated by Li's flatness condition \cite[Definition 0.4]{Li1995}, we use the Heisenberg weights one for horizontal derivatives and two for $T$.
\begin{defn}\label{def:weighted-flatness}
Let $\beta\ge2$, let $\Om_i\subset\Hn$ be open, and let
$L_i=L(\beta,i)>0$. Suppose that
$R_i\in\mathcal A_{\al,loc}(\Om_i)$ for a fixed $\al\in(0,1)$.
For $\beta=2$, we say that $\{R_i\}$ satisfies $(*)'_2$ if
there exists a constant $C_*\ge1$, independent of $i$, such that
\begin{equation}\label{eq:lf-two}
 |\nabla_\H^2R_i|+|TR_i|\le C_*L_i
 \quad\text{a.e. in }\Om_i.
\end{equation}
For $\beta>2$, we say that $\{R_i\}$ satisfies $(*)'_\beta$ if
there exist a constant $C_*\ge1$, independent of $i$, and
functions $r_i:\Om_i\to[0,\infty)$ that are locally Lipschitz
with respect to $d$, with Lipschitz constants bounded by $C_*$,
such that
\begin{equation}\label{eq:wf-first}
 C_*^{-1}L_i r_i^{\beta-1}
 \le |\nabla_\H R_i|+r_i|TR_i|
 \le C_*L_i r_i^{\beta-1}
 \quad\text{for every }\xi\in\Om_i,
\end{equation}
and
\begin{equation}\label{eq:wf-higher}
 \begin{aligned}
 |\nabla_\H^2R_i|+|TR_i|&\le C_*L_i r_i^{\beta-2}
 &&\text{a.e. in }\Om_i,\\
 |\nabla_\H(TR_i)|&\le C_*L_i r_i^{\beta-3}
 &&\text{a.e. in }\{\xi\in\Om_i:r_i(\xi)>0\}.
 \end{aligned}
\end{equation}
Here $\nabla_\H^2u=\frac14(Z_aZ_bu)_{a,b=1}^{2n}$ is the
array of all ordered second horizontal derivatives, consistently
with $\nabla_\H=Z/2$. The second horizontal derivatives are
understood weakly on $\Om_i$; for $\beta>2$, the mixed derivatives
$Z_a(TR_i)$ are understood weakly on
$\{\xi\in\Om_i:r_i(\xi)>0\}$. These weak derivatives are
represented by locally integrable functions on the indicated
open sets. Here and below, a.e. refers to Haar measure $\d\xi$.
\end{defn}

Since $R_i\in C^1(\Om_i)$ and $r_i$ is continuous,
\eqref{eq:wf-first} is a pointwise condition. For $\beta>2$,
the bound $|TR_i|\le C_*L_i r_i^{\beta-2}$ in
\eqref{eq:wf-higher} also holds everywhere by continuity.
Consequently,
\[
 \{\xi\in\Om_i:r_i(\xi)=0\}
 =\{\xi\in\Om_i:\nabla_\H R_i(\xi)=0,\ TR_i(\xi)=0\}.
\]
When $2<\beta<3$, the negative power $r_i^{\beta-3}$ is used
only where $r_i>0$. No bound for $T^2R_i$ is imposed.


On the sphere, $(*)'_\beta$ is understood in the centered Cayley charts of \eqref{eq:normalized-cayley}, restricted to a common small coordinate radius. For a sequence, the constants in Definition~\ref{def:weighted-flatness} are required to be uniform in
both the sequence index and the choice of chart; the auxiliary functions may depend on the chart. Changes of a bounded-scale chart are made before the dilations in Lemma~\ref{lem:wf-scaling}. Uniform coordinate invariance for sequences with unbounded auxiliary functions is not included in these conventions.

We first establish uniform a priori bounds for positive solutions under suitable assumptions on the homogeneous expansions at critical points.
\begin{thm}\label{thm:homogeneous-compactness}
Let $\al \in (0,1)$ and let $R \in \mathcal{A}_\al (\S^{2 n+1})$ be positive. In every fixed contact chart assume that the weak derivatives $Z_a(T R)$ are locally essentially bounded away from the full critical set $\{\nabla_{\theta_0} R=T_{\theta_0} R=0\}$.  For each critical point $\zeta_0$ of $R$, let $F_{\zeta_0}$ be the centered Cayley chart in \eqref{eq:normalized-cayley}, with $F_{\zeta_0}(0)=\zeta_0$, and set $R_{\zeta_0}(\xi):=R(F_{\zeta_0}(\xi))$. Assume that there exist a small neighborhood $\mathscr{O}_{\zeta_0}$ of $0$ in $\H^n$,  constants $L>0$, and $\beta=\beta(\zeta_0)\in(Q-2,Q)$ such that  $R_{\zeta_0}$ satisfies condition $(*)'_\beta$ in $\mathscr{O}_{\zeta_0}$ for $L$, and
\[
    R_{\zeta_0}(\xi)=R(\zeta_0)+Q_{\zeta_0}^{(\beta)}(\xi)+W_{\zeta_0}(\xi)\quad \text{ in }\,\mathscr{O}_{\zeta_0},
\]
    where $Q_{\zeta_0}^{(\beta)}\in C^1_{loc }(\H^{n}\backslash \{0\})$ satisfies
    $Q_{\zeta_0}^{(\beta)}(\delta_{\lam}\xi)=\lam ^{\beta}Q_{\zeta_0}^{(\beta)}(\xi)$, $\forall\, \lam  >0$, $\xi\in \Hn $,
  \be\label{0.6}
	c_{\mathrm{nd}}|\xi|^{\beta-1}\leq |\widetilde{\nabla} Q_{\zeta_0}^{(\beta)}(\xi)|+|\xi||TQ_{\zeta_0}^{(\beta)}(\xi)|
\ee
for some constant $c_{\mathrm{nd}}>0$, which may depend on $\zeta_0$,
    and
\be \label{0.8in jde}
	\left( 
	\begin{array}{l}
	\displaystyle\int_{\Hn } \overline{\nabla} Q_{\zeta_0}^{(\beta)}(\eta\circ\xi)U_{0,1}^{Q^* } \, \d  \xi\\[2.5mm]
    \displaystyle\int_{\Hn } Q_{\zeta_0}^{(\beta)}(\eta\circ\xi)U_{0,1}^{Q^* } \, \d  \xi 
	\end{array}  \right) \neq 0, \quad \forall\,  \eta\in \Hn ,
	\ee 
    and $W_{\zeta_0}(\xi)\in C^{1}(\mathscr{O}_{\zeta_0})$ satisfies $\lim _{\xi \to  0}\big(|W_{\zeta_0}(\xi)||\xi|^{-\beta}+|\widetilde{\nabla}W_{\zeta_0}(\xi)||\xi|^{1-\beta}+|TW_{\zeta_0}(\xi)||\xi|^{2-\beta}\big)=0$. Then for any $\var>0$, there exists a positive constant $C(n,\al,R,\var)$ such that for all $\mu\in [\var,1]$, any positive solution $v$ of \eqref{maineq} with $R$ replaced by $R_\mu=\mu R+(1-\mu)$ satisfies
    \be\label{0.9}
    1/C(n,\al,R,\var)\leq v(\zeta)\leq C(n,\al,R,\var),\quad \forall\,\zeta\in \S^{2n+1}.
    \ee
\end{thm}

We record several observations on the assumptions of Theorem~\ref{thm:homogeneous-compactness}.
\begin{rem}
\begin{enumerate}[(1)]
\item
For fixed $\alpha\in(0,1)$, the regularity assumptions on $R$ and
the nondegeneracy condition \eqref{0.6} imply
$\beta\geq 2+\alpha$, which is automatically satisfied
when $n\geq2$. In CR dimension one, every $\beta\in(2,4)$
is admissible for suitable $\alpha$. For instance, one may take locally
$R(\xi)=R_0+c\rho(\xi)^\beta$, where $R_0>0$,
$c\neq0$, and $0<\alpha<\min\{1,\beta-2\}$.

\item
The weighted flatness condition $(*)'_\beta$ is imposed
independently of the homogeneous expansion. In particular,
the first-order remainder estimates do not imply the weighted
second- and mixed-derivative bounds in
Definition~\ref{def:weighted-flatness}, even when
$W\in C^{1,1}$. Only $C^1$ regularity of the homogeneous
term away from the origin is needed in
Proposition~\ref{prop:wf-globalization} and
Theorem~\ref{thm:5.5}; the higher-order estimates are
provided separately by $(*)'_\beta$.

\item
For a fixed critical point, let $Q^{(\beta)}$
denote the corresponding homogeneous term in
Theorem~\ref{thm:homogeneous-compactness}, and define
\[
\mathcal M(\eta)
=\int_{\Hn}Q^{(\beta)}(\eta\circ\xi)
U_{0,1}^{Q^*}\,\d\xi,
\quad  \eta\in\Hn.
\]
Since $\beta<Q$, differentiation under the integral is
justified, and
\[
\overline\nabla_\eta\mathcal M(\eta)
=\int_{\Hn}\overline\nabla Q^{(\beta)}
(\eta\circ\xi)U_{0,1}^{Q^*}\,\d\xi.
\]
Consequently, condition \eqref{0.8in jde} is equivalent
to requiring that zero be a regular value of
$\mathcal M$.
Moreover, the restriction $\beta<Q$ is sharp for the
absolute convergence of this integral, since for a nonzero
homogeneous function the behavior at infinity is determined by
\[
\int_1^\infty r^{\beta-Q-1}\,\d r.
\]
The absolute-convergence restriction $\beta<Q$ does not imply failure of compactness when $\beta\geq Q$.

\item
The lower threshold $\beta>Q-2$ arises from the
blow-up estimates. Let $\xi_i$ be local maximum points of $u_i$ with
$m_i=u_i(\xi_i)\to\infty$, and set
$\ell_i=m_i^{-2/(Q-2)}$.
We write $L_i=L(\beta,i)$ for the constants
appearing in condition $(*)'_\beta$.
The normalized integral in
Corollary~\ref{cor:4.15} is bounded by
$o(1)+CL_i\ell_i^{\beta-(Q-2)}$,
where $C$ is independent of $i$ and $o(1)\to0$
as $i\to\infty$.
Consequently, bounded $L_i$ suffice when
$\beta>Q-2$, whereas at $\beta=Q-2$
the same estimate is obtained under the additional
assumption $L_i\to0$.
The conditions in Corollary~\ref{cor:4.15} are sufficient for the estimate; we do not claim that they are necessary for compactness.
\end{enumerate}
\end{rem}

We denote by $S_1^2(\S^{2n+1})$ the Folland--Stein space obtained as the completion of $C^\infty(\S^{2n+1})$ under the norm
\begin{equation}\label{eq:sphere-energy-norm}
\|v\|_{S_1^2(\S^{2n+1})}:=\Big(\int_{\S^{2n+1}} v L_{\theta_0}v\,\d V_{\theta_0}\Big)^{1/2};
\end{equation}
see \cite{FS1974}.
Under the weighted flatness condition $(*)'_{Q-2}$,
we obtain the following uniform energy estimate for positive solutions of \eqref{maineq}.
\begin{thm}\label{thm:intro-energy}
Let $\al \in (0,1)$ and let $R\in\mathcal A_\al (\S^{2n+1})$ be positive.  Assume that, for some $d>0$, $R$ satisfies $(*)'_{Q-2}$ with uniform constants $L>0$ and $C_*\ge1$ in
\[
\Om _d=\{\zeta\in\S^{2n+1}:|\nabla_{\theta_0}R(\zeta)|<d\}.
\]
Then any positive solution $v$ of \eqref{maineq} satisfies
\begin{equation}\label{eq:intro-energy}
\|v\|_{S_1^2(\S^{2n+1})}\le C,
\end{equation}
where $C$ depends only on $n$, $\al$, $\inf _{\S^{2n+1}} R$, $\|R\|_{\mathcal{A}_\al (\S^{2n+1})}$, $C_*$, $L$ and $d$.
\end{thm}

 By \eqref{maineq}, \eqref{eq:sphere-energy-norm}, and the positive lower and upper bounds for $R$, estimate \eqref{eq:intro-energy} is equivalent to $$
 \|v\|_{L^{Q^*}(\S^{2n+1})}\leq C.
 $$
However, the estimate \eqref{eq:intro-energy} is not sufficient to imply an $L^\infty$-bound for $v$ on $\S^{2n+1}$. For instance, define 
$$
v_{P, \lam }=|\operatorname{det} \d \varphi_{P, \lam }^{-1}|^{(Q-2) /(2 Q)},
$$
where $\varphi_{P,\lam}$ is defined in \eqref{def:conformal-map}. Jacobians of sphere diffeomorphisms are taken with respect to $\d\si $, equivalently $\d V_{\theta_0}$.
By a change of variables, we have
\[
    \int_{\S^{2 n+1}} v_{P, \lam }^{Q^*}\, \d  V_{\theta_0}=\operatorname{Vol}(\S^{2 n+1}, \theta_0), 
\]
but $v_{P, \lam }(P)=\lam ^{(Q-2) / 2} \to  \infty$ as $\lam\to  \infty$. Since these functions solve \eqref{maineq} with $R\equiv1$, the same family shows that the uniform bounds in Theorem~\ref{thm:homogeneous-compactness} cannot include the endpoint $\mu=0$, where $R_\mu\equiv1$. Multiple blow-up points occur in the classical Nirenberg problem \cite{Li1996}. For \eqref{maineq}, the following theorem excludes multiple blow-up points under a stronger weighted flatness condition.
\begin{thm}\label{thm:intro-onebubble}
Let $\al \in(0,1)$ and let $\{R_i\}\subset\mathcal A_\al (\S^{2n+1})$ be uniformly positive and bounded in the norm \eqref{eq:intrinsic-coefficient-norm}.  Suppose also that $\{R_i\}$ satisfies condition $(*) _\beta'$ for some constants $\beta>Q-2$, $L,d>0$ and a structural constant $C_*$ independent of $i$, in $\Om_{d,i}:=\{\zeta \in \S^{2n+1}: |\nabla_{\theta_{0}} R_i(\zeta)|<d\}$. Let $\{v_i\} \subset C^2(\S^{2n+1})$ be a sequence of corresponding positive solutions of \eqref{maineq} and $v_i(\zeta_i)=\max _{\S^{2n+1}} v_i$ for some $\zeta_i$. 

Then, after passing to a subsequence, either $\{v_i\}$ is bounded in $L^{\infty}(\S^{2n+1})$, or it blows up at exactly one point in the strong sense: There exists a sequence $\{\varphi_i\}$ of CR diffeomorphisms from $\S^{2n+1}$ to $\S^{2n+1}$ satisfying $\varphi_i(\zeta_i)=\zeta_i$ and $|\operatorname{det} \d \varphi_i(\zeta_i)|^{(Q-2)/(2Q)}=R_i(\zeta_i)^{-n/2}v_i^{-1}(\zeta_i)$ such that
	\[
\|R_i(\zeta_i)^{n/2}T_{\varphi_i} v_i-1\|_{C^0(\S^{2n+1})} \to  0 \quad \text { as }\, i \to  \infty,
	\]
	where $T_{\varphi_i} v_i:=(v_i \circ \varphi_i)|\operatorname{det} \d \varphi_i|^{(Q-2)/(2Q)}$.
\end{thm}

\begin{rem}
The CR automorphisms $\varphi_i$ play the role of
M\"obius transformations in the classical Nirenberg
problem. They are obtained by conjugating the Heisenberg
dilations $\delta_\lam(z,t)=(\lam z,\lam^2t)$
with Cayley charts centered at $\zeta_i$.
The dilation parameters are chosen such that
$R_i(\zeta_i)^{n/2}(T_{\varphi_i}v_i)(\zeta_i)=1$.
The corresponding charts and automorphisms are
defined in \eqref{eq:normalized-cayley} and
\eqref{def:conformal-map}, respectively.
\end{rem}

The above a priori estimates and compactness results have counterparts in the classical and fractional Nirenberg problems. For the classical Nirenberg problem, related results were established by Chang--Gursky--Yang \cite{CGY1993} and Schoen--Zhang \cite{SZ1996} in dimension three, and by Li \cite{Li1995} in higher dimensions; for the fractional problem, see Jin--Li--Xiong \cite{JLX2014}. In the CR setting, related compactness results include those of Afeltra \cite{Afeltra2024} for the three-dimensional CR Yamabe problem, Li--Tang--Zhou \cite{LTZ2026} for the prescribed Webster scalar curvature problem on the CR three-sphere, and Qiang et al.~\cite{QTZZ} for subelliptic equations with a potential on the Heisenberg group. We refer to these works and the references therein for further developments.

Our proofs follow the blow-up approach developed by Schoen--Zhang \cite{SZ1996} and Li \cite{Li1995}. We first establish the energy bound in Theorem~\ref{thm:intro-energy}, then reduce possible concentration to a single point in Theorem~\ref{thm:intro-onebubble}, and finally exclude the remaining blow-up under the assumptions of Theorem~\ref{thm:homogeneous-compactness}. The hypotheses are strengthened to exclude multiple blow-up points and then a single blow-up point.

A central difficulty in the CR setting is the different scaling of horizontal and Reeb derivatives under Heisenberg dilations. The bounds \eqref{eq:wf-first}--\eqref{eq:wf-higher} estimate both types of derivatives in the Pohozaev identity \eqref{Pohozaev}. The class $\mathcal A_\al $, together with the mixed-derivative assumptions, provides the regularity needed for rescaling without requiring a bound for $T^2R$. The distinction is particularly relevant in CR dimension one. Proposition~\ref{prop:4.17} shows that both the horizontal and Reeb derivatives of $R_i$ vanish in the limit at isolated blow-up points.

We begin by rescaling solutions near their local maxima. The classification of entire positive solutions on the Heisenberg
group gives the bubble approximation in Proposition~\ref{prop:4.4}, without assuming a global energy bound.
We then prove the sharp pointwise estimate \eqref{eq:sharp-isolated-bound} near isolated simple blow-up points.
Combining \eqref{eq:sharp-isolated-bound} with \eqref{Pohozaev}, we exclude a second concentration scale in Proposition~\ref{prop:4.16} and establish uniform separation in Proposition~\ref{prop:5.2}. At the flatness order $Q-2$, uniform separation bounds the number
of peaks, while Lemma~\ref{lem:4.11} bounds the energy near each peak and in the remaining region. The resulting energy bound is stated in Theorem~\ref{thm:5.3}, whose critical case gives \eqref{eq:intro-energy}.

The stronger weighted flatness condition in Theorem \ref{thm:intro-onebubble} is used to exclude multiple blow-up points.
For $\beta>Q-2$, the normalized terms involving derivatives of $R_i$ vanish by \eqref{eq:normalized-pohozaev-vanishing}.
If two or more blow-up points were present, the contribution from another point would produce a strictly positive regular part in
the normalized limiting profile. The resulting sign of the Pohozaev boundary term contradicts the
limiting identity, by Lemma~\ref{lem:positivepoho}. The contradiction proves the single-point alternative in Theorem~\ref{thm:5.4}.
To obtain the global convergence in Theorem~\ref{thm:intro-onebubble}, we combine the local bubble profile with the sharp decay estimate and uniform H\"older bounds for the conformally normalized solutions.

Theorem~\ref{thm:5.5} excludes a single blow-up point under the homogeneous expansion assumptions of Theorem~\ref{thm:homogeneous-compactness}.
The nondegeneracy condition \eqref{0.6} gives the displacement estimate \eqref{eq:center-shift}. After rescaling, \eqref{eq:global-translation-identities} gives the vanishing of the derivative integrals in \eqref{0.8in jde}.
Identity \eqref{eq:global-dilation-identity} and homogeneity give the vanishing of the remaining integral, contradicting \eqref{0.8in jde}. For $\beta<Q$, the integrals
used in this limiting argument converge.

Combining Theorems \ref{thm:5.4} and \ref{thm:5.5} gives the compactness result for a fixed prescribed curvature function in Corollary \ref{thm:5.6}. Proposition \ref{prop:wf-globalization} makes these arguments uniform along $R_\mu=1+\mu(R-1)$ for $\mu\in[\var,1]$.
The Harnack argument yielding \eqref{eq:uniform-positive-lower-bound} completes the proof of \eqref{0.9}. The resulting energy estimates, one-point concentration alternative, and compactness criteria apply in every CR dimension under the stated assumptions. In particular, the energy bound is derived from the local analysis rather than assumed in advance, and the compactness estimates remain uniform along this homotopy with $\mu$ bounded away from zero. The estimates provide the analytic foundation for the existence and degree results in our subsequent paper \cite{QTWPartII}.

The paper is organized as follows.
Section~\ref{sec:2} recalls the background on CR geometry,
the Heisenberg group, intrinsic regularity, and weighted flatness. In
Section~\ref{sec:3} we establish accurate blow up profiles of solutions of \eqref{maineq} near
isolated blow up points. In fact most of the estimates hold also for subcritical approximations to
such equations as well including in bounded domains of $\Hn$. In Section~\ref{sec:4}, we provide $S_{1}^2(\S^{2n+1}) $-norm
a priori estimates, at most one isolated simple blow up point, and $L^{\infty}(\S^{2n+1})$-norm a priori estimates
for solutions of \eqref{maineq} under appropriate hypotheses on $R$.
Section~\ref{sec:4} establishes a priori estimates and examines
the number and separation of blow-up points.
The proofs of the main theorems are completed in
Section~\ref{sec:5}.

\medskip

\noindent\textbf{Notation.}
Throughout the paper, we use the following notation and conventions.
\begin{itemize}
\setlength{\itemsep}{4pt}

\item
$Q=2n+2$ and $Q^*=2Q/(Q-2)$ denote the homogeneous dimension
and the critical Sobolev exponent, respectively.
For $\xi=(z,t)\in\Hn$, we write
$|\xi|=\rho(\xi)=(|z|^4+t^2)^{1/4}$ and
$d(\xi,\eta)=|\eta^{-1}\circ\xi|$.

\item
$Z$ and $\overline Z$ denote the arrays of left- and
right-invariant horizontal vector fields defined in
\eqref{vectorfields} and \eqref{right-vectorfields}.
We set $\nabla_\H=\frac12 Z$,
$\widetilde\nabla=\overline Z$,
$\overline\nabla=(\overline Z,T)$, and $T=\pa_t$.

\item
$\nabla_\theta$ and $T_\theta$ denote the horizontal gradient
and the Reeb vector field, respectively.
The latter is distinct from the coordinate vector field $T$.
The balls $B_r(\xi)$ and $\B(P,r)$ are defined using
$d$ and the spherical contact distance
\eqref{eq:contact-distance}, respectively.

\item
$F$ denotes the inverse Cayley transform \eqref{Cayley},
and $F_P$ is the centered chart defined in
\eqref{eq:normalized-cayley}.
The functions $\Lam$ and $U_{0,1}=2^n\Lam$
are defined in \eqref{standard-bubble}.

\item
On $\Hn$, $\d\xi=\d x\,\d y\,\d t$ denotes Haar measure,
and $\d S$ denotes Euclidean surface measure.
On $\S^{2n+1}$, integrals are taken with respect to
$\d\si$ unless otherwise specified, where
$\d V_{\theta_0}=2^{2n+1}n!\,\d\si$.

\item
The constant $C>0$ may vary from line to line. Its dependencies are
specified in the statements. In proofs, dependence on the fixed data,
coordinate neighborhoods, auxiliary radii and cutoff functions may be
suppressed. Unless otherwise stated, constants in estimates for sequences
are independent of $i$.

\item
We write $a\asymp b$ if $cb\leq a\leq Cb$
for some fixed constants $c,C>0$.
Unless otherwise stated, $O(\cdot)$ and $o(\cdot)$
refer to the limit along the sequence under consideration.

\end{itemize}


\section{Preliminaries}\label{sec:2}

We now recall the geometric and analytic facts used in the proofs, with the notation and assumptions on $R$ fixed in the introduction.

\subsection{CR geometry and the conformal sub-Laplacian}\label{sec:2.1}

Let $M$ be an orientable CR manifold without boundary of CR dimension $n$. Thus $M$ is an orientable differentiable manifold of real dimension $(2 n+1)$ endowed with a pair $(H(M), J)$, where $H(M)$ is a subbundle of the tangent bundle $T(M)$ of real rank $2 n$ and $J$ is an integrable complex structure on $H(M)$. Since $M$ is orientable, there exists a 1-form $\theta$ called a \emph{pseudo-Hermitian} structure on $M$. Then, we can associate with each structure $\theta$ a bilinear form $G_\theta$, called the \emph{Levi form}, which is defined only on $H(M)$ by
\[
G_\theta(X, Y)=-(\d \theta)(J X, Y), \quad \forall\, X, Y \in H(M) .
\]
Since $G_\theta$ is symmetric and $J$-invariant, we then call $(M, \theta)$ a \emph{strictly pseudo-convex} CR manifold if the Levi form $G_\theta$ associated with the structure $\theta$ is positive definite. The structure $\theta$ is then a contact form which immediately induces on $M$ the volume form $\d V_\theta$, i.e., $\d V_\theta=\theta \wedge (\d \theta)^n$. The volume of $M$ with respect to the contact form $\theta$ is denoted by $\operatorname{Vol}(M, \theta)$, i.e., $\operatorname{Vol}(M, \theta)=\int_M \,\d V_\theta$.

Moreover, $\theta$ on a strictly pseudo-convex CR manifold $(M, \theta)$ also determines a \emph{normal} vector field $T_\theta$ on $M$, called  \emph{the Reeb vector field} of $\theta$. Via the Reeb vector field $T_\theta$, one can extend the Levi form $G_\theta$ on $H(M)$ to a Riemannian metric $g_\theta$ on $T(M)$, called the Webster metric of $(M, \theta)$. Let
\[
\pi_H: T(M) \to  H(M)
\]
be the projection associated to the direct sum $T(M)=H(M) \oplus \mathbb{R} T_\theta$. Now, with the structure $\theta$, we can construct a unique affine connection $\nabla^{\mathrm{TW}}$, called the \emph{Tanaka-Webster connection} on $T(M)$. Using $g_\theta$ and $\pi_H$, we can define the \emph{horizontal gradient} $\nabla_\theta$ by
\[
\nabla_\theta u=\pi_H \operatorname{grad}_{g_\theta}u .
\]

Using the metric $g_\theta$ and the projection $\pi_H$, one can define the sub-Laplacian $\Delta_\theta$ acting on a $C^2$-function $u$ via
\[
\Delta_\theta u=\operatorname{div}(\pi_H \operatorname{grad}_{g_\theta}u) .
\]
Here $\operatorname{grad}_{g_\theta}u$ is the ordinary gradient of $u$ with respect to $g_\theta$ which can be written as $g_\theta(\operatorname{grad}_{g_\theta}u, X)=$ $X(u)$ for any $X$. Then integration by parts gives
\[
\int_M(\Delta_\theta u) f \,\theta \wedge(\d \theta)^n=-\int_M\langle\nabla_\theta u, \nabla_\theta f\rangle_\theta \,\theta \wedge(\d \theta)^n
\]
for any $f\in C_c^\infty(M)$. In the preceding formula, $\langle \cdot,\cdot\rangle_\theta$ denotes the inner product via the Levi form $G_\theta$ (or the Webster metric $g_\theta$ since both $\nabla_\theta u$ and $\nabla_\theta f$ are horizontal).

The Webster metric $g_\theta$ and the Tanaka--Webster connection $\nabla^{\mathrm{TW}}$ determine the pseudo-Hermitian Ricci and scalar curvatures. We denote the latter by $\mathrm{Scal}_\theta$, the Webster scalar curvature; see \cite[Definition~1.30]{DT2006}.

Since a pseudohermitian structure is defined only up to a conformal factor on a CR manifold, the CR Yamabe problem is a natural analogue of the Yamabe problem in Riemannian geometry. If $\widehat{\theta}=u^{2 / n} \theta$ for some smooth function $u>0$, the transformation law of the Webster curvature is
\[
\operatorname{Scal}_{\widehat{\theta}}=u^{-\frac{n+2}{n}}\Big(-\frac{2(n+1)}{n} \Delta_\theta u+\operatorname{Scal}_\theta u\Big).
\]
Clearly, the problem of solving $\operatorname{Scal}_{\widehat{\theta}}=h$ is equivalent to finding positive solutions $u$ to the following PDE:
\be\label{Yamabe}
-\Delta_\theta u+\frac{n}{2(n+1)} \operatorname{Scal}_\theta u=\frac{n}{2(n+1)} h u^{1+2 / n} \quad \text { on }\, M .
\ee
When $h$ is constant, \eqref{Yamabe} is known as the CR Yamabe problem.

Basic examples of CR manifolds include real hypersurfaces in $\C^{n+1}$; for example, any odd-dimensional unit sphere $\S^{2 n+1} \subset \C^{n+1}$ is a strictly pseudo-convex CR manifold.  Indeed, let $\theta_0$ be the standard contact form on the sphere $\S^{2 n+1}=\{\zeta=(\zeta^1, \ldots, \zeta^{n+1}) \in \C^{n+1} :| \zeta|^2=\sum_{j=1}^{n+1}|\zeta^j|^2=1\} \subset \C^{n+1}$, i.e.,
$$
\theta_0=\mathrm{i}(\bar{\pa}-\pa)|\zeta|^2=\mathrm{i} \sum_{j=1}^{n+1} (\zeta^j \, \d \bar{\zeta}^j-\bar{\zeta}^j \, \d \zeta^j) .
$$
Then $(\S^{2 n+1}, \theta_0)$ is a compact strictly pseudoconvex CR manifold of real dimension $(2 n+1)$. We use $g_{\theta_0}$ for full gradients and tensor norms. Distances in concentration estimates are measured by the contact distance \eqref{eq:contact-distance}. The distance \eqref{eq:contact-distance} is locally equivalent to the Carnot--Carath\'eodory distance, but not to the ordinary Riemannian distance of $g_{\theta_0}$ at the level of scaling. 
The Heisenberg group $\Hn$ is the corresponding flat example.   $\Hn$ plays a role among pseudoconvex pseudohermitian manifolds analogous to the role of $\Rn$ among Riemannian manifolds. In fact,  every pseudoconvex pseudohermitian manifold can locally be approximated by $\Hn$, through coordinates analogous to the normal coordinates of Riemannian geometry known as pseudohermitian normal coordinates.

For the standard sphere, we use the coordinate expression for the sub-Laplacian given in \eqref{Delta}.

With the surface measure $\d\si $ fixed above, Geller \cite{G1981} showed that, for regular $f:\S^{2n+1}\to\C$,
\begin{equation}\label{Green}
L_{\theta_0}^{-1}(f)(\zeta)
=c_n\int_{\S^{2n+1}}\operatorname{dist}(\zeta,\eta)^{-2n}f(\eta)\,\mathrm d\si (\eta),
\end{equation}
where $\Gamma$ denotes the Gamma function, $c_n=\frac{2^{n-1}\Gamma(\frac n2)^2}{\pi^{n+1}}$, and
$\operatorname{dist}(\zeta,\eta)=d_{\theta_0}(\zeta,\eta)$ as in \eqref{eq:contact-distance}.

\subsection{The Cayley transform, bubbles, and the Pohozaev identity}\label{sec:heisenberg-prelim}

We use the Heisenberg group law, dilations and vector fields introduced above. With the normalizations in \eqref{Green} and \eqref{Cayley}, conformal covariance gives
\be\label{1.7}
(L_{\theta_{0}} \phi) \circ F=-(2|J_F|)^{-(Q+2)/(2Q)} \Delta_{\H}((2|J_F|)^{(Q-2)/(2Q)}(\phi \circ F))\quad \text{ for }\,\phi \in C^{\infty}(\S^{2n+1}),
\ee
where  $$|J_F|=\frac{2^{2 n+1}}{((1+|z|^2)^2+t^2)^{n+1}}$$ is the surface Jacobian of $F$, characterized by $F^*(\d\si )=|J_F|\,\d\xi$.  The symbol $\circ$ denotes the Heisenberg group product when both entries are points of $\Hn$, and ordinary composition when applied to maps; the meaning will always be clear from the arguments. The multiplicative factor 2 in the  formula  \eqref{1.7} appears because we use the standard volume elements for $\Hn$ and $\S^{2n+1}$ instead of the volume forms associated with the standard contact forms on these two spaces,  see \cite{JL1987}.

Since the Heisenberg group has zero Webster curvature, setting $u=(2|J_F|)^{(Q-2)/(2Q)}(v\circ F)$ and $R(\xi)=R\circ F(\xi)$ transforms \eqref{maineq} into \eqref{maineq1}.
Conversely, if $u$ is a positive $C^2$ solution of \eqref{maineq1}, then
\[
v=[(2|J_F|)^{-\frac{Q-2}{2 Q}} u] \circ F^{-1}
\]
solves \eqref{maineq} on $\S^{2n+1}\backslash\{\mathcal S\}$. It corresponds to a positive $C^2$ solution on the whole sphere if and only if $v$ extends to a positive $C^2$ function across $\mathcal S$. In that case, the equation also holds at $\mathcal S$ by continuity. Thus the equivalence between \eqref{maineq} and \eqref{maineq1} is understood within this class of solutions.

Equations \eqref{maineq} and \eqref{maineq1} involve the critical exponent for the CR Sobolev embedding (or the Folland-Stein inequality \cite{FS1974}) and they will be called the \emph{CR Nirenberg equations}, where $R$ represents the normalized Webster scalar curvature function.

When $R\equiv 1$, \eqref{maineq1} becomes
\be\label{CRYamabe}
-\Delta_{\H} u=\frac{n^2}{4} u^{\frac{Q+2}{Q-2}}, \quad u>0 \quad  \text{ in }\,\Hn.
\ee
Jerison--Lee \cite{JL1988} proved that every finite-energy solution of \eqref{CRYamabe} is a CR Yamabe bubble. With $\Lam $ defined in \eqref{standard-bubble}, the bubble family can be written as
\begin{equation}\label{bubbles}
U_{a,\lam }(\xi)
=2^n\lam ^{(2-Q)/2}\Lam \bigl(\delta_{\lam ^{-1}}(a^{-1}\circ\xi)\bigr),
\quad a\in\Hn,\quad \lam >0.
\end{equation}
Garofalo--Vassilev \cite[Theorem~1.5]{GV2001} classified finite-energy entire solutions under cylindrical symmetry on groups of Iwasawa type, including $\Hn$.
  Catino, Li, Monticelli and Roncoroni \cite[Theorems~1.1 and~1.2]{CLMR2023} classified all positive solutions in $\H^1$ and positive solutions in $\Hn$, $n\geq2$, under a suitable decay condition at infinity; the latter condition does not itself imply finite energy. Very recently, a complete classification of positive entire solutions to the critical CR Yamabe equation for $n\geq 2$ has been obtained independently by Li--Liu--Ma \cite{LLM2026} and Li \cite{L2026}.

Let $E=\dot S_1^2(\Hn)$ be the completion of $C_c^\infty(\Hn)$ under the norm
\[
\|u\|_E=\Big(\int_{\Hn}|\nabla_{\H}u|^2\,\d\xi\Big)^{1/2}.
\]
We write $S^2_{1,loc }(\Hn)$ for the functions $u$ such that $u$ and $\nabla_\H u$ belong to $L^2_{loc }(\Hn)$.
With $Q^*=2Q/(Q-2)$ as above, the Folland--Stein embedding $E\hookrightarrow L^{Q^*}(\Hn)$ holds \cite{FS1974}. The sharp constant was determined by Jerison--Lee \cite{JL1988}:
\be\label{Sobolevineq}
S_n \Big(\int_{\Hn} |u|^{Q^*}\,\d\xi\Big)^{1 / Q^*} \leq \Big(\int_{\Hn} |\nabla_{\H} u|^2\,\d\xi\Big)^{1 / 2},
\ee
where $S_{n}=\frac{ n \sqrt{\pi}}{(2^{2 n} n !)^{1 /(2(n+1))}}$ is the best constant. Then for every $(a,\lam)\in \Hn\times (0,\infty)$,  $U_{a, \lam }$ is the solution to
\eqref{maineq1} with $R\equiv 1$. Moreover, the functions in \eqref{bubbles} and their non-zero constant multiples attain the sharp Sobolev inequality \eqref{Sobolevineq} and such functions are usually called \emph{Jerison-Lee's bubbles}.

As in the Riemannian case, where the blow-up  analysis requires using the Pohozaev identity for $\Rn$, in the CR case we will need a Pohozaev formula for the Heisenberg group. Such roles of the Pohozaev-type identity in analyzing blow-ups were first observed by Schoen \cite{Schoen43,Schoen44}.
 Pohozaev-type formulas have already been studied on pseudohermitian manifolds (see \cite{Afeltra2024,Afeltra2024-2}), the Heisenberg group (see \cite{GL1992,PR2003})   and
  more general Carnot groups (see \cite{CV2000}).

We associate any  point $(z, t)=(x,y,t) \in \Hn$  with a $(2n+1)\times (2n+1)$ symmetric matrix $A=A(\xi)=(a_{ij}(\xi))$ defined by
$$
\left( \begin{array}{ccc}
I_n & 0_n & 2 y \\
0_n & I_n & -2 x \\
2 y^{\mathsf T} & -2 x^{\mathsf T} & 4|z|^2
\end{array}\right) ,
$$
where $x,y$ are column vectors and $I_n,0_n$ are the $n\times n$ identity and zero matrices.
The matrix $A$ is related to  $\Delta_{\H}$ by the formula $\Delta_{\H} = \frac{1}{4}\operatorname{div}(A \nabla)$, where $\nabla$ denotes the gradient in $\R^{2n+1}$.

Let $\Om \subset \Hn $ be an open set. 
Recall that
\begin{align*}
\Xi =\sum_{j=1}^n  (x_j X_j+y_j Y_j) +2 t T
\end{align*}
is the infinitesimal generator of the Heisenberg dilations. Using this vector field, we obtain the following Pohozaev identity.
\begin{lem} \label{prop:4.7}
Let $B_{\si}$ be a ball in $\Hn$
		centered at the origin with radius $\si>0$, and $p\geq 1$. Let $U$ be an open neighborhood of $\overline{B}_\si $. Suppose that $R\in C^1(U)$ and $u \in C^2(U)$ satisfies \[
        -\Delta_{\H}u = \frac{n^2}{4}R(\xi)|u|^{p-1}u\quad \text{ in }\, B_\si.
        \]
		Then  we have
		\begin{align}
		&\,\frac{n^2}{4}\Big( \frac{Q}{p+1} - \frac{Q-2}{2}\Big) \int_{B_\si }R|u|^{p+1}+ \frac{n^2}{4}\frac{1}{p +1}\int_{B_\si } \Xi (R) |u|^{p+1}\notag\\=&\, \int_{\pa {B_\si }}B(\si , \xi, u, \nabla u)+\frac{n^2}{4}\frac{1}{p +1}\int_{\pa  B_\si } R |u|^{p+1} \Xi  \cdot \nu,\label{Pohozaev}
		\end{align}
		where $\nu$  is the outward Euclidean unit normal to $\pa B_{\si}$,	$\Xi  \cdot \nu=\Xi  \cdot \frac{\nabla d}{|\nabla d|}=\frac{\Xi  d}{|\nabla d|}=\frac{d}{|\nabla d|}$ with $d=d(\xi,0)=\rho(\xi)$,   and \[
		B(\si , \xi, u, \nabla u) =
		\frac{Q-2}{8}( A\nabla u \cdot \nu)u -\frac{1}{2}|\nabla_{\H} u|^2 \Xi  \cdot \nu +\frac{1}{4}(A \nabla u \cdot \nu) \Xi  (u).
\]
	\end{lem}
\begin{proof}
The proof  can be found  in \cite[Theorem 2.1]{GL1992} (or \cite[Theorem 4.1]{PR2003}), so we omit it.
\end{proof}
On a translated ball $B_\si (a)$, identity \eqref{Pohozaev} is applied to $u_a(\eta)=u(a\circ\eta)$ on $B_\si $; coordinate gradients in the boundary term are taken with respect to $\eta$. The boundary term $B(\si , \xi, u, \nabla u)$ has the following properties:
\begin{lem}\label{lem:positivepoho}
	\begin{itemize}
		\item[$(i)$]For $u(\xi) =|\xi|^{2-Q}$ and any $\si>0$, we have $B(\si, \xi, u, \nabla u)=0$ for  all  $\xi \in \pa B_\si$.
		\item[$(ii)$]		Let $u(\xi) = |\xi|^{2-Q} + a_0 + h(\xi)$, where $a_0>0$ is a  constant and $h \in C^1(U)$  for some neighborhood $U$ of the origin with $h(0)=0$. Then
we have
		$$
		\lim _{\si  \to  0} \int_{\pa  B_\si } B(\si , \xi, u, \nabla u)=-\frac{\sqrt{\pi}\Gamma(\frac{n+1}{2})}{8\Gamma(\frac{n}{2}+1)} a_0 (Q-2)^2|\S ^{2 n-1}|<0,
		$$
		where $\Gamma$ is the Gamma function and $|\S ^{2 n-1}|$ is the surface measure of the unit sphere in $\R ^{2n}$.
	\end{itemize}
\end{lem}
\begin{proof}
The proof can be found in \cite[Proposition 4.3]{PR2003}, so we omit it here.
\end{proof}

\subsection{Regularity and weighted flatness}\label{function space}

We first recall the higher-order Folland--Stein spaces and then record the regularity and flatness properties of the functions in Definition~\ref{def:intrinsic-coefficients}.

For each positive integer $k$, let
\[
\mathscr B_k=\{Z_{a_1}\cdots Z_{a_j}:1\le a_\ell\le2n,\ 1\le\ell\le j,\ 1\le j\le k\},
\]
where $Z_a$ are the left-invariant horizontal fields listed in \eqref{vectorfields}.

Let $\Om \subset\Hn$ be open. We use the Folland--Stein spaces $\Gamma_\beta(\Om )$ \cite{FS1974}, whose definition for $0<\beta<1$ was given above. If $\beta=1$, define $$\Gamma_{1}(\Om)=\Big\{f\in L^{\infty}(\Om)\cap C^0(\Om):\sup_{\substack{\xi\in \Om, \zeta\in\H^n\backslash\{0\}\\ \xi\circ\zeta, \xi\circ \zeta^{-1}\in\Om }}\frac{|f(\xi\circ\zeta)-2f(\xi)+f(\xi\circ\zeta^{-1})|}{|\zeta|}<\infty\Big\}.$$
If $\beta=k+\al$ with $k\in \mathbb{N}^+$  and $\al\in (0,1)$, define
$$\Gamma_{k+\al}(\Om)=\Big\{f\in L^{\infty}(\Om)\cap C^0(\Om):\mathscr{L}f\in \Gamma_{\al}(\Om)\,\text{ for $\mathscr{L}\in \mathscr{B}_k$} \Big\}.$$
All derivatives are understood in the distributional sense. For $\beta=1$, the norm is the sum of the $L^{\infty}$ norm and the displayed supremum. For higher orders, we set
$$
\|f\|_{\Gamma_{k+\al }(\Om )}=\|f\|_{L^{\infty}(\Om )}+\sum_{\mathscr{L} \in \mathscr{B}_k}\|\mathscr{L} f\|_{\Gamma_\al (\Om )} .
$$

  On $\S^{2 n+1}$, these spaces and their norms are defined using a finite contact atlas and a subordinate smooth partition of unity; see \cite{FS1974}.
  Note that we can  identify $\Hn$ with its Lie algebra which is Euclidean
space $\R^{2n+1}$ with the Euclidean norm $|\cdot|_{\mathrm{Eu}}$ and the linear coordinates $x_j$ via the exponential map. Hence we can discuss the usual function spaces $C^k$ for $0\leq k\leq \infty$. We refer to \cite{FS1974,F1975} for more details and regularity results.

We return to the class $\mathcal A_\al $ in Definition~\ref{def:intrinsic-coefficients}. Both ordinary $C^{1,1}$ functions and functions in $\Gamma_{2+\al }$ belong locally to $\mathcal A_\al $. For $R\in\Gamma_{2+\al}$, use $T=-[X_j,Y_j]/4$ and $\pa _{x_j}=X_j-2y_jT$, $\pa _{y_j}=Y_j+2x_jT$: the weak coordinate first derivatives are continuous, hence the function is $C^1$. For $R\in C^{1,1}$, the bounds in \eqref{eq:intrinsic-coefficient-norm} follow by expanding the vector fields. In particular, ordinary \(C^{1,1}\) functions are included in \(\mathcal A_{\al ,loc}\), while a bound for \(T^2R\) is not required.

\begin{lem}\label{lem:intrinsic-coefficient-calculus}
On fixed compactly contained coordinate domains the class $\mathcal A_\al $ is preserved by smooth contact changes of coordinates, by multiplication by a positive $C^{1,1}$ function, and by compositions $R\mapsto a+bR$, with the corresponding uniform bounds. For $0<\si \le1$, bounded translations $\eta$, and $K(\xi)=R(\eta\circ\delta_\si \xi)$, one has
\begin{equation}\label{eq:intrinsic-rescaled-increments}
 \begin{aligned}
 |ZK(\xi)-ZK(0)|&\le C\si ^2|\xi|,\\
 |TK(\xi)-TK(0)|&\le C\si ^{2+\al }|\xi|^\al .
 \end{aligned}
\end{equation}
The first estimate also holds for the right-invariant horizontal fields centered at $0$. The constants in \eqref{eq:intrinsic-rescaled-increments} depend only on $n$, $\al$, a bound in
\eqref{eq:intrinsic-coefficient-norm}, and the fixed coordinate domains. The bounds under changes of coordinates, multiplication, and affine composition also depend on the corresponding coordinate bounds, the $C^{1,1}$ norm of the multiplier, and $|a|+|b|$, respectively.
\end{lem}
\begin{proof}

We prove \eqref{eq:intrinsic-rescaled-increments}; the invariance properties follow from the contact chain rule and the Leibniz rule. 

For $\Psi_\si (\xi)=\eta \circ \delta_\si  \xi$, the dilation rules give
$$
Z K=\si (Z R) \circ \Psi_\si , \quad T K=\si ^2(T R) \circ \Psi_\si , \quad d(\Psi_\si  \xi, \Psi_\si  0)=\si |\xi| .
$$
Therefore, we have
$$
|Z K(\xi)-Z K(0)| \leq C \si ^2|\xi|, \quad|T K(\xi)-T K(0)| \leq C \si ^{2+\al }|\xi|^\al  .
$$

If $\xi=(z, t)$, the centered right-invariant fields satisfy
$$
\overline{X}_j-X_j=-4 y_j T, \quad \overline{Y}_j-Y_j=4 x_j T .
$$
They therefore agree with the left-invariant fields at 0 . Since $T K=\si ^2(T R) \circ \Psi_\si $, for the corresponding arrays $\overline Z,Z$ we obtain
$$
|\overline{Z} K(\xi)-\overline{Z} K(0)| \leq|Z K(\xi)-Z K(0)|+C|z|\|T K\|_{\infty} \leq C \si ^2|\xi| .
$$
The last inequality gives the right-invariant version of the first estimate in \eqref{eq:intrinsic-rescaled-increments}.
\end{proof}

We record the following local regularity consequence.

\begin{lem}\label{lem:intrinsic-regularity} Let $u_i>0$ be weak solutions of
$$
-\Delta_\H u_i=K_i u_i^{p_i} \quad \text { in } \, B \subset \Hn ,
$$
where $K_i$ are uniformly bounded in $\mathcal{A}_\al (B)$ and $p_i$ range over a compact subset of $(1, \infty)$. If $u_i$ and $u_i^{-1}$ are locally uniformly bounded, then, for every $B' \Subset B$ and $0<\al '<\al $,
\be\label{sch new}
\|u_i\|_{C^{2, \al ' / 2}(B')} \leq C .
\ee
Here $C$ depends on $n$, $\al$, $\al'$, the compact interval containing the exponents $p_i$, the domains $B'\Subset B''\Subset B$, and uniform bounds for $\|K_i\|_{\mathcal A_\al(B'')}$, $\|u_i\|_{L^\infty(B'')}$ and $\|u_i^{-1}\|_{L^\infty(B'')}$, where $B''$ is any fixed intermediate domain.
In particular, $\{u_i\}$ is precompact in $C_{loc}^2(B)$.
\end{lem}
\begin{proof}Fix $\al '<\theta<\al $. Interior subelliptic and Schauder estimates \cite[Theorem~6.1]{F1975} give, on smaller balls,
$$
\|u_i\|_{\Gamma_{2+\theta}}+\|T u_i\|_{\Gamma_\theta} \leq C .
$$
Since $[T, \Delta_\H]=0$, in distributions, we have
$$
-\Delta_\H(T u_i)=(T K_i) u_i^{p_i}+p_i K_i u_i^{p_i-1} T u_i .
$$
The right-hand side is uniformly bounded in $\Gamma_\theta$. Applying \cite[Theorem~6.1]{F1975} to the equation for $Tu_i$, we obtain
$$
\|T u_i\|_{\Gamma_{2+\theta}} \leq C, \quad\|Z^2 u_i\|_{\Gamma_\theta}+\|Z T u_i\|_{\Gamma_\theta}+\|T^2 u_i\|_{\Gamma_\theta} \leq C .
$$
Expressing the coordinate derivatives in terms of $Z, T$, and using $d(\xi, \eta) \leq C|\xi-\eta|_{\mathrm{Eu}}^{1 / 2}$, where $|\cdot|_{\mathrm{Eu}}$ denotes the Euclidean norm, we obtain \eqref{sch new}. The compactness follows from the Arzel\`a--Ascoli theorem.
\end{proof}

We next record the properties of weighted flatness, as defined in Definition~\ref{def:weighted-flatness}, that will be used below.
The following lemma shows that the auxiliary function $r$ can be chosen explicitly in terms of the horizontal and central derivatives of $R$.
\begin{lem}\label{lem:wf-radius}
Let $\beta>2$ and suppose that $R$ satisfies $(*)'_\beta$ with $L$, $C_*$ and $r$.
Set
\[
 a=L^{-1}|\nabla_\H R|,\quad b=L^{-1}|TR|,\quad
 r_R=\max\{a^{1/(\beta-1)},b^{1/(\beta-2)}\}.
\]
Then $C^{-1}r\le r_R\le Cr$, where $C$ depends only on $n,\beta,C_*$.
The function $r_R$ is locally Lipschitz for $d$ with a bound depending
only on $n$, $\beta$ and $C_*$, and it may replace $r$ in
Definition~\ref{def:weighted-flatness} after changing $C_*$.
\end{lem}
\begin{proof}
The proof is elementary. We omit the details here.
\end{proof}

With uniformly bounded auxiliary functions, the following lemma shows that condition $(*)_\beta'$ implies $(*)_\gamma'$ for every $2 \leq \gamma<\beta$.


\begin{lem}\label{lem:wf-lowering}
Let $\beta>2$ and suppose that $R$ satisfies $(*)'_\beta$ with $L$, $C_*$, $r$ and $0\le r\le M<\infty$.
For $2<\gamma<\beta$, it satisfies $(*)'_\gamma$ with the same $L$, a structural constant depending only on $n,C_*,\beta,\gamma,M$,
and the auxiliary function
\[
 s=\max\{(L^{-1}|\nabla_\H R|)^{1/(\gamma-1)},(L^{-1}|TR|)^{1/(\gamma-2)}\}.
\]
It also satisfies $(*)'_2$, since $|\nabla_\H^2R|+|TR|\le C_*L M^{\beta-2}$.
\end{lem}
\begin{proof}
By \eqref{eq:wf-first}, \eqref{eq:wf-higher} and the boundedness of $r$, we have
\[
s\leq Cr,\quad
r^{\beta-2}\leq Cs^{\gamma-2},\quad
r^{\beta-3}\leq Cs^{\gamma-3}\quad\text{on }\{s>0\}.
\]
The bounds in \eqref{eq:wf-first} and \eqref{eq:wf-higher} with exponent $\gamma$ follow immediately, and the local Lipschitz continuity of $s$ follows as in the proof of Lemma \ref{lem:wf-radius}.
The case $\gamma=2$ follows directly from \eqref{eq:wf-higher} and $r\leq M$.
\end{proof}

The following lemma gives the scaling rules for the constant $L$ and auxiliary function $r$.
\begin{lem}\label{lem:wf-scaling}
Let $\beta\ge2$ and let $R_i$ satisfy $(*)'_\beta$ in an open set $\Om _i\subset\Hn$ with 
$L_i$ and $C_*$. For $\si _i>0$, $b_i>0$,
$a_i\in\R$, and $\eta_i\in\Hn$, set
\[
 \Psi_i(\xi)=\eta_i\circ\delta_{\si _i}\xi,
 \quad \widetilde R_i=a_i+b_iR_i\circ\Psi_i.
\]
On $\Psi_i^{-1}(\Om _i)$, the transformed constant is $\widetilde L_i=b_i\si _i^\beta L_i$ for every $\beta\ge2$. And when $\beta>2$, the transformed auxiliary function is $\widetilde r_i=\si _i^{-1}r_i\circ\Psi_i$. The structural constant $C_*$ is unchanged. 
\end{lem}
\begin{proof}
Since the proof is elementary, we omit the details here.
\end{proof}






The following local estimates follow from weighted flatness and are used to estimate the integrals involving derivatives of $R_i$ in Section~\ref{sec:3}.
\begin{lem}\label{lem:wf-increments}
Let $\beta>2$ and suppose that $R$ satisfies $(*)'_\beta$ in an open set $\Om \subset\Hn$ with $L$, $C_*$, $r$. There is a dimensional constant $\kappa_n>1$ such that, if $h=d(\xi,\eta)>0$ and $B_{\kappa_nh}(\eta)\Subset\Om $, then
 \begin{align*}
 |\nabla_\H R(\xi)-\nabla_\H R(\eta)|
   &\le CLh(r(\eta)+h)^{\beta-2},\\
 |TR(\xi)-TR(\eta)|
   &\le CLh(r(\eta)+h)^{\beta-3},\\
 |R(\xi)-R(\eta)|
   &\le CLh(r(\eta)+h)^{\beta-1},
 \end{align*}
where the  constant $C>0$ depends only on $n,\beta,C_*$. With $s=r(\eta)$,
\begin{equation}\label{eq:wf-split}
 Lh^2(s+h)^{\beta-2}
 \le C\bigl(h|\nabla_\H R(\eta)|+h^2|TR(\eta)|+Lh^\beta\bigr).
\end{equation}
For $\beta=2$, \eqref{eq:lf-two} instead gives
\begin{equation}\label{eq:lf-two-increment}
 |\nabla_\H R(\xi)-\nabla_\H R(\eta)|\le CLh,
 \quad |TR|\le C_*L.
\end{equation}
No mixed-derivative condition is required for \eqref{eq:lf-two-increment}.
\end{lem}
\begin{proof}
Since the proof is elementary, we omit it here.
\end{proof}


We next extend the local flatness condition to the whole sphere and verify that it is preserved along $R_\mu=1+\mu(R-1)$. 

\begin{prop}\label{prop:wf-globalization}
Let $R\in\mathcal A_\al (\S^{2n+1})$ have locally bounded weak derivatives $Z_a(TR)$ off its full critical set $\mathcal C=\{\nabla_{\theta_0} R=T_{\theta_0} R=0\}$. Suppose that, near each $P\in\mathcal C$, $R$ satisfies $(*)'_{\beta_P}$ for some $\beta_P>2$ and admits the following expansion in the chart $F_P$, where $R(\xi)$ abbreviates $R(F_P(\xi))$:
\[
 R(\xi)=R(P)+Q_P^{(\beta_P)}(\xi)+W_P(\xi),
\]
where $Q_P^{(\beta_P)}$ is $\beta_P$-homogeneous and, for some $c_0>0$, satisfies
\be\label{nd}
 |ZQ_P^{(\beta_P)}|+\rho|TQ_P^{(\beta_P)}|\ge c_0\rho^{\beta_P-1},\quad \hbox{ a.e. }
\ee
and
\[
 \rho^{-\beta_P}|W_P|+\rho^{1-\beta_P}|ZW_P|+\rho^{2-\beta_P}|TW_P|\to 0.
\]
Then $\mathcal C=\{P_1,\ldots,P_N\}$ is finite. For each fixed $2\le\gamma\le\min_j\beta_{P_j}$, $R$ satisfies $(*)'_\gamma$ on the whole sphere with $L=1$ in the uniform chart sense above. Moreover, $R_\mu=1+\mu(R-1)$ satisfies this condition with $L_\mu=\mu$ and the same auxiliary functions and structural constants in each chart for $0<\mu\le1$. These structural constants may depend on $n$, $\al$, $\gamma$, $R$ and its local flatness data, but are independent of $\mu$ and the centered chart.
\end{prop}
\begin{proof}
In a sufficiently small punctured chart at $P\in\mathcal C$, we have
\[
 |ZR|+\rho|TR|
 \ge |ZQ_P^{(\beta_P)}|+\rho|TQ_P^{(\beta_P)}|
       -|ZW_P|-\rho|TW_P|
 \ge c_P\rho^{\beta_P-1}>0.
\]
Thus $\mathcal C$ is discrete and compact, hence finite. For $\gamma=2$, \eqref{eq:lf-two} follows directly from the uniform $\mathcal A_\al $ bound. 

Let $\gamma>2$. For each $q\in\S^{2n+1}$, in the centered chart $F_q$, let
\[
 K=R\circ F_q,\quad
 s=\max\{|\nabla_\H K|^{1/(\gamma-1)},|TK|^{1/(\gamma-2)}\}.
\]
Near $\mathcal C$, the original auxiliary functions vanish at the critical points and are bounded on smaller neighborhoods. The chain rule for contact changes of coordinates transfers the local flatness condition to each centered Cayley chart, with uniform constants because the chart family is bounded. Lemmas \ref{lem:wf-radius} and \ref{lem:wf-lowering} then give $(*)_\gamma'$ near $\mathcal{C}$. On the compact complement of these neighborhoods, by the $\mathcal A_\al $ bounds, the local bounds for $Z_a(TR)$, and the absence of critical points, we have, uniformly in $q$,
\[
 c\le s\le C,\quad
 |\nabla_\H^2K|+|TK|+|\nabla_\H(TK)|\le C,\quad
 |\nabla_\H s|\le C\quad\text{a.e.}
\]
Consequently, throughout each chart, we have
\begin{gather*}
  s^{\gamma-1}\le |\nabla_\H K|+s|TK|\le2s^{\gamma-1},\\
 |\nabla_\H^2K|+|TK|\le Cs^{\gamma-2}\quad\text{a.e.},\\
 |\nabla_\H(TK)|\le Cs^{\gamma-3}\quad\text{ a.e. on }\{s>0\}.   
\end{gather*}
The branch estimates of Lemma \ref{lem:wf-radius} also give a uniform local Lipschitz bound for $s$. Compactness of the unitary group makes the chart constants uniform.

Finally, for $K_\mu=1+\mu(K-1)$, we have
\[
 DK_\mu=\mu DK,\quad
 D\in\{\nabla_\H,T,\nabla_\H^2,\nabla_\H T\},\quad L_\mu=\mu,
\]
so the same auxiliary functions in Definition~\ref{def:weighted-flatness} work for every $0<\mu\le1$.
\end{proof}

Moreover, by \eqref{vectorfields} and \eqref{right-vectorfields}, the nondegeneracy condition \eqref{nd} is equivalent to the corresponding estimate written with the right-invariant horizontal fields \eqref{right-vectorfields}.

\section{Local analysis near isolated blow-up  points}\label{sec:3}
The analysis in this section and the next adapts the blow-up  analysis developed in \cite{SZ1996,Li1995}  to give accurate blow-up  profiles for solutions of sub-elliptic equations in $\Hn $.

Let $\Om\subset \Hn $  be a domain, let $\tau_i \geq 0$ satisfy $\lim _{i \to\infty} \tau_i=0$, let $p_i=\frac{Q+2}{Q-2}-\tau_i$, and let $R_i \in\mathcal A_\al (\Om)$, $0<\al<1$ satisfy, for some constants $A_1,A_2>0$ independent of $i$,
\be\label{4.1}
1 / A_1 \leq R_i(\xi) \leq A_1 \quad \text { for all }\, \xi \in \Om, \quad\|R_i\|_{\mathcal A_\al (\Om)} \leq A_2 .
\ee
Here $A_1$ gives the two-sided bound for $R_i$, and $A_2$ bounds its intrinsic norm. All estimates for blow-up sequences are understood for sufficiently large $i$, so that $p_i>1$; the starting index may depend on the sequence.
Let $u_i \in L^{\infty}(\Om)$ with $u_i \geq 0$ in $\Om$ satisfy
\be\label{4.2}
-\Delta_{\H} u_i=\frac{n^2}{4} R_i u_i^{p_i} \quad \text { in } \,\Om .
\ee

We say that $\{u_i\}$ blows up if $\|u_i\|_{L^{\infty}(\Om )} \to\infty$ as $i \to\infty$.

\begin{defn}\label{def:isolated-blowup}
 Suppose that $\{R_i\}$ satisfies \eqref{4.1} and $\{u_i\}$ satisfies \eqref{4.2}. We say a point $\bar{\xi} \in \Om $ is an isolated blow-up  point of $\{u_i\}$ if there exist $\bar{r}\in (0,d(\bar{\xi}, \pa\Om ))$, $A_3>0$, and a sequence $\xi_i$ tending to $\bar{\xi}$, such that $\xi_i$ is a local maximum point of $u_i$, $u_i(\xi_i) \to  \infty$ and $$
 u_i(\xi) \leq A_3d(\xi,\xi_i)^{-2 /(p_i-1)} \quad \text { for all }\, \xi \in B_{\bar{r}}(\xi_i)\backslash\{\xi_i\}.
 $$
\end{defn}
For such an isolated blow-up sequence, $A_3$ denotes the bound in Definition~\ref{def:isolated-blowup} and $\bar r$ its fixed isolation radius. For $\theta\in\pa B_1$ and a fixed $0<S<\bar r$, define $f_{u_i,\theta}:[0,S]\to\R$ by
	\be\label{4.3}
		f_{u_i, \theta}(s) = s^{2/(p_i-1)}u_i(\xi_i\circ \delta_s \theta),
	\ee
		where  $\delta_s \theta$ is the dilation in $\Hn $.
\begin{defn}\label{def:simple-blowup}
We say $\xi_i \to  \bar{\xi} \in \Om$ is an isolated simple blow-up  point if $\xi_i \to  \bar{\xi}$ is an isolated blow-up  point such that for some radius $0<\rho<S$ $($independent of $i$ and $\theta\in \pa B_1)$, $f_{u_i,\theta}$ has precisely one critical point in $(0, \rho)$ for every  $\theta\in \pa B_1$ and  large $i$.
\end{defn}

In Definition~\ref{def:simple-blowup}, $\rho$ is a fixed radius, distinct from the norm function $\rho(\xi)$. In this section, we are mainly concerned with the blow-up profile of $\{u_i\}$. Under certain conditions, we can show that isolated blow-up points must be simple.

\begin{lem}\label{lem:4.3}
Let $r>0$, $\xi_0\in\Hn$, and let $w>0$ be a $C^2$ solution of
$-\Delta_\H w=a w$ in $B_{2r}(\xi_0)$, where $a\in L^\infty(B_{2r}(\xi_0))$. Then
\begin{equation}\label{eq:local-harnack}
\sup_{B_r(\xi_0)}w\le C\inf_{B_r(\xi_0)}w,
\end{equation}
where $C$ depends only on $n$ and $r^2\|a\|_{L^\infty(B_{2r}(\xi_0))}$.

In particular, suppose that $u_i \in C^2(\Om)$ satisfies \eqref{4.2} with $\{R_i\}$ satisfying \eqref{4.1}, and $\xi_i \to 0$ is an isolated blow-up  point of $\{u_i\}$, i.e., for some positive constants $A_3$ and $\bar{r}$ as in Definition~\ref{def:isolated-blowup}, independent of $i$,
	\be\label{4.5}
	d(\xi,\xi_i)^{2 /(p_i-1)} u_i(\xi) \leq A_3 \quad \text { for all }\, \xi \in B_{\bar{r}}(\xi_i)\setminus\{\xi_i\},\quad B_{\bar r}(\xi_i)\subset\Om.
\ee
	 Then for any $r\in (0, \bar{r}/3)$, we have the following Harnack inequality:
	\begin{equation}\label{eq:annular-harnack}
\sup_{B_{2r}(\xi_i)\backslash B_{r/2}(\xi_i)} u_i \leq C \inf_{B_{2r}(\xi_i)\backslash B_{r/2}(\xi_i)}u_i,
	\end{equation}
	where $C$ is a positive constant depending only on $n,  A_3, \bar{r}$ and $\sup _i\|R_i\|_{L^{\infty}(B_{\bar r}(\xi_i))}$.
\end{lem}
\begin{proof}
The local estimate \eqref{eq:local-harnack} follows from \cite[Theorem~1.1]{CGL1993} by left translation and dilation.
For the annular estimate, let $r\in (0,\bar r/3)$ and set
	\[
	v_i(\xi)=r^{2 /(p_i-1)} u_i(\xi_i\circ \delta_{r}\xi) \quad \text { for }\, \xi \in B_3.
	\]
		It is easy to see that
	\[
	-\Delta_{\H}v_i(\xi)=\frac{n^2}{4}R_{i}(\xi_i\circ \delta_{r}\xi)v_i(\xi)^{p_i} \quad \text { in }\, B_3.
	\]
		By \eqref{4.5},
	\[
	v_i(\xi) \leq A_3 |\xi|^{-2 /(p_i-1)} \quad \text { for all }\, \xi \in B_3 .
	\]
		The potential $(n^2/4)R_i(\xi_i\circ\delta_r\xi)v_i^{p_i-1}$ is uniformly bounded on $B_{5/2}\backslash\overline B_{1/4}$ by \eqref{4.5}. Applying \eqref{eq:local-harnack} on a fixed finite family of overlapping balls in this annulus gives \eqref{eq:annular-harnack}.
\end{proof}

\begin{prop}\label{prop:4.4}
 Suppose that $u_i \in C^2(\Om)$ with $u_i \geq 0$ in $\Om$ satisfies \eqref{4.2} with $R_i \in\mathcal A_\al (\Om)$ satisfying \eqref{4.1}. Suppose also that $\xi_i \to0$ is an isolated blow-up  point of $\{u_i\}$ with \eqref{4.5}.
Then for any $S_i \to\infty$ and $\var_i \to0^{+}$, we have, after passing to a subsequence $($still denoted by $\{u_i\},\{\xi_i\}$, etc.$)$,
\begin{gather*}
\|m_i^{-1} u_i(\xi_i\circ\delta_{m_i^{-(p_i-1) / 2}} \xi  )-\Lam(\delta_{k_i} \xi)\|_{C^2(B_{2 S_i}(0))} \leq \var_i, \\
S_i m_i^{-(p_i-1) / 2} \to 0 \quad \text { as }\, i \to\infty,
\end{gather*}
where $m_i=u_i(\xi_i)$, $\Lam(\xi)=((1+|z|^2)^2+t^2)^{(2-Q)/4}$ and  $k_i= \sqrt{R_i(\xi_i)}/2$.

Here, passing to a subsequence means pairing the prescribed
$(S_j,\var_j)$ with suitably chosen $u_{i_j}$ and then
relabeling $j$ as $i$, rather than taking a common subsequence
of $(u_i,S_i,\var_i)$.
\end{prop}
\begin{proof}
 Let
	\[
	\phi_i(\xi)=m_i^{-1} u_i(\xi_i\circ\delta_{m_i^{-(p_i-1) / 2}} \xi  ) \quad \text { for }\, |\xi|<\bar{r} m_i^{(p_i-1) / 2}.
	\]
		It follows that
	\begin{gather}
	-\Delta_{\H} \phi_i= \frac{n^2}{4}R_i(\xi_i\circ\delta_{m_i^{-(p_i-1) / 2}} \xi) \phi_i^{p_i}, \notag\\
	0<\phi_i(\xi) \leq A_3|\xi|^{-2 /(p_i-1)} \quad \text{ for }\,0<|\xi|<\bar{r} m_i^{(p_i-1) / 2},\label{4.6}
	\end{gather}
	and since $\xi_i$ is a local maximum point of $u_i$, we also have
\[
	\phi_i(0)=1, \quad \nabla_{\H} \phi_i(0)=0,\quad T\phi_i(0)=0.
\]

	By the weak maximum principle, for any $r\in(0,1)$, we have $1=\phi_i(0) \geq\min _{\pa B_r} \phi_i$. It follows from Lemma \ref{lem:4.3} that
	\[
	\max _{\pa  B_r} \phi_i \leq C \min _{\pa B_r} \phi_i \leq C .
	\]
		Thus,
	\[
	\max _{\overline{B}_1} \phi_i \leq C
	\]
	for some $C>0$ depending on $n,A_1,A_2,A_3$. The bound on $\overline B_1$, together with \eqref{4.6}, implies that for any $\lam_0>1$,
	\[
	\max _{\overline{B}_{\lam_0}} \phi_i \leq C(\lam_0)
	\]
    for some $C(\lam_0)>0$ depending on $n,A_1,A_2,A_3$ and $\lam_0$.
    The rescaled functions $R_i(\xi_i\circ\delta_{m_i^{-(p_i-1)/2}}\xi)$ are uniformly bounded in
$\mathcal A_\al $ on compact sets by
Lemma~\ref{lem:intrinsic-coefficient-calculus}. The local Harnack estimate \eqref{eq:local-harnack} in Lemma~\ref{lem:4.3}, together with $\phi_i(0)=1$, gives a uniform positive lower bound on each compact set. Hence
Lemma~\ref{lem:intrinsic-regularity} gives local uniform ordinary
$C^{2,\al '/2}$ bounds for any $0<\al '<\al $.
Estimate \eqref{sch new} gives the $C^2$ compactness below even when the original
function $R_i$ has no bounded second central derivative at the blow-up point.
    Thus, after passing to a subsequence we have, for some nonnegative function  $\phi \in C^2(\Hn )$,
	\[
	\begin{cases}\phi_i \rightharpoonup \phi & \text { weakly in }\, S_{1,loc}^2(\Hn), \\ \phi_i \to\phi & \text { in }\, C_{loc}^2(\Hn ) .\end{cases}
	\]
	 After passing to a further subsequence, let $R_i(\xi_i)\to R_0>0$ and $k_i:=\sqrt{R_i(\xi_i)}/2\to k_0$. It follows that
	\[
 \phi(0)=1, \quad \nabla_{\H} \phi(0)=0, \quad T\phi(0)=0
	\]
	and $\phi$ satisfies
	\[
	-\Delta_{\H}\phi=\frac{n^2}{4}R_0\phi^{(Q+2)/(Q-2)}\quad \text{ in }\, \Hn .
\]

By the Liouville theorem of \cite[Theorem 1.1]{CLMR2023} for $n=1$, and of \cite[Corollary 1.3]{LLM2026} or  \cite[Corollary 1.2]{L2026} for $n\geq 2$, we have
	\[
	\phi(\xi)=((1+k_0^2|z|^2)^2+k_0^4t^2)^{(2-Q)/4}=\Lam(\delta_{k_0}\xi).
	\]
 A diagonal argument, together with $k_i\to k_0$, gives the stated estimate on $B_{2S_i}$ and completes the proof.
\end{proof}

\begin{rem}\label{rem1}
Since we may pass to a subsequence throughout the proof, we first fix
$S_i\to\infty$ and then choose $\var_i\to0^+$ sufficiently small,
depending on $S_i$, with $\var_i S_i^{Q-2}\to0$.
We then select a subsequence as in Proposition~\ref{prop:4.4}.
With $\var_i$ chosen sufficiently small, $u_i$ has a unique critical
point in $B_{S_i m_i^{-(p_i-1)/2}}(\xi_i)$, and
$f_{u_i,\theta}$ has a unique critical point in
$(0,S_i m_i^{-(p_i-1)/2})$ for every $\theta\in\pa B_1$,
for all sufficiently large $i$.
\end{rem}

For the local estimates below, write $m_i=u_i(\xi_i)$ and $s_i=S_i m_i^{-(p_i-1)/2}$. After a fixed translation and dilation, retaining $A_1,A_2$ for the resulting bounds on $R_i$, we may assume $B_3(0) \subset \Om $ and
$$
u_i(\xi) \leq A_3 d(\xi, \xi_i)^{-2 /(p_i-1) } \quad \text { in } \, B_3(0) \backslash\{\xi_i\} .
$$
\begin{prop}\label{prop:4.5}
Under the hypotheses of Proposition \ref{prop:4.4}, there exists some positive constant $C=C(n,A_1,A_2,A_3)$ such that
	\[
	u_i(\xi) \geq C^{-1} m_i\Lam(\delta_{k_im_i^{(p_i-1)/2}} (\xi_i^{-1}\circ\xi))
	\]
    for all $d(\xi,\xi_i) \leq 1$.
	In particular, for any $e \in \Hn $ with $|e|=1$, we have
	\[
	u_i(\xi_i\circ e) \geq C^{-1} m_i^{-1+((Q-2)/2)\tau_i},
	\]
	where $\tau_i=\frac{Q+2}{Q-2}-p_i$.
\end{prop}
\begin{proof}
Denote $s_i=S_i m_i^{-(p_i-1) / 2}$. It follows from Proposition \ref{prop:4.4} that $s_i \to0$ and
\[
u_i(\xi) \geq C^{-1} m_i S_i^{2-Q} \quad \text{ for all }\,  d(\xi,\xi_i)=s_i.
\]
	Set
	\[
	\varphi_i(\xi)=C^{-1} S_i^{2-Q} s_i^{Q-2} m_i(d(\xi,\xi_i)^{2-Q}-(3 / 2)^{2-Q}) \quad \text{ for all }\,s_i \leq  d(\xi,\xi_i) \leq 3 / 2 .
	\]
		Clearly, $\varphi_i$ satisfies
\begin{align*}
-\Delta_{\H} u_i\geq &\, 0=-\Delta_{\H} \varphi_i  \quad \text{ for all }\, s_i \leq d(\xi,\xi_i) \leq 3 / 2, \\
u_i(\xi) \geq &\, \varphi_i(\xi) \quad  \text { on }\, \pa B_{s_i}(\xi_i )\cup \pa B_{3/2}(\xi_i ).
\end{align*}
		Bony's maximum principle \cite[Corollary~3.1]{B1969} gives
	\[
u_i(\xi) \geq \varphi_i(\xi) \quad \text { for all }\, s_i \leq d(\xi,\xi_i) \leq 3 / 2 .
	\]
		Therefore, Proposition \ref{prop:4.5} follows immediately from Proposition \ref{prop:4.4}.
\end{proof}
\begin{lem}\label{lem:4.6}
 Under the hypotheses of Proposition \ref{prop:4.4}, assume in addition that $\xi_i \to 0$ is also an isolated simple blow-up  point with constant $\rho$. Then, for any fixed $0<\delta\ll (Q-2)/2$, we have
	\[
	u_i(\xi) \leq C_1 u_i(\xi_i)^{-\lam_i} d(\xi,\xi_i) ^{2-Q+\delta} \quad \text { for all } \, s_i \leq d(\xi,\xi_i)  \leq 1,
	\]
	where $\lam_i=(Q-2-\delta)(p_i-1) / 2-1$ and $C_1$ is some positive constant depending only on $n,A_1,A_3$, $\delta$ and $\rho$.
\end{lem}
\begin{proof}
 From Proposition \ref{prop:4.4}, we see that
	\be\label{4.7}
	u_i(\xi) \leq C u_i(\xi_i) S_i^{2-Q} \quad \text { for all }\, d(\xi,\xi_i) =s_i,
	\ee
where	$s_i=S_i m_i^{-(p_i-1) / 2}$ as before.
Let $f_{u_i,\theta}$ be defined in \eqref{4.3}. Definition~\ref{def:simple-blowup} and Remark~\ref{rem1} imply that there exists $\rho>0$ such that for every $\theta\in \pa B_1$, $f_{u_i,\theta}$ is strictly decreasing for $s_i<r<\rho$. Then it follows from \eqref{4.7} that for $r:=|\xi_i^{-1}\circ \xi|$ and $\theta:=\delta_{r^{-1}}(\xi_i^{-1}\circ \xi)$
\begin{align*}
r^{2 /(p_i-1)} u_i(\xi) =&\,f_{u_i, \theta}(r) \leq f_{u_i, \theta}(s_i)=s_i^{2 /(p_i-1)} u_i(\xi_i \circ \delta_{s_i} \theta) \\  \leq&\, C s_i^{2 /(p_i-1)}u_i(\xi_i)S_i^{2-Q} = C S_i^{-(Q-2)/2+o(1)} ,
\end{align*}
	where $o(1)$ denotes some quantity tending to 0 as $i \to\infty$. Applying Lemma \ref{lem:4.3} again, we obtain
\be\label{4.9}
u_i(\xi)^{p_i-1} \leq O(S_i^{-2+o(1)}) d(\xi,\xi_i)^{-2} \quad \text { for all }\, s_i \leq d(\xi,\xi_i) \leq \rho .
\ee
	Consider the operator
	\[
\mathscr{L}_i(\phi)=\Delta_{\H}\phi +\frac{n^2}{4} R_i u_i^{p_i-1} \phi  \quad \text { in }\,  B_2.
	\]
		Clearly, $u_i>0$ satisfies $\mathscr{L}_i(u_i)=0$ in $B_2$.

        Now, we look for a supersolution $\varphi_i$ of the operator $\mathscr{L}_i$ such that $\varphi_i \geq u_i$ on the boundary of the annulus $\mathscr{A}_i$, where $\mathscr{A}_i:=\{\xi \in \Hn : s_i<d(\xi, \xi_i)<\rho_\delta\}\subset B_2 $. Consider the function
       \[
        \Phi(\xi):=d(\xi, \xi_i)^{-\mu}-\var _0|z-z_i|^2 d(\xi, \xi_i)^{-\mu-2}
        \]
        for $\xi=(z, t) \in \mathscr{A}_i$, where $\xi_i=(z_i, t_i)$ and $\var _0$ is a small constant to be chosen later. By direct calculations and using \eqref{4.9}, we have
        \begin{align*}
            \mathscr{L}_i(\Phi)
\leq & \,-\frac{\mu(Q-2-\mu)}{4}|z-z_i|^2 d(\xi, \xi_i)^{-\mu-4} \\
&\, - \var _0 n d(\xi, \xi_i)^{-\mu-2}+\frac{1}{4}\var _0(\mu+2)(Q-\mu)|z-z_i|^4 d(\xi, \xi_i)^{-\mu-6} \\
&\, +O(S_i^{-2+o(1)} d(\xi, \xi_i)^{-2})(d(\xi, \xi_i)^{-\mu}-\var _0|z-z_i|^2 d(\xi, \xi_i)^{-\mu-2}).
        \end{align*}

Fix $0<\delta\ll1$. For
$\mu\in\{\delta,Q-2-\delta\}$, choose $\var_0=c_0\delta>0$,
with $c_0$ depending only on $n$ and sufficiently small that
$\var_0<1$ and
$\var_0(\mu+2)(Q-\mu)\leq\mu(Q-2-\mu)$.
Since $|z-z_i|\leq d(\xi,\xi_i)$, the preceding estimate gives
\[
\mathscr L_i(\Phi)
\leq\bigl(-n\var_0+CS_i^{-2+o(1)}\bigr)
 d(\xi,\xi_i)^{-\mu-2}\leq0
\]
for all sufficiently large $i$, including the points where $z=z_i$.
Thus, for a fixed $0<\rho_\delta\leq\rho$, we have
\be\label{4.14 in Com}
    \mathscr{L}_i(d(\xi, \xi_i)^{-\delta}-\var_0|z-z_i|^2 d(\xi, \xi_i)^{-\delta-2}) \leq 0
\ee
for $s_i \leq d(\xi, \xi_i) \leq \rho_\delta$. Similarly,
\be\label{4.15 in Com}
\mathscr{L}_i(d(\xi, \xi_i)^{2-Q+\delta}-\var_0|z-z_i|^2 d(\xi, \xi_i)^{-Q+\delta}) \leq 0
\ee
for $s_i \leq d(\xi, \xi_i) \leq \rho_\delta$.

Now set $\mathfrak{M}_i=\max _{\pa  B_{\rho_\delta}(\xi_i)} u_i$, $\lam_i=((Q-2-\delta)(p_i-1)) / 2-1$, and
\begin{align*}
\varphi_i(\xi)= &\, \frac{\mathfrak{M}_i}{1-\var_0} \rho_\delta^\delta\Big[ d(\xi, \xi_i)^{-\delta}-\var_0 \frac{|z-z_i|^2}{d(\xi, \xi_i)^2} d(\xi, \xi_i)^{-\delta}\Big] \\
& \,+B u_i(\xi_i)^{-\lam_i}\Big[ d(\xi, \xi_i)^{2-Q+\delta}-\var_0 \frac{|z-z_i|^2}{d(\xi, \xi_i)^2} d(\xi, \xi_i)^{2-Q+\delta}\Big] \quad \text { in }\, \mathscr{A}_i,
\end{align*}
where $B>1$ will be chosen later. By \eqref{4.14 in Com} and \eqref{4.15 in Com}, we see that $\varphi_i$ is a supersolution of $\mathscr{L}_i$ in $\mathscr{A}_i$. Furthermore,
\be\label{4.16 in Com}
\varphi_i(\xi) \geq \mathfrak{M}_i \geq u_i(\xi) \quad \text { for }\, d(\xi, \xi_i)=\rho_\delta .
\ee
Also,
\begin{align}
\varphi_i(\xi)  \geq&\,  B(1-\var_0) u_i(\xi_i)^{-\lam _i} d(\xi, \xi_i)^{2-Q+\delta} \notag\\
 \geq&\,  B(1-\var_0) u_i(\xi_i) S_i^{2-Q} \quad \text { for }\, d(\xi, \xi_i)=s_i \label{4.17 in Com}.
\end{align}
Comparing \eqref{4.17 in Com} with \eqref{4.7}, we choose $B$ large enough that $B(1-\var_0) \geq C$, where $C>0$ is the constant occurring in equation \eqref{4.7}. With this choice of $B$, we have
\[
\varphi_i(\xi) \geq u_i(\xi) \quad \text { for } \,d(\xi, \xi_i)=s_i .
\]

Set $w_i=\varphi_i/u_i$. Since $\mathscr L_i u_i=0$ and $\mathscr L_i\varphi_i\le0$, we have
\[
\Delta_{\H }w_i +2\langle\nabla_{\H }\ln u_i,\nabla_{\H }w_i\rangle =u_i^{-1}\mathscr L_i\varphi_i\le0
\quad\text{ in }\, \mathscr A_i.
\]
Moreover, from \eqref{4.16 in Com} and \eqref{4.17 in Com}, $w_i\ge1$ on $\pa \mathscr A_i$. 
Applying the maximum principle to the equation for $w_i$ (see \cite[Section~3]{B1969}), we obtain
\be\label{uphi}
u_i(\xi) \leq \varphi_i(\xi) \quad \text { for all }\, s_i \leq d(\xi, \xi_i) \leq \rho_\delta .
\ee 

Using Lemma \ref{lem:4.3} and \eqref{uphi}, we get, for any $\theta \in \pa  B_1$ and $s_i \leq s \leq \rho_\delta$,
\begin{align*}
\rho_\delta^{2 /(p_i-1) } \mathfrak{M}_i & \leq C \rho_\delta^{2 /(p_i-1) } \min _{d(\xi, \xi_i)=\rho_\delta} u_i(\xi) \\
& \leq C f_{u_i, \theta}(\rho_\delta) \\
& \leq C s^{2 /(p_i-1) } u_i(\xi_i \circ \delta_s \theta) \\
& \leq C s^{2 /(p_i-1) }(\mathfrak{M}_i \rho_\delta^{\delta} s^{-\delta}+B u_i(\xi_i)^{-\lam _i} s^{2-Q+\delta}) .
\end{align*}

		Choose $s$ sufficiently small, independently of $\theta$, such that
\[
	C s^{2 /(p_i-1)} \rho_\delta^{\delta} s^{-\delta} \leq \frac{1-\var_0}{2} \rho_\delta^{2 /(p_i-1)} .
\]
	It follows that $\mathfrak{M}_i \leq C u_i(\xi_i)^{-\lam_i}$. Then Lemma \ref{lem:4.6} follows from the Harnack estimate \eqref{eq:annular-harnack}.
\end{proof}
To sharpen the decay estimate in Lemma~\ref{lem:4.6}, we first estimate $\tau_i$ using \eqref{Pohozaev}.
\begin{lem}\label{lem:4.8}
	 Under the assumptions of Lemma \ref{lem:4.6}, we have
	\[
     \tau_i=O(u_i(\xi_i)^{-2/(Q-2)+o(1)}),\quad  \text{ and thus  } \quad u_i(\xi_i)^{\tau_i}=1+o(1).
     \]
The implicit constant depends only on $n,\al,A_1,A_2,A_3$ and $\rho$.
\end{lem}
\begin{proof}
 By the Pohozaev identity \eqref{Pohozaev} in Lemma \ref{prop:4.7}, we have
\begin{align*}
&\,\int_{\pa {B_1}}B(1 , \xi, u_i(\xi_i\circ \xi), \nabla_\xi[u_i(\xi_i\circ\xi)])\\
	 =&\,\Big( \frac{Q}{p_i+1} -\frac{Q-2}{2}\Big)\frac{n^2}{4} \int_{B_1 }R_i(\xi_i\circ\xi)u_i(\xi_i\circ\xi)^{p_i+1}+ \frac{n^2}{4(p_i +1)}\int_{B_1} \Xi (R_i(\xi_i\circ\xi)) u_i(\xi_i\circ\xi)^{p_i+1}\\
     &\,-\frac{n^2}{4(p_i+1)}\int_{\pa  B_1 } R_i(\xi_i\circ\xi) u_i(\xi_i\circ\xi)^{p_i+1} \Xi  \cdot \nu.
\end{align*}
		By direct computations, we conclude that
	\begin{align}
	\tau_i \int_{B_1} u_i(\xi_i\circ\xi)^{p_i+1} \leq&\, C(n,A_1,A_2)\Big\{\int_{ B_1} |\xi|  u_i(\xi_i\circ\xi)^{p_i+1}+\int_{ \pa  B_1} u_i(\xi_i\circ\xi)^{p_i+1} \Xi\cdot \nu\nonumber\\
    &\,+\Big|\int_{\pa {B_1}}B(1 , \xi, u_i(\xi_i\circ \xi), \nabla_\xi[u_i(\xi_i\circ\xi)])\Big|\Big\},\label{4.14}
	\end{align}
    where we also used the fact that
\[
|\Xi (R_i(\xi_i\circ\xi))| \leq C d( \xi_i\circ\xi, \xi_i) |\nabla_{\H } R_i|+C d( \xi_i\circ\xi, \xi_i) ^2|T R_i| \leq C d(\xi_i\circ\xi, \xi_i)=C|\xi|
\]
for $d(\xi_i\circ\xi, \xi_i) \leq 1$.

It follows from Proposition \ref{prop:4.5} and a change of variables that
	\begin{align}
\int_{B_1} u_i(\xi_i\circ\xi)^{p_i+1} \geq& \,C^{-1} \int_{B_{s_i}} \frac{m_i^{p_i+1}}{((1+k_i^2m_i^{p_i-1}|z|^2)^2+k_i^4m_i^{2(p_i-1)}|t|^2)^{(Q-2)(p_i+1)/4}}\, \d z\, \d t \nonumber\\
	\geq&\, C^{-1} m_i^{\tau_i(Q/2-1)} \int_{B_{S_i}} \frac{1}{((1+k_i^2|z|^2)^2+k_i^4|t|^2)^{(Q-2)(p_i+1)/4}}\, \d z\, \d t  \nonumber\\
	 \geq &\, C^{-1} m_i^{\tau_i(Q/2-1)}.\label{new 1}
	\end{align}
	Similarly, by Proposition \ref{prop:4.4}, Lemma \ref{lem:4.6} and a change of variables, we have
    \begin{align}
        \int_{B_1} |\xi|  u_i(\xi_i\circ\xi)^{p_i+1}=&\,\int_{B_1\backslash B_{s_i}} |\xi|  u_i(\xi_i\circ\xi)^{p_i+1}+\int_{B_{s_i}} |\xi|  u_i(\xi_i\circ\xi)^{p_i+1}\nonumber\\
        \leq &\,Cm_i^{-2/(Q-2)+o(1)}.\label{new 2}
    \end{align}

Fix
\[
0<\delta<\min\Big\{\frac14,\frac{Q-2}{8}\Big\}.
\]
By Lemmas \ref{lem:4.3} and \ref{lem:4.6} and the interior estimate \eqref{sch new}, applied after normalization, the last two boundary integrals on the right-hand side of \eqref{4.14} are bounded by
\[
C_\delta u_i(\xi_i)^{-2+\frac{4\delta}{Q-2}+o(1)}.
\]
It follows from \eqref{4.14}, \eqref{new 1} and \eqref{new 2} that
\[
\tau_i \le C_\delta u_i(\xi_i)^{-\frac{2}{Q-2}+o(1)} + C_\delta u_i(\xi_i)^{-2+\frac{4\delta}{Q-2}+o(1)}.
\]
By the choice of $\delta$ and the fact that $Q\ge4$, we have
$
-2+\frac{4\delta}{Q-2}<-\frac{2}{Q-2}.
$
Consequently,
\[
\tau_i=O(u_i(\xi_i)^{-\frac{2}{Q-2}+o(1)}).
\]
In particular, $\tau_i\ln  u_i(\xi_i)\to0$, and hence $u_i(\xi_i)^{\tau_i}=1+o(1)$.
\end{proof}

\begin{prop}\label{prop:4.9}
	 Under the assumptions of Lemma \ref{lem:4.6}, we have
\be\label{eq:sharp-isolated-bound}
	u_i(\xi) \leq C u_i(\xi_i)^{-1} d(\xi,\xi_i) ^{2-Q} \quad \text { for all }\, d(\xi,\xi_i)  \leq 1 .
\ee
Here $C$ depends only on $n,\al,A_1,A_2,A_3$ and $\rho$.
	Furthermore, after passing to a subsequence, we have
\[
    u_{i}(\xi_i)u_i(\xi)\to u(\xi)=a|\xi|^{2-Q}+h(\xi)\quad \text{ in }\, C^2_{loc}(B_1\backslash \{0\}),
 \]
	where $a>0$ is a constant and $-\Delta_{\H}h=0$ in $B_1$.
\end{prop}

\begin{proof}
For $ d(\xi,\xi_i) <s_i$, it follows from Proposition \ref{prop:4.4} that
\[
u_i(\xi)  \leq C m_i\Lam(\delta_{k_im_i^{(p_i-1)/2}}(\xi_i^{-1}\circ \xi))
 \leq  C m_i^{-1+ \frac{Q-2}{2}\tau_i} d(\xi,\xi_i) ^{2-Q}
\leq  C m_i^{-1} d(\xi,\xi_i) ^{2-Q},
\]
where Lemma \ref{lem:4.8} is used in the last inequality.

Suppose $ d(\xi,\xi_i)  \geq s_i$. Let $e \in \Hn $ with $|e|=1$, and set $v_i(\xi)=u_i(\xi_i\circ e)^{-1} u_i(\xi)$. Then $v_i$ satisfies
\[
-\Delta_{\H} v_i=\frac{n^2}{4}R_i u_i(\xi_i\circ e)^{p_i-1} v_i^{p_i} \quad  \text { in }\, B_2.
\]
The Harnack estimate \eqref{eq:annular-harnack}, the decay bound in Lemma~\ref{lem:4.6}, and interior regularity \cite[Theorem~6.1]{F1975} imply that, after passing to a subsequence, $\{v_i\}$ converges in $C_{loc }^2( B_2 \backslash\{0\}) $ to a positive function $v$, which satisfies
\[
\Delta_{\H } v=0 \quad \text { in }\, B_2 \backslash\{0\} .
\]


We claim that the singularity of $v$ at the origin is not removable. For every $\theta\in\pa  B_1$, set
\[
f_{v_i,\theta}(r) :=r^{\frac{2}{p_i-1}}v_i(\xi_i\circ\delta_r\theta) =\frac{f_{u_i,\theta}(r)}{u_i(\xi_i\circ e)}.
\]
Since $\xi_i\to0$, $p_i\to (Q+2)/(Q-2)$, and $v_i\to v$ in $C^2_{loc}(B_2\setminus\{0\})$, we have
\[
f_{v_i,\theta}(r)\to  f_{v,\theta}(r):=r^{\frac{Q-2}{2}}v(\delta_r\theta)\quad   \text{ as } \,i\to  \infty
\]
locally uniformly for $r\in(0,\rho)$, after shrinking $\rho$ if necessary.
By Definition~\ref{def:simple-blowup} and Remark~\ref{rem1}, $f_{u_i,\theta}$ is decreasing on $(s_i,\rho)$. The same holds for $f_{v_i,\theta}$. As $s_i\to0$, passing to the limit shows that $f_{v,\theta}$ is non-increasing on $(0,\rho)$.

If the singularity were removable, $v$ would be bounded near the origin, and hence
\[
\lim_{r\to 0}f_{v,\theta}(r)=0.
\]
For any fixed $r_0\in(0,\rho)$, by monotonicity, we have
\[
0<f_{v,\theta}(r_0)\le f_{v,\theta}(r) \quad \text{ for }\, 0<r<r_0.
\]
 We obtain a contradiction after letting $r\to 0$.

Therefore, $v$ has to have a singularity at $\xi=0$. By the B\^ocher-type theorem for the Heisenberg group (see \cite[Theorem 1.9]{BL2007}), we can write
\[
v(\xi)=a_v|\xi|^{2-Q}+h_v(\xi),
\]
where $a_v>0$ is a constant and
\[
\Delta_{\H}h_v=0 \quad \text{ in }\, B_2.
\]

We first establish the inequality in Proposition \ref{prop:4.9} for $ d(\xi,\xi_i)=1$. Namely, we prove that
\be\label{4.23}
u_i(\xi_i\circ e) \leq C u_i^{-1}(\xi_i) .
\ee
Suppose that \eqref{4.23} does not hold. Then along a subsequence we have
\be\label{3.19 in}
\lim _{i \to\infty} u_i(\xi_i\circ e) u_i(\xi_i)=\infty .
\ee
We multiply \eqref{4.2} by $u_i( \xi_i \circ e) ^{-1}$ and integrate over $B_1$ to obtain
\be \label{3.20 in}
\frac{n^2}{4} u_i( \xi_i \circ e) ^{-1} \int_{B_1} R_i u_i^{p_i}=-\int_{B_1} u_i( \xi_i \circ e) ^{-1} \Delta_{\H } u_i=-\frac{1}{4}\int_{\pa  B_1} A \nabla v_i \cdot \nu.
\ee
By direct computations, we have
\begin{align}
\lim _{i \to  \infty} \int_{\pa  B_1} A \nabla v_i \cdot \nu  =&\, \int_{\pa  B_1} A \nabla( a_v|\xi|^{2-Q})  \cdot \nu\nonumber \\
 =&\,(2-Q) a_v \int_{\pa  B_1} A \nabla(|\xi|) \cdot \frac{\nabla(|\xi|)}{|\nabla(|\xi|)|}\nonumber \\
 =&\,2(2-Q) a_v \int_{\pa  B_1} \frac{|z|^2}{(4|z|^6+t^2)^{1 / 2}}<0 .\label{3.21 in}
\end{align}
By Lemma \ref{lem:4.8} and computations similar to those in \eqref{new 2}, we see that
\[
\int_{B_1} R_i u_i^{p_i} \leq C u_i(\xi_i)^{-1}.
\]
By \eqref{3.19 in},
\[
\lim _{i \to\infty} u_i(\xi_i\circ e)^{-1} \int_{B_1} R_i u_i^{p_i}=0,
\]
which contradicts \eqref{3.20 in} and \eqref{3.21 in}. Thus we have proved \eqref{4.23}.

 By Lemma \ref{lem:4.3}, we have established the inequality in Proposition \ref{prop:4.9} for $\rho \leq d(\xi,\xi_i) \leq 1$. It remains to prove Proposition \ref{prop:4.9} for $s_i \leq d(\xi,\xi_i) \leq \rho$.
Indeed, by using a  scaling argument, we can reduce the case of $s_i \leq d(\xi,\xi_i)<\rho$ to $ d(\xi,\xi_i)=1$.

Suppose, to the contrary, that along a subsequence there are points $\widetilde{\xi}_{i}$ satisfying $s_{i} \leq d(\widetilde{\xi}_i,\xi_i) \leq \rho$ and \be
\label{prop4.9-1}
\lim_{i \to  \infty} u_{i}(\widetilde{\xi}_{i})u_{i}(\xi_{i})d(\widetilde{\xi}_i,\xi_i)^{Q-2}=\infty.
\ee
Set $\widetilde{s}_{i}=d(\widetilde{\xi}_i,\xi_i)$, $\widetilde{u}_{i}(\xi)=(\widetilde{s}_{i})^{2 /(p_{i}-1)} u_{i}(\xi_{i}\circ\delta_{\widetilde{s}_{i}}\xi)$. Then $\{\widetilde{u}_{i}\}$ satisfies
\[
-\Delta_{\H} \widetilde{u}_{i}=\frac{n^2}{4}R_{i}(\xi_{i}\circ\delta_{\widetilde{s}_{i}} \xi) \widetilde{u}_{i}(\xi)^{p_{i}}, \quad \xi \in B_{2},
\]
and all the hypotheses of Proposition \ref{prop:4.9} with 0 being an isolated simple
blow-up  point. It follows from \eqref{4.23} that
\[
\widetilde{u}_{i}(0) \widetilde{u}_{i}(\delta_{\widetilde{s}_{i}^{-1}}(\xi_{i}^{-1}\circ \widetilde{\xi}_{i} )) \leq C.
\]

Returning to the original variables, we obtain
\[
 u_i(\widetilde\xi_i) u_i(\xi_i)  d(\widetilde{\xi}_i,\xi_i)^{\,4/(p_i-1)}\le C.
\]
Set $ e_i=\frac{4}{p_i-1}-(Q-2)=\frac{(Q-2)\tau_i}{p_i-1}\ge0$.
Since $d(\widetilde{\xi}_i,\xi_i)\ge S_i m_i^{-(p_i-1)/2}$ and $S_i\ge1$ for large $i$, by Lemma \ref{lem:4.8} we have
\[
d(\widetilde{\xi}_i,\xi_i)^{-e_i} \le m_i^{(Q-2)\tau_i/2}\to  1.
\]
Therefore,
\[
\limsup _{i \to  \infty} u_{i}(\widetilde{\xi}_{i}) u_{i}(\xi_{i})d(\widetilde{\xi}_i,\xi_i)^{Q-2}<\infty,
\]
which contradicts \eqref{prop4.9-1}. We complete the proof of Proposition \ref{prop:4.9}.

The boundary integration argument in the proof of \cite[Proposition~2.3]{Li1995} gives the constant $a$ in Proposition~\ref{prop:4.9}:
\[
a=2^{Q-2}\lim_{i\to  \infty} (R_i(0))^{(2-Q)/2}.
\] 
We omit the details here.
\end{proof}

The bubble approximation in Proposition~\ref{prop:4.4} and the decay estimate \eqref{eq:sharp-isolated-bound} give the following integral bounds.
\begin{lem}\label{lem:4.11}
Under the assumptions of Lemma~\ref{lem:4.6}, with $m_i=u_i(\xi_i)$ and $s_i=S_i m_i^{-(p_i-1)/2}$ as above, the following estimates hold:
\[
\int_{ d(\xi,\xi_i)  \leq s_i} d(\xi,\xi_i) ^s u_i(\xi)^{p_i+1}= \left\{\begin{aligned}
&O(u_i(\xi_i)^{-2s/(Q-2)}) &&\text{ if }\, -Q<s<Q, \\ &O(u_i(\xi_i)^{-Q^* } \ln  u_i(\xi_i)) & & \text{ if }\, s=Q, \\ &o(u_i(\xi_i)^{-Q^* }) &&\text{ if }\, s>Q,\end{aligned}\right.
\]
and
\[
\int_{s_i< d(\xi,\xi_i)  \leq 1} d(\xi,\xi_i) ^s u_i(\xi)^{p_i+1}=  \left\{\begin{aligned}   &o(u_i(\xi_i)^{-2s/(Q-2)}) && \text{ if }\, -Q<s<Q, \\ &O(u_i(\xi_i)^{-Q^* } \ln  u_i(\xi_i)) & &\text{ if }\,  s=Q, \\ &O(u_i(\xi_i)^{-Q^* }) && \text{ if }\, s>Q ,\end{aligned}\right.
\]
where the constants implicit in $O(\cdot)$ depend only on $n,\al,A_1,A_2,A_3,\rho$ and $s$.
\end{lem}
\begin{proof}
The first set of estimates follows from Proposition \ref{prop:4.4} and Lemma \ref{lem:4.8}, and the second one from Proposition \ref{prop:4.9} and Lemma \ref{lem:4.8}.
\end{proof}

For later applications, we replace $R_i$ by $R_i(\xi) H_i(\xi)^{\tau_i}$ in \eqref{4.2} and consider
\be\label{4.26}
-\Delta_{\H} u_i(\xi)=\frac{n^2}{4} R_i(\xi) H_i(\xi)^{\tau_i} u_i(\xi)^{p_i} \quad \text { in }\, B_2,
\ee
where $H_i \in C^{1,1}(B_2)$ satisfies
\be\label{4.27}
A_4^{-1} \leq H_i(\xi) \leq A_4, \quad \| H_i\|_{C^{1,1}(B_2)} \leq A_5 , 
\ee
for all $\xi\in B_2$, with $A_4,A_5>0$ independent of $i$. Here $A_4$ bounds $H_i$ and $H_i^{-1}$, whereas $A_5$ bounds its $C^{1,1}$ norm.
By Lemma~\ref{lem:intrinsic-coefficient-calculus}, the preceding local estimates apply to $R_iH_i^{\tau_i}$ with two-sided and intrinsic norm bounds depending only on $n,\al,A_1,A_2,A_4,A_5$, for large $i$. Flatness is imposed on $R_i$ itself; derivatives of $H_i^{\tau_i}$ are estimated separately. 

\begin{lem}\label{lem:4.12}
 Suppose that $\{R_i\}$ satisfies \eqref{4.1} with constants $A_1,A_2$ and condition $(*)'_\beta$ for $2\le\beta<Q$, with $L_i=L(\beta,i)>0$ and a structural constant $C_*$ independent of $i$. Assume that $\{H_i\}$ satisfies \eqref{4.27} with constants $A_4,A_5$. Let $u_i\in C^2(B_2)$, $u_i\ge0$, solve \eqref{4.26}. If $\xi_i\to0$ is an isolated simple blow-up point with bound $A_3$ in \eqref{4.5} and radius $\rho$ in Definition~\ref{def:simple-blowup}, then
\[
\tau_i \leq  C u_i(\xi_i)^{-2}+C|\nabla_\H R_i(\xi_i)| u_i(\xi_i)^{-2/(Q-2)} +C|T R_i(\xi_i)| u_i(\xi_i)^{-4/(Q-2)}+CL(\beta, i) u_i(\xi_i)^{- 2\beta/(Q-2)},
\]
where $C>0$ depends only on $n,\al,C_*,A_1,A_2,A_3,A_4,A_5,\beta$ and $\rho$, and is independent of $i$ and $L(\beta,i)$.
\end{lem}
\begin{proof}
Given $\xi_i=(\hat{x}, \hat{y}, \hat{t})$, we define the following operator
\be\label{Xi}
\Xi_i=\sum_{j=1}^n((x-\hat{x})_j X_j+(y-\hat{y})_j Y_j)+2(t-\hat{t}+2(\hat{x} \cdot y-\hat{y} \cdot x)) T.
\ee

If $\beta=2$, on a fixed sufficiently small ball, by
\eqref{eq:lf-two-increment} and \eqref{Xi}, we have
\[
 |\Xi_iR_i(\xi)|\le C d(\xi,\xi_i)|\nabla_\H R_i(\xi_i)|
                  +CL(2,i) d(\xi,\xi_i)^2.
\]
It follows from Lemmas \ref{prop:4.7} and \ref{lem:4.11} that
\[
 \tau_i\le Cu_i(\xi_i)^{-2}+C|\nabla_\H R_i(\xi_i)| u_i(\xi_i)^{-2/(Q-2)}+CL(2,i)u_i(\xi_i)^{-4/(Q-2)},
\]
where the $H_i^{\tau_i}$ contribution is $C\tau_iu_i(\xi_i)^{-2/(Q-2)}$ and is absorbed
for large $i$, giving the stated estimate when $\beta=2$. 

If $\beta>2$, applying Lemma \ref{prop:4.7} to \eqref{4.26} with $\si=a$ and center $\xi_i$ for a fixed sufficiently small $a>0$, we have
    \begin{align*}
        &\,\int_{\pa B_a} B(a,\eta,u_i(\xi_i\circ\eta),\nabla_\eta[u_i(\xi_i\circ\eta)])\,\d S_\eta\\
        =&\,\Big( \frac{Q}{p_i+1} -\frac{Q-2}{2}\Big)\frac{n^2}{4} \int_{B_a(\xi_i) }R_iH_i^{\tau_i}u_i^{p_i+1}+ \frac{n^2}{4(p_i +1)}\int_{B_a(\xi_i)} \Xi_i (R_iH_i^{\tau_i}) u_i^{p_i+1}\\
     &\,-\frac{n^2}{4(p_i+1)}\int_{\pa  B_a(\xi_i)} R_iH_i^{\tau_i} u_i^{p_i+1} \Xi_i  \cdot \nu_i,
    \end{align*}
where $\d S_\eta$ is Euclidean surface measure on $\pa B_a$ and $\nu_i$ is the outward Euclidean unit normal to $\pa B_a(\xi_i)$.
It follows from Propositions \ref{prop:4.4} and \ref{prop:4.9} that
	\begin{align}
\tau_i  \leq&\, C u_i(\xi_i)^{-2}+C\Big|\int_{B_a(\xi_i)} \Xi_i (R_iH_i^{\tau_i}) u_i^{p_i+1}\Big|\notag \\
\leq& \, C u_i(\xi_i)^{-2}+C \tau_i \int_{B_a(\xi_i)} d(\xi,\xi_i)u_i^{p_i+1}+C\Big|\int_{B_a(\xi_i)}\Xi_i(R_i) H_i^{\tau_i} u_i^{p_i+1}\Big| \label{tau}.
\end{align}

By direct computations, we have
    \begin{align}
        &\,\Big|\int_{B_a(\xi_i)}\Xi_i(R_i) H_i^{\tau_i} u_i^{p_i+1}\Big|\notag\\
        \leq&\,          C \int_{B_a(\xi_i)} \bigl(d(\xi, \xi_i)|\nabla_{\H } R_i(\xi)|+d(\xi, \xi_i)^2|T R_i(\xi)|\bigr) u_i^{p_i+1}:=C(I_1+I_2)\label{I}.
    \end{align}

	Making use of Lemma \ref{lem:4.11}, we have
	\begin{align*}
|I_1| \leq&\, C|\nabla_{\H} R_i(\xi_i)| \int_{B_a(\xi_i)} d(\xi,\xi_i)  u_i^{p_i+1}+C \int_{B_a(\xi_i)} d(\xi,\xi_i) |\nabla_{\H} R_i(\xi)-\nabla_{\H} R_i(\xi_i)| u_i^{p_i+1} \\
\leq&\, C|\nabla_{\H} R_i(\xi_i)| u_i(\xi_i)^{-2 /(Q-2)}+C \int_{B_a(\xi_i)} d(\xi,\xi_i) |\nabla_{\H} R_i(\xi)-\nabla_{\H} R_i(\xi_i)| u_i^{p_i+1} .
	\end{align*}
For $d(\xi,\xi_i)<a$, Lemma~\ref{lem:wf-increments} gives
\begin{align}
&\,  |\nabla_{\H} R_i(\xi)  -\nabla_{\H} R_i(\xi_i)| \leq CL(\beta,i)d(\xi,\xi_i)(r_i(\xi_i)+d(\xi,\xi_i))^{\beta-2}.\label{star beta new}
	\end{align}
 When $r_i(\xi_i)\leq d(\xi,\xi_i)$,  we have
 \begin{align*}
     L(\beta,i)d(\xi,\xi_i)^2(r_i(\xi_i)+d(\xi,\xi_i))^{\beta-2}\leq C L(\beta,i)d(\xi,\xi_i)^\beta.
 \end{align*}
 When $r_i(\xi_i)> d(\xi,\xi_i)$, the lower bound in \eqref{eq:wf-first}, equivalently the argument of \eqref{eq:wf-split}, gives
 \begin{align*}
     L(\beta,i)d(\xi,\xi_i)^2(r_i(\xi_i)+d(\xi,\xi_i))^{\beta-2}& \leq CL(\beta,i)d(\xi,\xi_i)^2 r_i(\xi_i)^{\beta-2}\\
     &\leq C\frac{d(\xi,\xi_i)^2}{r_i(\xi_i)}|\nabla_\H R_i(\xi_i)|+Cd(\xi,\xi_i)^2|T R_i(\xi_i)|\\
     & \leq C d(\xi,\xi_i)|\nabla_\H R_i(\xi_i)|+Cd(\xi,\xi_i)^2|T R_i(\xi_i)|,
 \end{align*}
 where we have used the fact that $L(\beta,i)r_i(\xi_i)^{\beta-1}\leq C_* (|\nabla_\H R_i(\xi_i)|+r_i(\xi_i)|T R_i(\xi_i)|)$ in the second inequality. Hence we conclude from the above two estimates that
\[
     L(\beta,i)d(\xi,\xi_i)^2(r_i(\xi_i)+d(\xi,\xi_i))^{\beta-2}\leq C(d(\xi,\xi_i)|\nabla_\H R_i(\xi_i)|+d(\xi,\xi_i)^2|T R_i(\xi_i)|+L(\beta,i)d(\xi,\xi_i)^\beta).
\]
It follows from Lemma \ref{lem:4.11} and \eqref{star beta new} that
	\begin{align*}
	& \,\int_{B_a(\xi_i)} d(\xi,\xi_i) |\nabla_{\H} R_i(\xi)-\nabla_{\H} R_i(\xi_i)| u_i^{p_i+1} \\
\leq&\, C|\nabla_{\H} R_i(\xi_i)| u_i(\xi_i)^{-2/(Q-2)}+C|T R_i(\xi_i)| u_i(\xi_i)^{-4/(Q-2)}+CL(\beta, i) u_i(\xi_i)^{- 2\beta/(Q-2)} ,
	\end{align*}
    Therefore,
    \be\label{I1 new}
        |I_1|\leq C|\nabla_{\H} R_i(\xi_i)| u_i(\xi_i)^{-2/(Q-2)}+C|T R_i(\xi_i)| u_i(\xi_i)^{-4/(Q-2)}+CL(\beta, i) u_i(\xi_i)^{- 2\beta/(Q-2)} .
    \ee

    For $I_2$, condition $(*)_\beta'$ gives $|T R_i|\leq CL(\beta,i)(r_i(\xi_i)+d(\xi,\xi_i))^{\beta-2}$. It follows from Lemma \ref{lem:4.11} and \eqref{star beta new} that
    \begin{align}
    |I_2|    &\leq \ CL(\beta,i)\int_{B_a(\xi_i)}d(\xi, \xi_i)^2 (r_i(\xi_i)+d(\xi,\xi_i))^{\beta-2}u_i^{p_i+1}\notag\\
        &\leq \ C|\nabla_\H R_i(\xi_i)|u_i(\xi_i)^{-2/(Q-2)}+C|T R_i(\xi_i)| u_i(\xi_i)^{-4/(Q-2)}+CL(\beta, i) u_i(\xi_i)^{- 2\beta/(Q-2)} .\label{I2 new}
    \end{align}

	Lemma \ref{lem:4.12} follows immediately from \eqref{tau}, \eqref{I}, \eqref{I1 new}, \eqref{I2 new} and Lemma \ref{lem:4.11}.
\end{proof}
\begin{lem}\label{lem:4.13}
Under the hypotheses of Lemma \ref{lem:4.12},
	\[
|\nabla_{\H} R_i(\xi_i)|+u_i(\xi_i)^{-2/(Q-2)}|TR_i(\xi_i)| \leq C u_i(\xi_i)^{-2}+CL(\beta, i) u_i(\xi_i)^{-2(\beta-1)/(Q-2)},
	\]
	where $C>0$ depends only on $n,\al,C_*,A_1,A_2,A_3,A_4,A_5,\beta$ and $\rho$, and is independent of $i$ and $L(\beta,i)$.
\end{lem}
\begin{proof}
We first introduce the translated right-invariant vector fields centered at $\xi_i$. Writing
$\xi_i=(\hat{x}, \hat{y}, \hat{t})$,
we define, for each $i$ and $j=1,\ldots,n$,
\begin{align*}
\overline{X}_j^i & :=\overline{X}_j+4 \hat{y}_j \pa _t=X_j-4(y_j-\hat{y}_j) \pa _t, \\
\overline{Y}_j^i & :=\overline{Y}_j-4 \hat{x}_j \pa _t=Y_j+4(x_j-\hat{x}_j) \pa _t .
\end{align*}
They are divergence-free, commute with all the left-invariant horizontal vector fields, and satisfy
\[
\overline{X}_j^i|_{\xi_i}=X_j|_{\xi_i},\quad \overline{Y}_j^i|_{\xi_i}=Y_j|_{\xi_i} .
\]

First suppose $\beta=2$. For $1\le j\le 2n$, define the centered right-invariant horizontal field as follows: 
\[
\overline Z_j^{\,i}:=
\left\{\begin{aligned}
   &\overline X_j^{i} && \text{ if }\, 1\le j\le n,\\
&\overline Y_{j-n}^{i} && \text{ if }\,  n+1\le j\le 2n. 
\end{aligned}\right.
\]
By \eqref{eq:lf-two-increment}, 
$$
|\overline{Z}_j^i R_i(\xi)-\overline{Z}_j^i R_i(\xi_i)| \leq C L(2, i) d(\xi, \xi_i)
$$
Multiplying \eqref{4.26} by $\eta\overline{Z}_j^i u_i$, where $\eta$ is a fixed cut-off function equal to 1 near $\xi_i$, and using Proposition \ref{prop:4.9} and Lemma \ref{lem:4.11}, we obtain
\begin{align*}
|\nabla_\H R_i(\xi_i)| & \leq C u_i(\xi_i)^{-2}+C \tau_i+C L(2, i) \int_{B_a(\xi_i)} d(\xi, \xi_i) u_i^{p_i+1}  \\
& \leq C u_i(\xi_i)^{-2}+C \tau_i+C L(2, i) u_i(\xi_i)^{-2 /(Q-2)},
\end{align*}
where $a>0$ is fixed and sufficiently small. Substituting the estimate for $\tau_i$ from the proof of Lemma \ref{lem:4.12} yields
$$
|\nabla_\H R_i(\xi_i)| \leq  C u_i(\xi_i)^{-2}+C u_i(\xi_i)^{-2 /(Q-2)}|\nabla_\H R_i(\xi_i)| +C L(2, i) u_i(\xi_i)^{-2 /(Q-2)} .
$$
Using $|T R_i| \leq C_* L(2, i)$ from \eqref{eq:lf-two-increment}, we conclude that
$$
|\nabla_\H R_i(\xi_i)|+u_i(\xi_i)^{-2 /(Q-2)}|T R_i(\xi_i)|  \leq C u_i(\xi_i)^{-2}+C L(2, i) u_i(\xi_i)^{-2 /(Q-2)}.
$$

We now consider $\beta>2$.
   After a fixed local rescaling if necessary, choose a cut-off function $\eta \in C_c^{\infty}(B_{1 / 2})$ satisfying
\[
	\eta(\xi)=1, \quad|\xi| \leq 1 / 4, \quad \text { and } \quad \eta(\xi)=0, \quad|\xi| \geq 1 / 2 .
\]
	Define another cut-off function centered at $\xi_i$, namely,
    \[
\eta_i(\xi):=\eta(\xi_i^{-1} \circ \xi) .
\]
	Since $\xi_i \to  0$, for all sufficiently large $i$, $B_{1 / 2}(\xi_i) \subset B_1$. Moreover, $\overline{X}_j^i \eta_i$ and $\overline{Y}_j^i \eta_i$ are uniformly bounded and supported in
$B_{1 / 2}(\xi_i) \backslash \overline{B_{1 / 4}(\xi_i)}$.

	Multiplying \eqref{4.26} by $(\overline{X}^i_ju_i)\eta_i$, and integrating by parts over $B_1$, we obtain
\begin{align*}
&\,\frac{n^2}{4(p_i+1)}  \int_{B_1} u_i^{p_i+1} \eta_i \overline{X}^i_j(R_i H_i^{\tau_i}) \\
 = &\, -\int_{B_1} \eta_i \nabla_{\H } u_i \cdot \nabla_{\H }(\overline{X}^i_j u_i)-\int_{B_{1 / 2}(\xi_i) \backslash \overline{B_{1 / 4}(\xi_i)}}(\overline{X}^i_j u_i) \nabla_{\H } u_i \cdot \nabla_{\H } \eta_i \\
& \, -\frac{n^2}{4(p_i+1)} \int_{B_{1 / 2}(\xi_i) \backslash \overline{B_{1 / 4}(\xi_i)}} u_i^{p_i+1} R_i H_i^{\tau_i}(\overline{X}^i_j \eta_i) .
\end{align*}
Since $\overline{X}_j^i$ commutes with all the left-invariant horizontal vector fields,
\[
-\int_{B_1} \eta_i  \nabla_{\H } u_i \cdot \nabla_{\H }(\overline{X}_j^i u_i)  =-\frac{1}{2} \int_{B_1} \eta_i \overline{X}_j^i|\nabla_{\H } u_i|^2 =\frac{1}{2} \int_{B_{1 / 2}(\xi_i) \backslash \overline{B_{1 / 4}(\xi_i)}}|\nabla_{\H } u_i|^2 \overline{X}_j^i \eta_i .
\]
Therefore,
\begin{align*}
&\,\frac{n^2}{4(p_i+1)}  \int_{B_1} u_i^{p_i+1} \eta_i \overline{X}^i_j(R_i H_i^{\tau_i}) \\
=\, & \frac{1}{2} \int_{B_{1 / 2}(\xi_i) \backslash \overline{B_{1 / 4}(\xi_i)}}|\nabla_{\H } u_i|^2 \overline{X}_j^i \eta_i -\int_{B_{1 / 2}(\xi_i) \backslash \overline{B_{1 / 4}(\xi_i)}}(\overline{X}^i_j u_i) \nabla_{\H } u_i \cdot \nabla_{\H } \eta_i \\
& \,-\frac{n^2}{4(p_i+1)} \int_{B_{1 / 2}(\xi_i) \backslash \overline{B_{1 / 4}(\xi_i)}} u_i^{p_i+1} R_i H_i^{\tau_i}(\overline{X}^i_j \eta_i) .
\end{align*}

By Young's inequality,
\[
\Big|\int_{B_{1 / 2}(\xi_i) \backslash \overline{B_{1 / 4}(\xi_i)}}(\overline{X}_j^i u_i)  \nabla_{\H } u_i \cdot \nabla_{\H } \eta_i\Big| \leq C \int_{B_{1 / 2}(\xi_i) \backslash \overline{B_{1 / 4}(\xi_i)}}(|\nabla_{\H} u_i|^2+|\overline{X}_j^i u_i|^2) .
\]
It follows from Proposition \ref{prop:4.9} that
\be\label{4.29 in Com}
\Big|\int_{B_1} u_i^{p_i+1} \eta_i \overline{X}^i_j(R_i H_i^{\tau_i})\Big|  \leq C \int_{B_{1 / 2}(\xi_i) \backslash \overline{B_{1 / 4}(\xi_i)}}|\nabla_{\H } u_i|^2+C \int_{B_{1 / 2}(\xi_i) \backslash \overline{B_{1 / 4}(\xi_i)}}|\overline{X}_j^i u_i|^2+C u_i(\xi_i)^{-p_i-1}.
\ee

Using the fact $\overline{X}_j^i u_i=X_j u_i-4(y_j-\hat{y}_j) \pa _t u_i$,
we have
\be\label{4.295}
\int_{B_{1 / 2}(\xi_i) \backslash \overline{B_{1 / 4}(\xi_i)}}|\overline{X}_j^i u_i|^2 \leq C \int_{B_{1 / 2}(\xi_i) \backslash \overline{B_{1 / 4}(\xi_i)}}(|\nabla_{\H } u_i|^2+|\pa _t u_i|^2).
\ee
By Proposition \ref{prop:4.9} and local subelliptic estimates, we have
\be\label{4.30 in Com}
    \int_{B_{1 / 2}(\xi_i) \backslash \overline{B_{1 / 4}(\xi_i)}}|\nabla_{\H } u_i|^2 \leq C u_i( \xi_i \circ e ) ^2 \leq C u_i(\xi_i)^{-2},
\ee
where $e \in \Hn $ with $|e|=1$.
Similarly, 
\[
   \int_{B_{1 / 2}(\xi_i) \backslash \overline{B_{1 / 4}(\xi_i)}}(|\nabla_{\H }^2 u_i|^2+|\pa _t u_i|^2) \leq C u_i(\xi_i)^{-2}
\]
and
\be\label{4.31 in Com}
    \int_{B_{1 / 2}(\xi_i) \backslash \overline{B_{1 / 4}(\xi_i)}}|\overline{X}_j^i u_i|^2 \leq C u_i( \xi_i \circ e) ^2 \leq C u_i(\xi_i)^{-2} .
\ee

It follows from \eqref{4.29 in Com}, \eqref{4.295}, \eqref{4.30 in Com} and \eqref{4.31 in Com} that
\[
    \Big|\int_{B_1}u_i^{p_i+1} \eta_i \overline{X}^i_j(R_i H_i^{\tau_i})\Big| \leq C u_i(\xi_i)^{-2} .
\]
It follows from Lemma \ref{lem:4.11} that
\be\label{extra1}
    \Big|\int_{B_1}u_i^{p_i+1} \eta_iH_i^{\tau_i} \overline{X}^i_j(R_i )\Big| \leq C u_i(\xi_i)^{-2} +C\tau_i.
\ee

At the center $\xi_i$, we have $\overline{X}_j^i (R_i )(\xi_i)=X_j (R_i )(\xi_i)$.
Therefore,
\begin{align}
& \,|X_j (R_i )(\xi_i)| \int_{B_1} u_i(\xi)^{p_i+1} \eta_i(\xi)H_i^{\tau_i}(\xi) \notag \\
\leq &\,\Big|\int_{B_1}(\overline{X}_j^i(R_i ))(\xi) u_i(\xi)^{p_i+1} \eta_i(\xi)H_i^{\tau_i}(\xi) \Big|+C\int_{B_1}|(\overline{X}_j^i (R_i))(\xi)-X_j (R_i )(\xi_i)| u_i(\xi)^{p_i+1} \eta_i(\xi)\label{extra2} .
\end{align}

Using the fact $\overline{X}_j^i=X_j-4(y_j-\hat{y}_j) \pa _t$,
we have
\begin{align}
&\,\int_{B_1}|(\overline{X}_j^i (R_i ))(\xi)-X_j (R_i )(\xi_i)|u_i(\xi)^{p_i+1} \eta_i(\xi)\notag\\
\leq &\,\int_{B_1}(|X_j (R_i )(\xi)-X_j (R_i )(\xi_i)|+4|y_j-\hat{y}_j | | \pa _t (R_i )(\xi)|)u_i(\xi)^{p_i+1} \eta_i(\xi):=I_3+I_4 .\label{II}
\end{align}

By \eqref{star beta new} and Lemma \ref{lem:4.11}, we conclude that
\begin{align}
    |I_3|\leq&\, CL(\beta,i)\int_{B_1}d(\xi,\xi_i)(r_i(\xi_i)+d(\xi,\xi_i))^{\beta-2}u_i(\xi)^{p_i+1}\eta_i(\xi)\notag\\
    \leq &\, CL(\beta,i)(r_i(\xi_i)^{\beta-2}\int_{B_1}d(\xi,\xi_i)u_i(\xi)^{p_i+1}+\int_{B_1}d(\xi,\xi_i)^{\beta-1}u_i(\xi)^{p_i+1} )\notag\\
    \leq &\, CL(\beta,i)(r_i(\xi_i)^{\beta-2}u_i(\xi_i)^{-2/(Q-2)}+u_i(\xi_i)^{-2(\beta-1)/(Q-2)}).\label{I3 new}
\end{align}
Similarly, we have
\begin{align}
    |I_4|\leq CL(\beta,i)(r_i(\xi_i)^{\beta-2}u_i(\xi_i)^{-2/(Q-2)}+u_i(\xi_i)^{-2(\beta-1)/(Q-2)}).\label{I4 new}
\end{align}

Combining \eqref{extra1}, \eqref{extra2}, \eqref{II}, \eqref{I3 new} and \eqref{I4 new}, we have
\begin{align*}
   & \,|X_j (R_i )(\xi_i)| \int_{B_1} u_i(\xi)^{p_i+1} \eta_i(\xi)H^{\tau_i}_i(\xi)\\
    \leq&\, C u_i(\xi_i)^{-2} +C\tau_i+CL(\beta,i)(r_i(\xi_i)^{\beta-2}u_i(\xi_i)^{-2/(Q-2)}+u_i(\xi_i)^{-2(\beta-1)/(Q-2)}).
\end{align*}
It follows from Proposition \ref{prop:4.4} that
\[
 |X_j (R_i )(\xi_i)|\leq C u_i(\xi_i)^{-2} +C\tau_i+CL(\beta,i)(r_i(\xi_i)^{\beta-2}u_i(\xi_i)^{-2/(Q-2)}+u_i(\xi_i)^{-2(\beta-1)/(Q-2)}).
\]
Similarly,
\[
    |Y_j (R_i )(\xi_i)|\leq C u_i(\xi_i)^{-2} +C\tau_i+CL(\beta,i)(r_i(\xi_i)^{\beta-2}u_i(\xi_i)^{-2/(Q-2)}+u_i(\xi_i)^{-2(\beta-1)/(Q-2)}).
\]
Summing the estimates for $X_j(R_i)(\xi_i)$ and $Y_j(R_i)(\xi_i)$ over $j=1,\ldots,n$,  we obtain
\be\label{3.6 1}
    |\nabla_\H R_i (\xi_i)|\leq C u_i(\xi_i)^{-2} +C\tau_i+CL(\beta,i)(r_i(\xi_i)^{\beta-2}u_i(\xi_i)^{-2/(Q-2)}+u_i(\xi_i)^{-2(\beta-1)/(Q-2)}).
\ee

Take a large constant $M\geq 2$. When $r_i(\xi_i)\leq Mu_i(\xi_i)^{-2/(Q-2)}$, by $(*)'_\beta$, we have
\be\label{3.6 5}
    |\nabla_{\H} R_i(\xi_i)|+u_i(\xi_i)^{-2/(Q-2)}|TR_i(\xi_i)|\leq CL(\beta,i)u_i(\xi_i)^{-2(\beta-1)/(Q-2)}.
\ee
 When $r_i(\xi_i)> Mu_i(\xi_i)^{-2/(Q-2)}$, multiplying \eqref{4.26} by $(Tu_i)\eta_i$ and integrating by parts, we argue as above to obtain, for $2<\beta<3$,
\be\label{3.6 2}
|T R_i(\xi_i)| \leq C u_i(\xi_i)^{-2}+C \tau_i+C L(\beta, i) u_i(\xi_i)^{-2 /(Q-2)} r_i(\xi_i)^{\beta-3},
\ee
where we have also used the commutation of $T$ with the horizontal fields, together with Lemmas \ref{lem:wf-increments} and \ref{lem:4.11}. The negative power in \eqref{3.6 2} is well defined  because $r_i(\xi_i)>0$. For $\beta \geq 3$, Lemma \ref{lem:4.11} gives an additional term $C L(\beta,i) u_i(\xi_i)^{-2(\beta-2) /(Q-2)}$, which can be absorbed into $C L(\beta, i) u_i(\xi_i)^{-2 /(Q-2)} r_i(\xi_i)^{\beta-3}$ since $r_i(\xi_i)> Mu_i(\xi_i)^{-2/(Q-2)}$.

By \eqref{3.6 1} and \eqref{3.6 2}, we have
\begin{align*}
    |\nabla_\H R_i (\xi_i)|&\leq C u_i(\xi_i)^{-2} +C\tau_i+CL(\beta,i)u_i(\xi_i)^{-2/(Q-2)}r_i(\xi_i)^{\beta-2},\\
    |T R_i (\xi_i)|&\leq C u_i(\xi_i)^{-2} +C\tau_i+CL(\beta,i)u_i(\xi_i)^{-2/(Q-2)}r_i(\xi_i)^{\beta-3}.
\end{align*}
Take $\bar{r}_i(\xi_i)=r_i(\xi_i)/(1+r_i(\xi_i))$. Then $\bar{r}_i(\xi_i)\leq 1$ and $\bar{r}_i(\xi_i)\leq r_i(\xi_i)$. By direct computations, we have
\be\label{3.6 3}
    |\nabla_\H R_i (\xi_i)|+\bar{r}_i(\xi_i)|T R_i (\xi_i)|\leq C u_i(\xi_i)^{-2} +C\tau_i+CL(\beta,i)r_i(\xi_i)^{\beta-2}u_i(\xi_i)^{-2/(Q-2)}.
\ee

On the other hand, by $(*)'_\beta$, we have
\[
     |\nabla_\H R_i (\xi_i)|+\bar{r}_i(\xi_i)|T R_i (\xi_i)|\geq \frac{|\nabla_\H R_i (\xi_i)|+r_i(\xi_i)|T R_i (\xi_i)|}{1+r_i(\xi_i)} \geq C^{-1}L(\beta,i)\frac{r_i(\xi_i)^{\beta-1}}{1+r_i(\xi_i)}.
\]
Hence
\begin{align*}
    L(\beta,i)u_i(\xi_i)^{-2/(Q-2)}r_i(\xi_i)^{\beta-2}&\leq C\Big(\frac{u_i(\xi_i)^{-2/(Q-2)}}{r_i(\xi_i)}+u_i(\xi_i)^{-2/(Q-2)}\Big)(|\nabla_\H R_i (\xi_i)|+\bar{r}_i(\xi_i)|T R_i (\xi_i)|)\\
    &\leq C\Big(\frac{1}{M}+u_i(\xi_i)^{-2/(Q-2)}\Big)(|\nabla_\H R_i (\xi_i)|+\bar{r}_i(\xi_i)|T R_i (\xi_i)|).
\end{align*}
Substituting the above estimate into \eqref{3.6 3}, we obtain
\begin{align*}
 &\,|\nabla_\H R_i(\xi_i)|+\bar r_i(\xi_i)|TR_i(\xi_i)|\\
 \le&\, C u_i(\xi_i)^{-2}+C\tau_i
   +C(M^{-1}+u_i(\xi_i)^{-2/(Q-2)})
     (|\nabla_\H R_i(\xi_i)|+\bar r_i(\xi_i)|TR_i(\xi_i)|).
\end{align*}
Fixing $M$ large enough and then taking $i$ sufficiently large, we have
\[
    |\nabla_\H R_i (\xi_i)|+\bar{r}_i(\xi_i)|T R_i (\xi_i)|\leq C u_i(\xi_i)^{-2} +C\tau_i.
\]
Since $r_i(\xi_i)> Mu_i(\xi_i)^{-2/(Q-2)}$, for large $i$, we have $\bar{r}_i(\xi_i)\geq u_i(\xi_i)^{-2/(Q-2)}$. Hence
\[
    |\nabla_\H R_i (\xi_i)|+u_i(\xi_i)^{-2/(Q-2)}|T R_i (\xi_i)|\leq |\nabla_\H R_i (\xi_i)|+\bar{r}_i(\xi_i)|T R_i (\xi_i)|\leq C u_i(\xi_i)^{-2} +C\tau_i.
\]
The above estimate, \eqref{3.6 5} and Lemma \ref{lem:4.12} immediately imply Lemma \ref{lem:4.13}.
\end{proof}

\begin{lem}\label{lem:4.14}
Under the assumptions of Lemma \ref{lem:4.12}, we have
	\[
	\tau_i \leq C u_i(\xi_i)^{-2}+CL(\beta, i) u_i(\xi_i)^{- 2\beta/(Q-2)},
	\]
where the  constant has the same dependencies as in Lemma~\ref{lem:4.12}.
\end{lem}
\begin{proof}
	Substituting the gradient estimate of Lemma~\ref{lem:4.13} into Lemma~\ref{lem:4.12} gives the bound for $\tau_i$.
\end{proof}
\begin{cor}\label{cor:4.15}
 In addition to the assumptions of Lemma \ref{lem:4.12}, we further assume that one of the following two conditions holds:
 \begin{itemize}
 	\item[(i)]$\beta=Q-2$ and $L(\beta, i)=o(1)$,
  	\item[(ii)]	$\beta>Q-2$ and $L(\beta, i)=O(1)$.
 \end{itemize}
Then for any $\si\in (0,1)$ we have
\begin{equation}\label{eq:normalized-pohozaev-vanishing}
\lim _{i \to\infty} u_i(\xi_i)^2 \int_{B_\si(\xi_i)}\Xi_i(R_i H_i^{\tau_i}) u_i^{p_i+1}=0 ,
\end{equation}
where $\Xi_i$ is defined in \eqref{Xi}.
\end{cor}
\begin{proof}
A direct calculation shows that
\begin{align*}
&\,\Big|\int_{B_\si (\xi_i)}\Xi_i(R_i H_i^{\tau_i}) u_i^{p_i+1}\Big| \\
 \leq&\,\Big|\int_{B_\si (\xi_i)}\Xi_i(R_i) H_i^{\tau_i} u_i^{p_i+1}\Big|+\tau_i\Big|\int_{B_\si (\xi_i)}\Xi_i(H_i) H_i^{\tau_i-1} R_i u_i^{p_i+1}\Big| \\
 \leq&\, C|\nabla_{\H} R_i(\xi_i)| \int_{B_\si (\xi_i)} d(\xi,\xi_i)  u_i^{p_i+1} +C \int_{B_\si (\xi_i)} d(\xi,\xi_i)|\nabla_{\H} R_i(\xi)-\nabla_{\H} R_i(\xi_i)| u_i^{p_i+1}\\
 &\,+C\int_{B_\si (\xi_i)}d(\xi,\xi_i)^2|TR_i(\xi)|  u_i^{p_i+1}+C\tau_i \int_{B_\si (\xi_i)} d(\xi,\xi_i)  u_i^{p_i+1}.
\end{align*}
Combining Proposition~\ref{prop:4.9}, Lemma~\ref{lem:4.11},
\eqref{I1 new}, \eqref{I2 new}, and Lemmas~\ref{lem:4.13}
and~\ref{lem:4.14} yields \eqref{eq:normalized-pohozaev-vanishing}.	
\end{proof}

Under the critical weighted flatness condition, every isolated blow-up point is simple.
\begin{prop}\label{prop:4.16}
 Suppose that $\{R_i\}$ satisfies \eqref{4.1} with constants $A_1,A_2$ and condition $(*)'_{Q-2}$ with $0<L_i\le L$ and a structural constant $C_*$ independent of $i$. Assume that $\{H_i\}$ satisfies \eqref{4.27} with constants $A_4,A_5$. Let $u_i \in  C^2(B_2)$ be a solution of \eqref{4.26}. If $\xi_i \to0$ is an isolated blow-up  point of $\{u_i\}$ with \eqref{4.5} for some positive constant $A_3$, then $\xi_i \to0$ is an isolated simple blow-up  point.
\end{prop}
\begin{proof}
 By Remark \ref{rem1}, $f_{u_i,\theta}$ has precisely one critical point in the interval $0<r<s_i$, where $s_i=S_i u_i(\xi_i)^{-(p_i-1) / 2}$ as before. Suppose $\xi_i \to0$ is not an isolated simple blow-up  point and let $\mu_i$ be the second critical point of $f_{u_i,\theta}$ defined as follows:
 $$\mu_i:=\min\Big\{r>s_i: \frac{\mathrm{d}}{\mathrm{d}r}f_{u_i,\theta}(r)=0  \text{ for some } \theta\in \pa B_1\Big\}.$$
  Then we see that
	\[
	\mu_i \geq s_i \quad \text{ and }\lim_{i\to \infty} \mu_i=  0.
	\]
		Without loss of generality, we assume that $\xi_i=0$. Set
	\[
	\phi_i(\xi)=\mu_i^{2 /(p_i-1)} u_i(\delta_{\mu_i} \xi), \quad \xi \in B_{2/\mu_i}.
\]
    Clearly, $\phi_i$ satisfies
	\begin{gather}
	 -\Delta_{\H} \phi_i(\xi)=\frac{n^2}{4}\widetilde{R}_i(\xi) \widetilde{H}_i^{\tau_i}(\xi) \phi_i^{p_i}(\xi),\label{3.2 in Li} \\
	 |\xi|^{2 /(p_i-1)} \phi_i(\xi) \leq A_3, \quad|\xi|<1 / \mu_i, \notag\\
	 \lim _{i \to\infty} \phi_i(0)=\infty, \notag
	\end{gather}
	where $\widetilde{R}_i(\xi)=R_i(\delta_{\mu_i} \xi)$ and $\widetilde{H}_i(\xi)=H_i(\delta_{\mu_i} \xi)$.
    Furthermore, note that for any $\theta \in \pa  B_1$,
\[
f_{\phi_i, \theta}(s)=s^{2 /(p_i-1)} \mu_i^{2 /(p_i-1)} u_i( \delta_{s \mu_i} \theta) ,
\]
	which shows that 0 is an isolated simple blow-up  point of $\phi_i$.


     On each fixed annulus $\{\var<|\xi|<R\}$, by Lemma \ref{lem:4.3}, Propositions \ref{prop:4.5} and \ref{prop:4.9} we have
$$
0<c(\var, R) \leq\phi_i(0)\phi_i(\xi) \leq C(\var, R).
$$
Since $\phi_i(0) \to  \infty$, the equation for $\phi_i(0)\phi_i(\xi)$,
$$
-\Delta_\H \phi_i(0)\phi_i(\xi)=\frac{n^2}{4} \widetilde{R}_i \widetilde{H}_i^{\tau_i}\phi_i(0) \phi_i^{p_i}(\xi),
$$
has a potential tending to zero locally uniformly away from the origin. Local sub-elliptic estimates and a diagonal subsequence therefore yield $\phi_i(0)\phi_i(\xi) \to  G>0$ in $C_{\text {loc }}^2(\Hn  \backslash\{0\})$, where $-\Delta_\H G=0$. By the B\^ocher theorem \cite{BL2007}, we have
$$
G=a_0|\xi|^{2-Q}+h, \quad a_0>0,
$$
with $h$ harmonic on $\Hn $. Since $G \geq 0, h \geq-a_0|\xi|^{2-Q}$, $h$ is bounded below on $\Hn $. The Harnack inequality in Lemma~\ref{lem:4.3}, whose constant depends only on $n$ when $a=0$, applied on expanding balls to $h-\inf_{\Hn}h+\var$ and then with $\var\to0$, implies that $h$ is constant. Thus $h=b$ for some constant $b\geq0$. By direct computations, we have
\[
0=\frac{\d}{\d s}[s^{(Q-2) / 2}(a_0 s^{2-Q}+b)]\Big|_{s=1}=\frac{Q-2}{2}(b-a_0),
\]
namely,
\be\label{ha}
    b=a_0>0.
\ee



	Applying the Pohozaev identity in Lemma \ref{prop:4.7} to \eqref{3.2 in Li} and using Proposition \ref{prop:4.9}, we obtain, for sufficiently small positive $\si $,
    \begin{align*}
        &\,\int_{\pa B_\si }B(\si ,\xi,\phi_i,\nabla \phi_i)\\
         \geq &\,\frac{n^2}{4}\frac{1}{p_i +1}\int_{B_\si } \Xi (\widetilde{R}_i\widetilde{H}_i^{\tau_i}) \phi_i^{p_i+1}-\frac{n^2}{4}\frac{1}{p_i +1}\int_{\pa  B_\si } \widetilde{R}_i\widetilde{H}_i^{\tau_i} \phi_i^{p_i+1} \Xi  \cdot \nu\\
       \geq &\, \frac{n^2}{4}\frac{1}{p_i +1}\int_{B_\si } \Xi (\widetilde{R}_i\widetilde{H}_i^{\tau_i}) \phi_i^{p_i+1}-O(\phi_i(0)^{-p_i-1}).
    \end{align*}
Multiplying the above by $\phi_i(0)^2$ and sending $i$ to $\infty$, we have
\begin{align}
    \int_{\pa B_\si }B(\si ,\xi,G,\nabla G)=&\,\lim_{i\to \infty}\phi_i(0)^2\int_{\pa B_\si }B(\si ,\xi,\phi_i,\nabla \phi_i)\notag\\
    \geq &\,\lim_{i\to \infty}\phi_i(0)^2 \frac{n^2}{4}\frac{1}{p_i +1}\int_{B_\si } \Xi (\widetilde{R}_i\widetilde{H}_i^{\tau_i}) \phi_i^{p_i+1}=0.\label{3.42 in L}
\end{align}
The last equality in \eqref{3.42 in L} follows from \eqref{eq:normalized-pohozaev-vanishing}, since Lemma~\ref{lem:wf-scaling} implies that $\widetilde{R}_i$ satisfies $(*)_{Q-2}'$ with the same structural constant and with
$$
\widetilde{L}_i=\mu_i^{Q-2} L_i \to  0 .
$$

On the other hand, we use \eqref{ha} and apply Lemma \ref{lem:positivepoho} to obtain that
\[
    \int_{\pa B_\si }B(\si ,\xi,G,\nabla G)<0
\]
for $\si >0$ sufficiently small, which contradicts \eqref{3.42 in L}. Proposition \ref{prop:4.16} is thus established.
\end{proof}

\begin{prop}\label{prop:4.17}
	 Suppose the assumptions of Proposition \ref{prop:4.16} hold except for condition $(*)_{Q-2}'$ for $R_i$. Then $|\nabla_{\H} R_i(\xi_i)| +|TR_i(\xi_i)|\to0$ as $i \to\infty$.
	\end{prop}
\begin{proof}
We first prove that $|\nabla_{\H} R_i(\xi_i)|\to 0$ as $i \to\infty$.

Suppose that
\be\label{4.37}
|\nabla_{\H} R_i(\xi_i)| \to d>0 .
\ee
Constants in this contradiction argument may also depend on $d$.
Without loss of generality, we assume $\xi_i=0$. There are two cases.

\textbf{Case 1}. 0 is an isolated simple blow-up  point. 

In this case, we argue as in the proof of Lemma \ref{lem:4.13} and obtain
\[
\Big|\int_{B_1} \widetilde{\nabla} R_i H_i^{\tau_i} u_i^{p_i+1}\Big| \leq C u_i(0)^{-2}+C \tau_i .
\]
It follows from Lemmas~\ref{lem:intrinsic-coefficient-calculus},  \ref{lem:4.8},  and Proposition \ref{prop:4.4} that
\[
|\nabla_{\H} R_i(0)|=\frac{1}{2}|\widetilde{\nabla}R_i(0)| \leq C \int_{B_1}|\widetilde{\nabla} R_i(\xi)-\widetilde{\nabla} R_i(0)| H_i^{\tau_i} u_i^{p_i+1}+o(1)=o(1) .
\]

\textbf{Case 2}. 0 is not an isolated simple blow-up  point. 

In this case we argue as in the proof of Proposition \ref{prop:4.16}. The only difference is that we cannot derive \eqref{3.42 in L} from Corollary \ref{cor:4.15}, since condition $(*)_{Q-2}'$ for $R_i$ is not assumed. Instead, we will use \eqref{4.37} to show \eqref{3.42 in L}.

Let $\mu_i$, $\phi_i$, $\widetilde{R}_i$ and $\widetilde{H}_i$ be as in the proof of Proposition \ref{prop:4.16}. The computation at the end of the proof of Proposition \ref{prop:4.16} gives
\[
        \int_{\pa B_\si }B(\si ,\xi,\phi_i,\nabla \phi_i)
         \geq \frac{n^2}{4}\frac{1}{p_i +1}\int_{B_\si } \Xi (\widetilde{R}_i\widetilde{H}_i^{\tau_i}) \phi_i^{p_i+1}-O(\phi_i(0)^{-p_i-1}).
\]

Now we estimate the integral term $\int_{B_\si } \Xi (\widetilde{R}_i\widetilde{H}_i^{\tau_i}) \phi_i^{p_i+1}$. Arguing as in the proof of Lemma \ref{lem:4.8}, we obtain
\begin{align*}
\tau_i \leq&\, C \phi_i(0)^{-2}+C \int_{B_\si }|\xi||\nabla_{\H} \widetilde{R}_i(\xi)| \widetilde{H}_i^{\tau_i} \phi_i^{p_i+1}+C \int_{B_\si }|\xi|^2|T \widetilde{R}_i(\xi)| \widetilde{H}_i^{\tau_i} \phi_i^{p_i+1}\\
\leq &\,C \phi_i(0)^{-2}+C \mu_i \phi_i(0)^{-2/(Q-2)} .
\end{align*}
Arguing as in the proof of Lemma \ref{lem:4.13}, we obtain
$$
\Big|\int_{B_\si } \widetilde{\nabla} \widetilde{R}_i \widetilde{H}_i^{\tau_i} \phi_i^{p_i+1}\Big| \leq C \phi_i(0)^{-2}+C \tau_i .
$$
It follows from Proposition \ref{prop:4.4} that
\begin{align*}
|\widetilde{\nabla} \widetilde{R}_i(0)|  \leq &\, C \int_{B_\si }|\widetilde{\nabla} \widetilde{R}_i(\xi)-\widetilde{\nabla} \widetilde{R}_i(0)| \phi_i^{p_i+1}+C \phi_i(0)^{-2}+C \tau_i \\
 \leq &\,  C \mu_i \phi_i(0)^{-2/(Q-2)}+C \phi_i(0)^{-2}+C \tau_i .
\end{align*}
Since $|\widetilde{\nabla} \widetilde{R}_i(0)|=2 \mu_i|\nabla_{\H} R_i(0)| \geq(d / 2) \mu_i$, we have
\[
\mu_i \leq C \phi_i(0)^{-2}+C \tau_i .
\]
It follows that $\tau_i \leq C \phi_i(0)^{-2}$ and $\mu_i \leq C \phi_i(0)^{-2}$. Therefore,
\[
\Big|\int_{B_\si } \Xi (\widetilde{R}_i\widetilde{H}_i^{\tau_i}) \phi_i^{p_i+1}\Big| \leq C \phi_i(0)^{-2-2/(Q-2)},
\]
and \eqref{3.42 in L} follows immediately.

We now prove $|TR_i(\xi_i)|\to0$. Suppose instead that $|TR_i(\xi_i)|\to d>0$, and translate $\xi_i$ to $0$.
If $0$ is an isolated simple blow-up point, the proof is similar to that in  \textbf{Case 1} with $T$ in place of the horizontal fields. 
If 0 is not an isolated simple blow-up point, we still need to show \eqref{3.42 in L}. Arguing as in Lemma \ref{lem:4.13}, we obtain
\[
\Big|\int_{B_\si }
D(\widetilde R_i\widetilde H_i^{\tau_i})
\phi_i^{p_i+1}\Big|
\le C\phi_i(0)^{-2},
\quad
D\in\{\overline{X}_j,\overline{Y}_j, T:1\le j\le n\}.
\]
Since $\tau_i\to0$ and $|TR_i(0)|\to d>0$, we have
\[
|T(\widetilde R_i\widetilde H_i^{\tau_i})(0)|
\ge c\mu_i^2.
\]
Moreover, Lemma~\ref{lem:intrinsic-coefficient-calculus}, applied to $R_iH_i^{\tau_i}$, gives
\begin{align*}
|T(\widetilde R_i\widetilde H_i^{\tau_i})(\xi)
-T(\widetilde R_i\widetilde H_i^{\tau_i})(0)|
&\le C\mu_i^{2+\al }|\xi|^\al ,\\
|\widetilde\nabla(\widetilde R_i\widetilde H_i^{\tau_i})(\xi)
-\widetilde\nabla(\widetilde R_i\widetilde H_i^{\tau_i})(0)|
&\le C\mu_i^2|\xi|.
\end{align*}
It follows from Lemma \ref{lem:4.11} and the positive lower bound for $\int_{B_\si }\phi_i^{p_i+1}$ that
\[
\mu_i^2
\le C\phi_i(0)^{-2}
+C\mu_i^{2+\al }\phi_i(0)^{-2\al /(Q-2)},
\]
namely, $\mu_i^2\le C\phi_i(0)^{-2}$.
Similarly, we have
\[
|\widetilde\nabla
(\widetilde R_i\widetilde H_i^{\tau_i})(0)|
\le C\phi_i(0)^{-2}
+C\mu_i^2\phi_i(0)^{-2/(Q-2)}
\le C\phi_i(0)^{-2}.
\]
Using Lemma~\ref{lem:intrinsic-coefficient-calculus}, we obtain
\begin{align*}
\Big|\int_{B_\si }
\Xi(\widetilde R_i\widetilde H_i^{\tau_i})
\phi_i^{p_i+1}\Big|
&\le
C|\widetilde\nabla
(\widetilde R_i\widetilde H_i^{\tau_i})(0)|
\phi_i(0)^{-2/(Q-2)}
+C\mu_i^2\phi_i(0)^{-4/(Q-2)}\\
&=o(\phi_i(0)^{-2}).
\end{align*}
Thus \eqref{3.42 in L} follows, yielding the same contradiction as before.
Hence $|TR_i(\xi_i)|\to0$.
	\end{proof}

\section{A priori estimates}\label{sec:4}
Fix $\al \in(0,1)$ and constants $A_1,A_2>0$. As in Section~\ref{sec:3}, $A_1$ will denote the two-sided bound for $R$ and $A_2$ its intrinsic norm bound, now on the sphere. Consider the subcritical equations
\be\label{maineq4}
-\Delta_{\theta_0}v +\frac{n^2}{4} v=\frac{n^2}{4} R  v^{p},\quad v>0 \quad \text { on }\, \S^{2n+1},
\ee
where  $p=\frac{Q+2}{Q-2}-\tau$ for $0\leq \tau<\frac{2}{Q-2}$ and $R$ satisfies
\[
A_1^{-1} \leq R \leq A_1 \quad \text { on }\, \S^{2n+1} ,
\]
and
\be\label{5.3}
\|R\|_{\mathcal A_\al (\S^{2n+1})} \leq A_2 .
\ee

Throughout this section we use the contact distance \eqref{eq:contact-distance} and the normalized charts \eqref{eq:normalized-cayley} with flatness understood as in Definition~\ref{def:weighted-flatness}.

\begin{prop}\label{prop:5.1}
 Let $v \in C^2(\S^{2n+1})$ be a positive solution of \eqref{maineq4}. For any $\var\in (0,1)$ and $S>1$, there exist large positive constants $C_1, C_2$ depending on $n,\al,A_1,A_2,\var$ and $S$ such that, if
\[
\max _{\S^{2n+1}} v \geq C_1,
\]
then $0\leq \tau<\var$, and there exists a finite set $\wp(v) \subset \S^{2n+1}$ such that:
\begin{itemize}
	\item[(i)] If $P \in \wp(v)$, then it is a local maximum point of $v$ and in the Cayley coordinate system $\xi=(x,y,t)$ with $P$ as the north pole,
	\[
	\|v^{-1}(P) v(F_P(\delta_{v^{-(p-1) / 2}(P)} \xi))-((1+k_P^2|z|^2)^2+k_P^4t^2)^{(2-Q)/ 4}\|_{C^2(B_{2 S})} \leq \var,
	\]
	where $k_P=\sqrt{R(P)} / 2$.
	\item[(ii)] If $P_1, P_2$ are distinct points of $\wp(v)$, then
\[
	\B(P_1,S v(P_1)^{-(p-1) /2}) \cap \B(P_2,S v(P_2)^{-(p-1) /2})=\emptyset .
\]
	\item[(iii)] $v(P) \leq C_2\{d_{\theta_0}(P, \wp(v))\}^{-2 /(p-1)}$ for all $P \in \S^{2n+1}$.
\end{itemize}
\end{prop}
\begin{proof}
Rescaling at a maximum point and applying the subcritical Liouville theorem \cite[Theorem~1.1]{MO2023} yields $\tau\to  0$ as $\max v\to  \infty$.
Using the classification results cited in the proof of Proposition~\ref{prop:4.4}, with the same distinction between $n=1$ and $n\ge2$, the proof of Proposition~\ref{prop:5.1} is similar to those of \cite[Proposition 4.1]{Li1995} and \cite[Lemma 3.1]{SZ1996}, and is omitted here. We refer to \cite{Li1995,SZ1996} for details.
\end{proof}

\begin{prop}\label{prop:5.2}
Assume the hypotheses of Proposition \ref{prop:5.1} hold. Suppose that there exists some constant $d>0$ such that $R$ satisfies $(*)_{Q-2}'$ for some $L>0$ and $C_*>0$ in $\Om_d=\{P \in \S^{2n+1}:|\nabla_{\theta_0} R(P)|<d\}$. Then, for $\var\in(0,1)$ and $S>1$, let $C_1$ and $\wp(v)$ be chosen as in Proposition~\ref{prop:5.1}. For any solution $v$ of \eqref{maineq4} with $\max _{\S^{2n+1}}v>C_1$, we have
\[
d_{\theta_0}(P_1, P_2) \geq \delta^*>0 \quad \text { for any }\, P_1, P_2 \in \wp(v),\, P_1 \neq P_2,
\]
where $\delta^*$ depends only on $n,\al,\var,S,C_*,A_1,A_2,L,d$.
\end{prop}
\begin{proof}
 Suppose the contrary: there exist sequences $\{p_i\}$ and $\{R_i\}$ satisfying the above assumptions, and a sequence of corresponding solutions $\{v_i\}$ such that
\be\label{5.5}
	\lim_{i \to  \infty}d_{\theta_0}(P_{1 i}, P_{2 i})=0, \quad P_{1 i}, P_{2 i} \in \wp(v_i), \quad d_{\theta_0}(P_{1 i}, P_{2 i})=\min _{ \substack{P_1, P_{2} \in \wp(v_i)\\P_1 \neq P_2}}d_{\theta_0}(P_{ 1}, P_{ 2}) .
\ee 
		Since $\B(P_{1i},S v_i(P_{1i})^{-(p_i-1) /2})$ and $\B(P_{2i},S v_i(P_{2i})^{-(p_i-1) /2})$ are disjoint, we infer from  \eqref{5.5} that $v_i(P_{1 i}) \to \infty$ and $v_i(P_{2 i}) \to \infty$. Proposition \ref{prop:5.1} also gives $\tau_i \to 0$.
	
	Use the normalized Cayley chart $F=F_{P_{1i}}$ defined in \eqref{eq:normalized-cayley}, and write $R_i(\xi)$ for $R_i(F(\xi))$. Expressions involving $\nabla_{\theta_0}R_i$ still refer to the prescribed curvature $R_i$ on the sphere. Then equation \eqref{maineq4} is transformed into
\[
	-\Delta_{\H} u_i(\xi)=\frac{n^2}{4} R_i(\xi) H^{\tau_i}(\xi) u_i(\xi)^{p_i}, \quad \xi \in \Hn ,
\]
	where
\be\label{5.8}
	u_i(\xi)=(2|J_F|)^{(Q-2)/(2Q)} v_i(F(\xi)), \quad H(\xi)=(2|J_F|)^{(Q-2)/(2Q)}.
\ee
Here $J_F$ denotes the surface Jacobian of the current chart $F=F_{P_{1i}}$, so $H=H_{P_{1i}}$. On each fixed coordinate ball, these factors satisfy \eqref{4.27} with constants independent of $i$. 
	 Denote the Cayley coordinates of $P_{2 i}$ again by $P_{2 i}$, and set $\vartheta_i=|P_{2 i}| \to  0$. For simplicity, we assume $P_{2 i}$ is a local maximum point of $u_i$: we can always reselect a sequence of points as in the proof of Proposition \ref{prop:5.1} to substitute for $P_{2 i}$.
	
	By Proposition \ref{prop:5.1}(ii), there exists some constant $C$ such that
\be\label{5.9}
	\vartheta_i>\frac{S}{C} \max \{ u_i(0)^{-(p_i-1) / 2}, u_i(P_{2 i})^{-(p_i-1) / 2}\} .
\ee
		Set
	\[
	w_i(\xi)=\vartheta_i^{2 /(p_i-1)} u_i(\delta_{\vartheta_i} \xi) \quad \text { in }\, \Hn  .
	\]
		It is easy to see that $w_i$, which is positive in $\Hn $, satisfies
\be\label{5.10}
	-\Delta_{\H} w_i(\xi)=\frac{n^2}{4}\widetilde{R}_i(\xi) \widetilde{H}^{\tau_i}(\xi) w_i(\xi)^{p_i} \quad \text { in }\, \Hn 
\ee
	and
	\[
	w_i \in C^2(\Hn ), \quad \liminf _{|\xi| \to  \infty} w_i(\xi)<\infty,
	\]
	where $\widetilde{R}_i(\xi)=R_i(\delta_{\vartheta_i} \xi)$ and $\widetilde{H}(\xi)=H(\delta_{\vartheta_i} \xi)$.

	By Proposition \ref{prop:5.1}, $u_i$ satisfies
\begin{align*}
    	u_i(\xi) \leq A_3|\xi|^{-2 /(p_i-1)}  &\quad \text { for all }\,|\xi| \leq \vartheta_i / 2, \\
	u_i(\xi) \leq A_3d(\xi,P_{2 i})^{-2 /(p_i-1)}  &\quad \text { for all }\,d(\xi,P_{2 i}) \leq \vartheta_i / 2.
\end{align*}
Here $A_3$ is obtained from the constant $C_2$ in Proposition~\ref{prop:5.1}(iii) and the fixed chart bounds. Thus it depends only on $n,\al,A_1,A_2,\var,S$; we use the same bound for the rescaled functions below.
	In view of \eqref{5.9}, we have
\[
 w_i(0)\ge c>0,\quad
 w_i(\delta_{\vartheta_i^{-1}}P_{2i})\ge c>0.
\]
Passing to a subsequence, assume $\delta_{\vartheta_i^{-1}}P_{2i}\to\bar P$ with $|\bar P|=1$. We claim that both $w_i(0)$ and $w_i(\delta_{\vartheta_i^{-1}}P_{2i})$ tend to infinity.

Suppose, to the contrary, that $w_i(0)\le C$ along a subsequence. By Proposition \ref{prop:5.1}, we know that $\sup_{B_r}w_i\le C$ for some fixed $r>0$. If $\{w_i\}$ is not locally bounded in $\Hn $, the minimality in \eqref{5.5} and Proposition \ref{prop:5.1} yield an isolated blow-up sequence $\zeta_i\to\zeta$, with $\zeta_i$ reselected as local maxima after the conformal change. Cases~1 and~2 below, applied at $\zeta_i$, show that $\zeta$ is an isolated simple blow-up point. Using Proposition \ref{prop:4.9} and the Harnack estimate \eqref{eq:annular-harnack}, we have
\[
 0<c\le w_i(0)\le Cw_i(\zeta_i)^{-1}\to 0,
\]
a contradiction. Hence $\{w_i\}$ is locally bounded. 
After passing to a further subsequence, we may assume that $R_i(0)\to R_*>0$. Since $\widetilde R_i\widetilde H^{\tau_i}\to R_*$ locally uniformly, the local Harnack inequality in Lemma~\ref{lem:4.3} and the regularity estimate \eqref{sch new} give
 \begin{gather*}
 w_i\to  w>0
 \quad\text{ in }\, C^2_{loc}(\Hn ),\\
 -\Delta_\H w=\frac{n^2}{4}R_*w^{(Q+2)/(Q-2)},\quad R_*>0,\\
 \nabla_\H w(0)=\nabla_\H w(\bar P)=0,\quad
 Tw(0)=Tw(\bar P)=0.
 \end{gather*}
The classification used in Proposition \ref{prop:4.4} implies that $w$ is a single bubble, which has only one full critical point, contradicting
$|\bar P|=1$. The same argument at
$\delta_{\vartheta_i^{-1}}P_{2i}$ proves
\[
 w_i(0)\to \infty,\quad
 w_i(\delta_{\vartheta_i^{-1}}P_{2i})\to \infty.
\]
Together with the preceding pointwise estimates, we obtain that $0$ and $\bar P$ are isolated blow-up points of $w_i$.

\textbf{Case 1}. $|\nabla_{\theta_0} R_i(P_{1i})| \leq d / 2$. 

In this case, $\{0\}$ is an isolated simple blow-up  point of $w_i$ by Lemma \ref{lem:wf-scaling} and Proposition \ref{prop:4.16}. 

\textbf{Case 2}. $|\nabla_{\theta_0} R_i(P_{1i})| \geq d / 2$.

In this case, we first prove that $\{0\}$ is an isolated simple blow-up  point of $w_i$.
By Remark \ref{rem1}, $f_{w_i,\theta}$ has precisely one critical point in the interval $0<r<s_i$, where $s_i=S_i w_i(0)^{-(p_i-1) / 2}$. Suppose $0$ is not an isolated simple blow-up  point of $w_i$ and let $\mu_i$ be the second critical point of $f_{w_i,\theta}$ defined as follows:
 $$\mu_i:=\min\Big\{r>s_i: \frac{\mathrm{d}}{\mathrm{d}r}f_{w_i,\theta}(r)=0  \text{ for some } \theta\in \pa B_1\Big\}.$$
  Then we see that
	\[
	\mu_i \geq s_i \quad \text{ and }\lim_{i\to \infty} \mu_i=  0.
	\]
		 Set
	\[
	\phi_i(\xi)=\mu_i^{2 /(p_i-1)} w_i(\delta_{\mu_i} \xi), \quad \xi \in B_{2/\mu_i}.
\]
    Clearly, $\phi_i$ satisfies
	\begin{gather}
	 -\Delta_{\H} \phi_i(\xi)=\frac{n^2}{4}\hat{R}_i(\xi) \hat{H}^{\tau_i}(\xi) \phi_i^{p_i}(\xi),\label{3.2new in Li} \\
	 |\xi|^{2 /(p_i-1)} \phi_i(\xi) \leq A_3, \quad|\xi|<1 / (2\mu_i), \notag\\
	 \lim _{i \to\infty} \phi_i(0)=\infty, \notag
	\end{gather}
	where $\hat{R}_i(\xi)=\widetilde{R}_i(\delta_{\mu_i}\xi)=R_i(\delta_{\vartheta_i\mu_i} \xi)$ and $\hat{H}(\xi)=\widetilde{H}(\delta_{\mu_i}\xi)=H(\delta_{\vartheta_i\mu_i} \xi)$.
    Furthermore, note that for any $\theta \in \pa  B_1$,
\[
f_{\phi_i, \theta}(s)=s^{2 /(p_i-1)} \mu_i^{2 /(p_i-1)} w_i( \delta_{s \mu_i} \theta) ,
\]
	which shows that 0 is an isolated simple blow-up  point of $\phi_i$.

     On each fixed annulus $\{\var<|\xi|<R\}$, by Lemma \ref{lem:4.3}, Propositions \ref{prop:4.5} and \ref{prop:4.9}, we have
$$
0<c(\var, R) \leq\phi_i(0)\phi_i(\xi) \leq C(\var, R).
$$
Since $\phi_i(0) \to  \infty$, the equation for $\phi_i(0)\phi_i(\xi)$,
$$
-\Delta_\H \phi_i(0)\phi_i(\xi)=\frac{n^2}{4} \hat{R}_i \hat{H}^{\tau_i}\phi_i(0) \phi_i^{p_i}(\xi),
$$
has a potential tending to zero locally uniformly away from the origin. Local sub-elliptic estimates and a diagonal subsequence therefore yield $\phi_i(0)\phi_i(\xi) \to  G>0$ in $C_{loc}^2(\Hn  \backslash\{0\})$, where $-\Delta_\H G=0$. By the B\^ocher theorem \cite{BL2007}, we have
$$
G=a_0|\xi|^{2-Q}+h, \quad a_0>0,
$$
with $h$ harmonic on $\Hn $. Since $G \geq 0, h \geq-a_0|\xi|^{2-Q}$, $h$ is bounded below on $\Hn $. The Harnack inequality in Lemma~\ref{lem:4.3}, whose constant depends only on $n$ when $a=0$, applied on expanding balls to $h-\inf_{\Hn}h+\var$ and then with $\var\to0$, implies that $h$ is constant. Thus $h=b$ for some constant $b\geq0$. By direct computations, we have
$$
0=\frac{\d}{\d s}[s^{(Q-2) / 2}(a_0 s^{2-Q}+b)]\Big|_{s=1}=\frac{Q-2}{2}(b-a_0),
$$
namely,
\be\label{ha new}
    b=a_0>0.
\ee

	Applying the Pohozaev identity in Lemma \ref{prop:4.7} to \eqref{3.2new in Li} and using Proposition \ref{prop:4.9}, we obtain, for sufficiently small positive $\si $,
    \begin{align*}
        &\,\int_{\pa B_\si }B(\si ,\xi,\phi_i,\nabla \phi_i)\\
         \geq &\,\frac{n^2}{4}\frac{1}{p_i +1}\int_{B_\si } \Xi (\hat{R}_i\hat{H}^{\tau_i}) \phi_i^{p_i+1}-\frac{n^2}{4}\frac{1}{p_i +1}\int_{\pa  B_\si } \hat{R}_i\hat{H}^{\tau_i} \phi_i^{p_i+1} \Xi  \cdot \nu\\
       \geq &\, \frac{n^2}{4}\frac{1}{p_i +1}\int_{B_\si } \Xi (\hat{R}_i\hat{H}^{\tau_i}) \phi_i^{p_i+1}-O(\phi_i(0)^{-p_i-1}).
    \end{align*}

Now we estimate the integral term $\int_{B_\si } \Xi (\hat{R}_i\hat{H}^{\tau_i}) \phi_i^{p_i+1}$. Arguing as in the proof of Lemma \ref{lem:4.8}, we obtain
\begin{align*}
\tau_i \leq&\, C \phi_i(0)^{-2}+C \int_{B_\si }|\xi||\nabla_{\H} \hat{R}_i(\xi)| \hat{H}^{\tau_i} \phi_i^{p_i+1}+C \int_{B_\si }|\xi|^2|T \hat{R}_i(\xi)| \hat{H}^{\tau_i} \phi_i^{p_i+1}\\
\leq &\,C \phi_i(0)^{-2}+C \vartheta_i\mu_i \phi_i(0)^{-2/(Q-2)} .
\end{align*}
Arguing as in the proof of Lemma \ref{lem:4.13}, we obtain
$$
\Big|\int_{B_\si } \widetilde{\nabla} \hat{R}_i \hat{H}^{\tau_i} \phi_i^{p_i+1}\Big| \leq C \phi_i(0)^{-2}+C \tau_i .
$$
It follows from Proposition \ref{prop:4.4}, Lemmas \ref{lem:intrinsic-coefficient-calculus} and \ref{lem:4.11} that
\begin{align*}
|\widetilde{\nabla} \hat{R}_i(0)|  \leq &\, C \int_{B_\si }|\widetilde{\nabla} \hat{R}_i(\xi)-\widetilde{\nabla} \hat{R}_i(0)| \phi_i^{p_i+1}+C \phi_i(0)^{-2}+C \tau_i \\
 \leq &\,  C \vartheta_i\mu_i \phi_i(0)^{-2/(Q-2)}+C \phi_i(0)^{-2}+C \tau_i .
\end{align*}
Since $|\widetilde{\nabla} \hat{R}_i(0)|=2 \vartheta_i\mu_i|\nabla_{\H} R_i(0)| \geq(d / 4) \vartheta_i\mu_i$, we have
\[
\vartheta_i\mu_i \leq C \phi_i(0)^{-2}+C \tau_i .
\]
It follows that $\tau_i \leq C \phi_i(0)^{-2}$ and $\vartheta_i\mu_i \leq C \phi_i(0)^{-2}$. Therefore,
\[
\Big|\int_{B_\si } \Xi (\hat{R}_i\hat{H}^{\tau_i}) \phi_i^{p_i+1}\Big| \leq C \phi_i(0)^{-2-2/(Q-2)},
\]
which contradicts \eqref{ha new} and Lemma \ref{lem:positivepoho}. This proves that $\{0\}$ is an isolated simple blow-up  point of $w_i$.
Similarly, $\{\bar{P}\}$ is also an isolated simple blow-up  point of $w_i$. 

By Proposition \ref{prop:4.9},
\[
w_i(0) w_i(\xi) \leq C_{\var} \quad \text { for all }\, \var \leq|\xi| \leq 1 / 2,
\]
where $C_{\var}$ is independent of $i$. By Proposition \ref{prop:5.1}, the Harnack estimate \eqref{eq:annular-harnack}, the interior estimates \cite[Theorem~6.1]{F1975} and the choice of $P_{1 i}, P_{2 i}$, there exists an at most countable set $\wp \subset \Hn $ such that
\[
\inf \{d(\xi,\widetilde{\xi}): \xi,\widetilde{\xi} \in \wp, \xi \neq \widetilde{\xi}\} \geq 1,
\]
and
\begin{align*}
\lim _{i \to  \infty} w_i(0) w_i(\xi)=G(\xi) &\quad  \text { in }\, C_{loc}^2(\Hn \backslash \wp), \\
G(\xi)>0 &\quad  \text { for }\, \xi \in \Hn \backslash \wp.
\end{align*}

Let $\wp_1 \subset \wp$ contain those points near which $G$ is singular. Lemma \ref{lem:4.3} and Propositions \ref{prop:4.5} and \ref{prop:4.9} imply that $0, \bar{P} \in \wp_1$. Since $p_i>1$, it follows from \eqref{5.10} that
\[
-\Delta_{\H} G=0 \quad\text{ in }\,  \Hn  \backslash \wp_1 .
\]
By Proposition \ref{prop:4.9} and the maximum principle, there exist positive constants $N_1, N_2$ and some nonnegative function $h$ satisfying
\[
-\Delta_{\H} h=0\quad\text{ in }\, \Hn  \backslash (\wp_1 \backslash\{0, \bar{P}\}),
\]
such that
\be\label{G}
G(\xi)=N_1|\xi|^{2-Q}+N_2|\bar{P}^{-1}\circ \xi|^{2-Q}+h(\xi), \quad \xi \in \Hn \backslash \wp_1 .
\ee

To verify that $h\ge0$, apply the minimum principle in a large ball $B_R$, with small balls around the remaining
poles removed. The positive singular terms give $h\ge0$ on sufficiently small inner boundaries.
On $\pa  B_R$, the positivity of $G$ gives
\[
h\ge -N_1|\xi|^{2-Q} -N_2|\bar P^{-1}\circ\xi|^{2-Q} \ge -CR^{2-Q}.
\]
Shrinking the inner boundaries and then letting $R\to\infty$, we obtain $h\ge0$. Consequently, the regular part of $G$ at the origin has value
\[
N_2|\bar P|^{2-Q}+h(0)>0.
\]

For $\si \in (0,1)$, we apply Lemma \ref{prop:4.7} to \eqref{5.10} and obtain
\begin{align}
    \int_{\pa B_\si }B(\si ,\xi,G,\nabla G)=&\,\lim_{i\to \infty}w_i(0)^2\int_{\pa B_\si }B(\si ,\xi,w_i,\nabla w_i)\notag\\
    \geq&\, \lim_{i\to \infty}w_i(0)^2 \frac{n^2}{4}\frac{1}{p_i +1}\int_{B_\si } \Xi (\widetilde{R}_i\widetilde{H}^{\tau_i}) w_i^{p_i+1}=0.\label{4.11 in L}
\end{align}
If $|\nabla_{\theta_0} R_i(P_{1i})| \leq d / 2$, the last equality in \eqref{4.11 in L} follows from \eqref{eq:normalized-pohozaev-vanishing} applied to \eqref{5.10} and the fact  that condition $(*)'_{Q-2}$ with $L=o(1)$ holds for $\widetilde{R}_i$. If $|\nabla_{\theta_0} R_i(P_{1i})| \geq d / 2$, we can apply the argument in the proof of Proposition \ref{prop:4.17} to conclude that $\vartheta_i, \tau_i\leq Cw_i(0)^{-2}$.

Using \eqref{G} and applying Lemma \ref{lem:positivepoho}, we obtain
\[
    \int_{\pa B_\si }B(\si ,\xi,G,\nabla G)<0,
\]
for $\si >0$ sufficiently small, which contradicts \eqref{4.11 in L}. Proposition \ref{prop:5.2} is established.		
\end{proof}

Now we consider
\be\label{maineq5}
-\Delta_{\theta_0}v_i +\frac{n^2}{4} v_i=\frac{n^2}{4} R_i  v_i^{p_i},\quad v_i>0 \quad \text { on }\, \S^{2n+1},
\ee
where $p_i=\frac{Q+2}{Q-2}-\tau_i$, $0\leq \tau_i<2/(Q-2)$ and $\tau_i \to 0$ as $i\to \infty$.
\begin{thm}\label{thm:5.3}
 Suppose every $R_i$ satisfies the assumptions on $R$ in Proposition~\ref{prop:5.2}, with the same constants $A_1,A_2,C_*,L,d$ for all $i$. Let $v_i$ be solutions of \eqref{maineq5}. Then
\[
	\|v_i\|_{S_1^2(\S^{2n+1})} \leq C,
	\]
	where $C>0$ depends only on $n,\al,C_*,A_1,A_2,L,d$. Furthermore, after passing to a subsequence, either $\{v_i\}$ stays bounded in $L^{\infty}(\S^{2n+1})$, or $\{v_i\}$ has only isolated simple blow-up  points and the distance between any two such points is bounded below by some positive constant depending only on $n,\al,C_*,A_1,A_2,L,d$.
\end{thm}
\begin{proof}
The theorem follows immediately from Propositions \ref{prop:5.2}, \ref{prop:4.17}, \ref{prop:4.16}, \ref{prop:4.4} and Lemma \ref{lem:4.11}.
\end{proof}


\begin{thm}\label{thm:5.4}
	 Suppose  $\{R_i\}$ satisfies the assumptions on $R$ in Proposition~\ref{prop:5.2}, with the same constants $A_1,A_2,C_*,L,d$ for all $i$. Let $v_i$ be solutions of \eqref{maineq5}. Suppose further that either $\{R_i\}$ satisfies condition $(*)_{Q-2}'$ for a sequence $L(Q-2, i)=o(1)$ in $\Om_{d, i}=\{q \in \S^{2n+1}:|\nabla_{\theta_0} R_i|<d\}$, or $\{R_i\}$ satisfies condition $(*)_\beta'$ with $\beta\in (Q-2,Q)$ and $L(\beta,i)=O(1)$ in $\Om_{d, i}$. Then, after passing to a subsequence, either $\{v_i\}$ stays bounded in $L^{\infty}(\S^{2n+1})$, or $\{v_i\}$ has precisely one isolated simple blow-up  point.
\end{thm}
\begin{proof}
	The strategy is the same as in the proof of Proposition \ref{prop:5.2}. We assume there are two isolated blow-up  points. Choose a fixed Cayley chart whose omitted point is not a blow-up point. Under the assumption on $\{R_i\}$, Corollary \ref{cor:4.15} applies to $u_i$ as defined in \eqref{5.8}. Moreover, since the distance between these blow-up  points is uniformly bounded below due to Proposition \ref{prop:5.2}, we can apply Lemma \ref{lem:positivepoho} to contradict \eqref{Pohozaev} and \eqref{eq:normalized-pohozaev-vanishing}.
\end{proof}

\begin{thm}\label{thm:5.5}
Suppose that $\{R_i\}\subset\mathcal A_\al (\S^{2n+1})$ satisfies \eqref{5.3} with a constant $A_2$ independent of $i$. Assume $R_i(P_0)\ge b_0>0$ for some $P_0\in\S^{2n+1}$ and a constant $b_0$ independent of $i$. For some $\var_0>0$ independent of $i$ and $\beta\in(2,Q)$, suppose that
\[
R_i(\xi)=R_i(0)+Q_i^{(\beta)}(\xi)+W_i(\xi), \quad|\xi| \leq \var_0,
\]
where $\xi=(x,y,t)$ denotes the normalized Cayley coordinates given by $F_{P_0}$ and $R_i(\xi)$ abbreviates $R_i(F_{P_0}(\xi))$. The function $Q_i^{(\beta)}(\xi)$ satisfies $Q_i^{(\beta)}(\delta_{\lam}\xi)=\lam^\beta Q_i^{(\beta)}(\xi)$ for all $\lam>0$ and $\xi \in \Hn $, and $W_i(\xi)$ satisfies
\[
|W_i(\xi)||\xi|^{-\beta}+|\widetilde{\nabla}W_i(\xi)||\xi|^{1-\beta}+|TW_i(\xi)||\xi|^{2-\beta} \to  0
\]
uniformly in $i$ as $\xi \to  0$. Suppose also that $Q_i^{(\beta)} \to  Q^{(\beta)}$ in $C^1_{loc}(\H^{n}\backslash \{0\})$ and for some positive constant $c_{\mathrm{nd}}$,
\[
c_{\mathrm{nd}}|\xi|^{\beta-1} \leq|\widetilde{\nabla} Q^{(\beta)}(\xi)|+|\xi||TQ^{(\beta)}(\xi)|, \quad|\xi| \leq \var_0,
\]
and
\[
\left(
\begin{array}{ll}
\displaystyle\int_{\Hn } \overline{\nabla} Q^{(\beta)}(\widetilde{\xi}\circ\xi)U_{0,1}^{Q^* } \, \d  \xi\\[1.75ex]
\displaystyle\int_{\Hn } Q^{(\beta)}(\widetilde{\xi}\circ\xi)U_{0,1}^{Q^* } \, \d  \xi
\end{array}
\right)
\neq 0, \quad \forall\, \widetilde{\xi} \in \Hn  .
\]
Let $v_i$ be positive solutions of \eqref{maineq} with $R=R_i$. If $P_0$ is an isolated simple blow-up  point of $v_i$, then $v_i$ must have at least one other blow-up  point.
\end{thm}
\begin{proof}
Suppose the contrary: $P_0$ is the only blow-up point. In the coordinates centered at $P_0$, let $u_i$ and $H$ be as in \eqref{5.8}, with $F=F_{P_0}$, and let $\xi_i\to0$ be local maximum points of $u_i$. Write
\[
m_i=u_i(\xi_i),\quad \ell_i=m_i^{-2/(Q-2)},\quad  \eta_i=\delta_{\ell_i^{-1}}\xi_i,\quad
\varphi_i(\zeta)=m_i^{-1}
u_i(\xi_i\circ\delta_{\ell_i}\zeta).
\]
First we establish that
\be\label{eq:center-shift}
|\xi_i|=O(\ell_i).
\ee

Since $v_i$ has no other blow-up point, \eqref{eq:sharp-isolated-bound} and the local Harnack inequality in Lemma~\ref{lem:4.3}, applied in fixed charts away from $P_0$, give
\be\label{ex est}
u_i(\xi)\le C(\var)m_i^{-1}H(\xi),
\quad |\xi|\ge\var>0.
\ee
Multiplying \eqref{maineq1}, with $R=R_i$, by $\overline X_ju_i$, $\overline Y_ju_i$ and $Tu_i$, respectively, and integrating by parts over $\Hn $, we obtain
\begin{equation}\label{eq:global-translation-identities}
\int_{\Hn }\widetilde\nabla R_i\,u_i^{Q^*}=0,
\quad \text{ and }\quad 
\int_{\Hn }TR_i\,u_i^{Q^*}=0.
\end{equation}
Combining \eqref{eq:global-translation-identities} with \eqref{ex est}, we obtain, for fixed small $\var>0$,
\[
\Big|\int_{B_\var}
(DR_i)(\xi_i\circ\xi)
u_i(\xi_i\circ\xi)^{Q^*}\,\d\xi\Big|
\le C(\var)m_i^{-Q^*}
=C(\var)\ell_i^Q,
\]
where $D\in\{\overline X_j,\overline Y_j,T:1\leq j\leq n\}$.

By Propositions~\ref{prop:4.4} and~\ref{prop:4.9}, together with
\eqref{ex est}, after passing to a subsequence we obtain
\begin{align}
\varphi_i &\to \Lam\circ\delta_k
\quad \text{ in } C^2_{loc}(\Hn), \notag\\
\varphi_i(\zeta) &\le C(1+|\zeta|)^{2-Q}
\quad \text{ for all } \,\zeta\in\Hn, \label{bound var}
\end{align}
where $k=\lim_{i\to\infty}\sqrt{R_i(\xi_i)}/2>0$
and $C$ is independent of $i$.

Suppose that \eqref{eq:center-shift} is false, namely, $r_i:=|\eta_i|\to\infty$ along a subsequence. Write $\mathcal D_i=(\widetilde\nabla,r_iT)$. Using the expansion of $R_i$ and changing variables $\xi=\delta_{\ell_i}\zeta$, we have
\begin{align}
&\,\Big|\int_{B_{\var/\ell_i}}
(\mathcal D_iQ_i^{(\beta)})(\eta_i\circ\zeta)\varphi_i^{Q^*}\,\d \zeta\Big|\notag\\
\le &\, C(\var)\ell_i^{Q-\beta+1}+\om (\var)
\int_{B_{\var/\ell_i}}( |\eta_i\circ\zeta|^{\beta-1}+r_i|\eta_i\circ\zeta|^{\beta-2})\varphi_i^{Q^*}\,\d \zeta,\label{D new 1}
\end{align}
where $\om (\var)\to0$ as $\var\to0$, uniformly in $i$. Here we also used $\ell_ir_i=|\xi_i|\to0$.

Choose $S_i\to\infty$ sufficiently slowly such that $S_i/r_i\to0$.
By homogeneity and the $C^1$ convergence of $Q_i^{(\beta)}$ away from the origin, we have
\be\label{D new 2}
\int_{B_{S_i}} (\mathcal D_iQ_i^{(\beta)})(\eta_i\circ\zeta)\varphi_i^{Q^*}\,\d \zeta = c_0\mathcal D_iQ_i^{(\beta)}(\eta_i) +o(r_i^{\beta-1}),
\ee
where
\[
c_0=\int_{\Hn }\Lam (\delta_k\zeta)^{Q^*}\,\d \zeta>0.
\]
Moreover,
\be\label{D new 3}
|\mathcal D_iQ_i^{(\beta)}(\eta_i)| \ge c r_i^{\beta-1}.
\ee

On the other hand, by the pointwise bound \eqref{bound var}, we have
\[
\int_{|\zeta|>S_i}
(
|\eta_i\circ\zeta|^{\beta-1} +r_i|\eta_i\circ\zeta|^{\beta-2} )\varphi_i^{Q^*}\,\d \zeta
\le
Cr_i^{\beta-1}(S_i^{-Q}+r_i^{-Q})
=o(r_i^{\beta-1}).
\]
By the same estimates, we have
$$
\int_{\Hn }(|\eta_i \circ \zeta|^{\beta-1}+r_i|\eta_i \circ \zeta|^{\beta-2}) \varphi_i^{Q^*} \, \d \zeta \leq C r_i^{\beta-1}.
$$
It follows from \eqref{D new 1}, \eqref{D new 2} and \eqref{D new 3} that
$$
c r_i^{\beta-1} \leq C(\var) \ell_i^{Q-\beta+1}+C \om (\var) r_i^{\beta-1}+o(r_i^{\beta-1}) .
$$
Thus, choosing $\var$ small and then $i$ large, we obtain
\[
r_i^{\beta-1}
\le C(\var)\ell_i^{Q-\beta+1}\to 0,
\]
a contradiction. This proves \eqref{eq:center-shift}.

Applying \eqref{Pohozaev} on $B_\rho(\xi_i)$ and letting $\rho\to\infty$, we obtain
\begin{equation}\label{eq:global-dilation-identity}
\int_{\Hn }
\Xi_\xi\big[R_i(\xi_i\circ\xi)\big] u_i(\xi_i\circ\xi)^{Q^*}\,\d \xi=0,
\end{equation}
where $\Xi_\xi=\sum_{j=1}^n(x_j\pa _{x_j}+y_j\pa _{y_j})+2t\pa _t$ acts on the variable $\xi=(x,y,t)$.
It follows from \eqref{ex est} that
\[
\Big|\int_{B_\var} \Xi_\xi[R_i(\xi_i\circ\xi)] u_i(\xi_i\circ\xi)^{Q^*}\,\d \xi\Big| \le C(\var)\ell_i^Q.
\]
Using the remainder assumptions, we have
\[
|\Xi_\xi[W_i(\xi_i\circ\xi)]| \le C\om (\var)( |\xi|\,|\xi_i\circ\xi|^{\beta-1}
+(|\xi|^2+|\xi_i||\xi|) |\xi_i\circ\xi|^{\beta-2}).
\]
By \eqref{eq:center-shift} and \eqref{bound var}, we have
\[
\Big|\int_{B_\var}
\Xi_\xi\big[Q_i^{(\beta)}(\xi_i\circ\xi)\big] u_i(\xi_i\circ\xi)^{Q^*}\,\d \xi\Big|
\le C(\var)\ell_i^Q +C\om (\var)\ell_i^\beta.
\]

Passing to a subsequence, let $\eta_i\to\eta$. Multiplying the last inequality by $\ell_i^{-\beta}$, changing variables $\xi=\delta_{\ell_i}\zeta$, and letting $i\to\infty$, we obtain
\[
\Big|\int_{\Hn }
\Xi_\zeta[Q^{(\beta)}(\eta\circ\zeta)] \Lam (\delta_k\zeta)^{Q^*}\,\d \zeta\Big| \le C\om (\var).
\]
Sending $\var\to0$ gives
\[
\int_{\Hn } \Xi_\zeta\big[Q^{(\beta)}(\eta\circ\zeta)\big] \Lam (\delta_k\zeta)^{Q^*}\,\d \zeta=0.
\]
Similarly, we also obtain
\[
\int_{\Hn }
\overline\nabla Q^{(\beta)}(\eta\circ\zeta) \Lam (\delta_k\zeta)^{Q^*}\,\d \zeta=0.
\]

Write $\eta=(a,b,c)$. By direct computations, we have
\[
\beta Q^{(\beta)}(\eta\circ\zeta)
=\Xi_\zeta\big[Q^{(\beta)}(\eta\circ\zeta)\big]+\sum_{j=1}^n (a_j\overline X_j+b_j\overline Y_j)
Q^{(\beta)}(\eta\circ\zeta) +2c\,TQ^{(\beta)}(\eta\circ\zeta).
\]
Multiplying the preceding identity by $\Lam(\delta_k\zeta)^{Q^*}$ and integrating, we obtain
\[
\int_{\Hn } Q^{(\beta)}(\eta\circ\zeta) \Lam (\delta_k\zeta)^{Q^*}\,\d \zeta=0.
\]
Using the change of variables $\xi=\delta_k\zeta$, we have
\[
\int_{\Hn } \overline\nabla Q^{(\beta)} ((\delta_k\eta)\circ\xi)U_{0,1}^{Q^*}\,\d \xi=0,
\quad \text{ and }\quad 
\int_{\Hn } Q^{(\beta)} ((\delta_k\eta)\circ\xi)U_{0,1}^{Q^*}\,\d \xi=0.
\]
The two vanishing integrals contradict the nonvanishing assumption of Theorem~\ref{thm:5.5}.
\end{proof}

\section{Proofs of the main theorems}\label{sec:5}
In this section, we prove the main theorems.
We begin with a compactness result for a fixed prescribed
curvature function. Theorem \ref{thm:homogeneous-compactness} further provides uniform a priori bounds along the homotopy $R_\mu=1+\mu(R-1)$, $\mu\in[\var,1]$,
which are needed for the degree arguments in our subsequent paper \cite{QTWPartII}.
\begin{cor}\label{thm:5.6}
Let $\al \in (0,1)$ and $R \in \mathcal{A}_\al (\S^{2 n+1})$ be positive. In every fixed contact chart assume that the weak derivatives $Z_a(T R)$ are locally essentially bounded away from the full critical set $\{\nabla_{\theta_0} R=T_{\theta_0} R=0\}$.  For each critical point $\zeta_0$ of $R$, let $F_{\zeta_0}$ be the centered Cayley chart in \eqref{eq:normalized-cayley}, with $F_{\zeta_0}(0)=\zeta_0$, and set $R_{\zeta_0}(\xi):=R(F_{\zeta_0}(\xi))$. Assume that there exist a small neighborhood $\mathscr{O}_{\zeta_0}$ of $0$ in $\H^n$,  constants $L>0$, and $\beta=\beta(\zeta_0)\in(Q-2,Q)$ such that  $R_{\zeta_0}$ satisfies condition $(*)'_\beta$ in $\mathscr{O}_{\zeta_0}$ for $L$, and
\[
    R_{\zeta_0}(\xi)=R(\zeta_0)+Q_{\zeta_0}^{(\beta)}(\xi)+W_{\zeta_0}(\xi)\quad \text{ in }\,\mathscr{O}_{\zeta_0},
\]
    where $Q_{\zeta_0}^{(\beta)}\in C^1_{loc }(\H^{n}\backslash \{0\})$ satisfies
    $Q_{\zeta_0}^{(\beta)}(\delta_{\lam}\xi)=\lam ^{\beta}Q_{\zeta_0}^{(\beta)}(\xi)$, $\forall\, \lam  >0$, $\xi\in \Hn $,
  \[
	c_{\mathrm{nd}}|\xi|^{\beta-1}\leq |\widetilde{\nabla} Q_{\zeta_0}^{(\beta)}(\xi)|+|\xi||TQ_{\zeta_0}^{(\beta)}(\xi)|
\]
for some constant $c_{\mathrm{nd}}>0$, which may depend on $\zeta_0$,
    and
\[
	\left( 
	\begin{array}{l}
	\displaystyle\int_{\Hn } \overline{\nabla} Q_{\zeta_0}^{(\beta)}(\eta\circ\xi)U_{0,1}^{Q^* } \, \d  \xi\\[2.5mm]
    \displaystyle\int_{\Hn } Q_{\zeta_0}^{(\beta)}(\eta\circ\xi)U_{0,1}^{Q^* } \, \d  \xi 
	\end{array}  \right) \neq 0, \quad \forall\,  \eta\in \Hn ,
	\] 
    and $W_{\zeta_0}(\xi)\in C^{1}(\mathscr{O}_{\zeta_0})$ satisfies $\lim _{\xi \to  0}\big(|W_{\zeta_0}(\xi)||\xi|^{-\beta}+|\widetilde{\nabla}W_{\zeta_0}(\xi)||\xi|^{1-\beta}+|TW_{\zeta_0}(\xi)||\xi|^{2-\beta}\big)=0$. Then there exists a positive constant $C \geq 1$ depending on $n,\al$ and $R$ such that for any solution $v$ of \eqref{maineq},
\[
1 / C \leq v \leq C \quad \text { on }\, \S^{2n+1} .
\]
Moreover, the set of positive solutions of \eqref{maineq} is compact in $C^2(\mathbb S^{2n+1})$.
\end{cor}
\begin{proof}
The upper bound follows from Proposition \ref{prop:wf-globalization}, Theorems \ref{thm:5.4} and \ref{thm:5.5}.
At a maximum point, equation \eqref{maineq} gives $\max_{\S^{2n+1}} v\ge(\sup R)^{-n/2}$. Since $R$ and $v$ are uniformly bounded, the local Harnack inequality in Lemma~\ref{lem:4.3}, applied on a fixed finite covering by Cayley charts, gives
\begin{equation}\label{eq:uniform-positive-lower-bound}
\min_{\S^{2n+1}} v\ge C^{-1}\max_{\S^{2n+1}} v \ge C^{-1}\Big(\sup_{\S^{2n+1}} R\Big)^{-n/2}>0.
\end{equation}
Compactness follows from \eqref{sch new} and the Arzel\`a--Ascoli theorem. By the uniform positive
lower bound, every $C^2$ limit remains positive, and the equation passes to the limit.
\end{proof}

\begin{proof}[Proof of Theorem~\ref{thm:homogeneous-compactness}]
By Proposition \ref{prop:wf-globalization}, the family $R_\mu=1+\mu(R-1), \mu \in[\var, 1]$, satisfies the hypotheses of Theorem \ref{thm:5.4} uniformly. Theorems~\ref{thm:5.4} and~\ref{thm:5.5} give the uniform upper bound in \eqref{0.9}; the argument yielding \eqref{eq:uniform-positive-lower-bound} gives the lower bound uniformly in $\mu$.
\end{proof}

\begin{proof}[Proof of Theorem~\ref{thm:intro-energy}]
Estimate \eqref{eq:intro-energy} is the critical case of Theorem~\ref{thm:5.3}.
\end{proof}

\begin{proof}[Proof of Theorem~\ref{thm:intro-onebubble}]
By Lemma \ref{lem:wf-radius}, the auxiliary functions $r_i$ are uniformly bounded in $\Om _{d,i}$. Choose 
$\gamma\in (Q-2,\min\{\beta,Q\})$. Hence Lemma \ref{lem:wf-lowering} implies that $\{R_i\}$ satisfies $(*)_{\gamma}'$ in $\Om _{d, i}$, with constant $L$ and a structural constant $\widetilde C_*$ independent of $i$. Since $Q-2<\gamma<Q$, Theorem \ref{thm:5.4} applies, so it remains to prove Theorem~\ref{thm:intro-onebubble} in the unbounded case.

    After passing to a subsequence, $\zeta_i \to  \bar{\zeta}$ is the only isolated simple blow-up  point of $v_i$.
    Let $\kappa_i:=R_i(\zeta_i)$, $\widehat{R}_i:=\kappa_i^{-1}R_i$ and $\widehat{v}_i:=\kappa_i^{n/2}v_i$. By direct computations, we have
    $$
    L_{\theta_0}\widehat{v}_i=\frac{n^2}{4}\widehat{R}_i\widehat{v}_i^\frac{Q+2}{Q-2}.
    $$
     For simplicity, assume that $\zeta_i$ is the north pole and $\widehat{R}_i(\zeta_i)=1$. Let $F: \Hn  \to  \S^{2n+1}\backslash\{\mathcal{S}\}$ be the inverse of the Cayley transform defined in \eqref{Cayley} with $\mathcal{S}$ being the south pole of $\S^{2n+1}$. As previously defined,  for any $\lam>0$, let $
		\delta_\lam(z,t)=(\lam z,\lam^2 t)$, $\forall\,(z,t)\in\Hn$.
	Set $\varphi_i=F \circ \delta_{\lam_i} \circ F^{-1}$ with $\lam_i=\widehat{v}_i(\zeta_i)^{-2/(Q-2)}$. Then $T_{\varphi_i} \widehat{v}_i=(\widehat{v}_i\circ\varphi_i)(|\operatorname{det}\d \varphi_i|)^{(Q-2)/(2Q)}$ satisfies
	\begin{equation}\label{eq:normalized-sphere-equation}
	-\Delta_{\theta_0}(T_{\varphi_i} \widehat{v}_i)+\frac{n^2}{4}(T_{\varphi_i} \widehat{v}_i)=\frac{n^2}{4} \widehat{R}_i \circ \varphi_i (T_{\varphi_i}\widehat{v}_i)^{\frac{Q+2}{Q-2}} \quad \text { on }\, \S^{2n+1} .
	\end{equation}
	For $\xi \in \Hn $, let
	\[
	u_i(\xi)=\frac{2^{n}}{((1+|z|^2)^2+t^2)^{n/2}} \widehat{v}_i \circ F(\xi)
	\]
	and
	\[
	\widetilde{u}_i(\xi)=\frac{2^{n}}{((1+|z|^2)^2+t^2)^{n/2}} T_{\varphi_i} \widehat{v}_i \circ F(\xi) .
	\]
    Note that
	\[
	|\operatorname{det} \d \varphi_i(F(\xi))|^{(Q-2)/(2Q)}=\lam_i^n \Big(\frac{(1+| z|^2)^2+t^2}{(1+|\lam_i z|^2)^2+\lam_i^4t^2} \Big)^{(Q-2)/4} .
	\]
	Hence $\widetilde{u}_i(\xi)=\lam_i^n u_i(\delta_{\lam_i}\xi)$ for any $\xi \in \Hn $ and $0<\widetilde{u}_i \leq 2^n$. Using the classification argument in the proof of Proposition \ref{prop:4.4}, we obtain
\[
	\widetilde{u}_i(\xi) \to \frac{2^{n}}{((1+|z|^2)^2+t^2)^{n/2}} \quad \text { in }\, C_{loc}^2(\Hn ) .
\]
	Therefore, $T_{\varphi_i}\widehat{v}_i \to  1$ in $C_{loc}^2(\S^{2n+1} \backslash\{\mathcal{S}\})$.

    By \eqref{eq:sharp-isolated-bound} and the local Harnack inequality in Lemma~\ref{lem:4.3} away from the unique blow-up point, the rescaled functions satisfy $\widetilde{u}_i(\xi) \leq C \min \{1,|\xi|^{-2 n}\} \leq C H_F(\xi)$. Hence $T_{\varphi_i} \widehat{v}_i$ is uniformly bounded on the whole sphere. Equation \eqref{eq:normalized-sphere-equation}, the interior estimates, and the Sobolev embedding \cite[Theorems~6.1 and~5.15]{F1975} yield uniform H\"older bounds, which, together with the convergence away from the south pole, imply $T_{\varphi_i} \widehat{v}_i \to  1$ uniformly on $\S^{2 n+1}$.
     This completes the proof.
\end{proof}

\medskip

\noindent\textbf{Data availability statement.}~	 Data sharing is not applicable to this article as no datasets were generated or analysed during the current study.

\medskip

\noindent\textbf{Declarations.}~ \textbf{Conflict of interest.} ~ On behalf of all authors, the corresponding author states that there is no conflict of interest.

\medskip


\end{document}